\documentclass[10pt,reqno]{amsart}
\usepackage{amsaddr}
\usepackage[a4paper, total={6.5in, 8.4in}]{geometry}

\usepackage{tikz}
\usepackage[symbol]{footmisc}

\usepackage{amsmath}
\usepackage{amssymb}
\usepackage{amsthm}
\usepackage[latin1]{inputenc}
\usepackage{eurosym}
\usepackage{graphicx}
\usepackage{epsfig}
\usepackage{hyperref}
\usepackage{dsfont}
\usepackage[space]{cite}
\usepackage{subcaption}
\usepackage{caption}
\usepackage{placeins}
\usepackage{mathtools}
\usepackage[ruled,vlined]{algorithm2e}
\usepackage[title]{appendix}

\usepackage[displaymath,mathlines]{lineno}
\modulolinenumbers[1]

\allowdisplaybreaks

\usepackage{hyperref}

\usepackage{ifthen}
\usepackage{comment}

\makeindex
\usepackage{color}
\usepackage{float}
\usepackage{commath}

\newcommand{\argmin}{\operatorname{argmin}}

\renewcommand{\div}{\operatorname{div}}

\def\leq{\leqslant}
\def\geq{\geqslant}

\numberwithin{equation}{section}

\newtheoremstyle{thmlemcorr}{10pt}{10pt}{\itshape}{}{\bfseries}{.}{10pt}{{\thmname{#1}\thmnumber{
#2}\thmnote{ (#3)}}}
\newtheoremstyle{thmlemcorr*}{10pt}{10pt}{\itshape}{}{\bfseries}{.}\newline{{\thmname{#1}\thmnumber{
#2}\thmnote{ (#3)}}}
\newtheoremstyle{defi}{10pt}{10pt}{\itshape}{}{\bfseries}{.}{10pt}{{\thmname{#1}\thmnumber{
#2}\thmnote{ (#3)}}}
\newtheoremstyle{remexample}{10pt}{10pt}{}{}{\bfseries}{.}{10pt}{{\thmname{#1}\thmnumber{
#2}\thmnote{ (#3)}}}
\newtheoremstyle{ass}{10pt}{10pt}{}{}{\bfseries}{.}{10pt}{{\thmname{#1}\thmnumber{
A#2}\thmnote{ (#3)}}}

\theoremstyle{thmlemcorr}
\newtheorem{theorem}{Theorem}
\numberwithin{theorem}{section}
\newtheorem{lemma}[theorem]{Lemma}
\newtheorem{corollary}[theorem]{Corollary}

\theoremstyle{thmlemcorr*}
\newtheorem{theorem*}{Theorem}
\newtheorem{lemma*}[theorem]{Lemma}
\newtheorem{corollary*}[theorem]{Corollary}
\newtheorem{proposition*}[theorem]{Proposition}
\newtheorem{problem*}[theorem]{Problem}
\newtheorem{conjecture*}[theorem]{Conjecture}

\theoremstyle{defi}

\newtheorem{hyp}{Assumption}

\newtheorem{problem}{Problem}

\theoremstyle{remexample}
\newtheorem{remark}[theorem]{Remark}

\newtheorem{pro}[theorem]{Proposition}

\theoremstyle{ass}

\newtheorem*{notations*}{Notations}

\allowdisplaybreaks

\title{Sparse Gaussian Processes for solving nonlinear PDEs}
\author{Rui Meng$^1$, Xianjin Yang$^{2,3,*}$}

\address[R. Meng]{
	$^1$Lawrence Berkeley National Laboratory, Berkeley, California, USA.}
\email{mengrui6351@gmail.com}

\thanks{$^*$Corresponding author.}
\address[X. Yang]{$^2$Yau Mathematical Sciences Center, Tsinghua  University, Haidian  District, Beijing, 100084, China.}
\address{$^3$Beijing Institute of Mathematical Sciences and Applications, Huairou District, Beijing, 101408, China.}
\email{yxjmath@gmail.com}


\thanks{
}
\date{\today}

\begin{document}
\maketitle

\begin{abstract}
{\color{black}This article proposes an efficient numerical method for solving nonlinear partial differential equations (PDEs) based on sparse Gaussian processes (SGPs). Gaussian processes (GPs) have been extensively studied for solving PDEs by formulating the problem of finding a reproducing kernel Hilbert space (RKHS) to approximate a PDE solution. The approximated solution lies in the span of base functions generated by evaluating derivatives of different orders of kernels at sample points. However, the RKHS specified by GPs can result in an expensive computational burden due to the cubic computation order of the matrix inverse. Therefore, we conjecture that a solution exists on a ``condensed" subspace that can achieve similar approximation performance, and we propose a SGP-based method to reformulate the optimization problem in the ``condensed" subspace. This significantly reduces the computation burden while retaining desirable accuracy. The paper rigorously formulates this problem and provides error analysis and numerical experiments to demonstrate the effectiveness of this method. The numerical experiments show that the SGP method uses fewer than half the uniform samples as inducing points and achieves comparable accuracy to the GP method using the same number of uniform samples, resulting in a significant reduction in computational cost.


Our contributions include formulating the nonlinear PDE problem as an optimization problem on a ``condensed" subspace of RKHS using SGP, as well as providing an existence proof and rigorous error analysis. Furthermore, our method can be viewed as an extension of the GP method to account for general positive semi-definite kernels.

}
\end{abstract}

\section{Introduction}
This work develops a numerical method based on sparse Gaussian processes (SGPs) \cite{liu2020gaussian, titsias2009variational} to solve nonlinear partial differential equations (PDEs). PDEs have been widely used to model applications in science, economics, and biology \cite{quarteroni2008numerical, thomas2013numerical}. However, very few PDEs admit explicit solutions. Standard numerical approaches including finite difference \cite{thomas2013numerical} and finite element methods \cite{hughes2012finite} for solving PDEs are prone to the curse of dimensions. Recently, to keep pace with increasing problem sizes, machine learning methods {\color{black}have attracted} a lot of attention \cite{raissi2019physics, zang2020weak, lu2021learning, chen2021solving, mou2022numerical}.  Probabilistic algorithms to solve linear problems are presented in \cite{owhadi2015bayesian, owhadi2017multigrid, cockayne2017probabilistic}. The authors of \cite{chen2021solving} extend these ideas and propose a Gaussian process (GP) framework for solving general nonlinear PDEs. In particular, for time-dependent PDEs, GP methods based on time-discretization are  considered in \cite{raissi2018numerical, kramer2021probabilistic, wang2021bayesian}. 
{\color{black}Most of the algorithms mentioned above seek approximations of a solution to a PDE in a normed vector space with base functions chosen beforehand. For instance, for the GP method in \cite{chen2021solving}, by the representer theorem \cite[Section 17.8]{owhadi2019operator}, a solution of a PDE is approximated in the span of base functions, each of which is associated with a sample. If the samples are not chosen properly, the base functions may contain redundant information about the approximated solution. Therefore, a set of random samples may result in consuming unnecessary extra storage and in a waste of computational time.  The computational cost of the GP method \cite{chen2021solving}} grows cubically with respect to the number of samples due to the inversion of covariance matrices {\color{black}of samples}. To improve the performance {\color{black}of} \cite{chen2021solving}, \cite{mou2022numerical} proposes to approximate the solution of a PDE in the space generated by  random Fourier features \cite{rahimi2007random, yu2016orthogonal}, where the corresponding covariance matrix is a low-rank approximation to that of the GP method in \cite{chen2021solving}. {\color{black}The numerical experiments in \cite{mou2022numerical} show that approximating solutions of PDEs in the spaces generated by low-rank kernels reduces the precomputation time, saves the memory, and achieves comparable accuracy to the GP method. This motivates us to question whether the space we use to approximate a solution of a PDE contains redundant information for general numerical methods of solving nonlinear PDEs and whether we can ``compress" the space until it reaches a minimum subspace but contains sufficient information to express a solution of the PDE. In other words, we are interested in the following problem:
\begin{problem}
\label{main_prob}
Let $\mathcal{U}$ be a normed vector space equipped with the norm $\|\cdot\|_{\mathcal{U}}$ and let $u^*\in \mathcal{U}$ be the solution of a nonlinear PDE. For any $\mathcal{W}\subset \mathcal{U}$, denote by $u^\dagger_{\mathcal{W}}$ a best approximation of $u^*$ in $\mathcal{W}$ in terms of a specific numerical method and by $|\mathcal{W}|$ the cardinality of $\mathcal{W}$. 
Let ${u^\dagger_{\widetilde{\mathcal{U}}}}$ be the approximation of $u^*$ in a subspace $\widetilde{\mathcal{U}}$ of $\mathcal{U}$, and $\|{u^\dagger_{\widetilde{\mathcal{U}}}}-u^*\|_{\mathcal{U}}\leq \epsilon$ for some $\epsilon>0$. Given a positive number $\delta$ such that $\epsilon\leq \delta$, find a subspace $\mathcal{U}_c$, with the minimum cardinality,  of $\widetilde{\mathcal{U}}$ such that the approximation $u^\dagger_{{\mathcal{U}}_c}$ of $u^*$ in $\mathcal{U}_c$ satisfies $\|u^* - u^\dagger_{{\mathcal{U}}_c}\|_{\mathcal{U}}\leq \delta$. That is, we find $\mathcal{U}_c$ such that
\begin{align*}
\mathcal{U}_c = \argmin_{\mathcal{V}\subset\widetilde{\mathcal{U}}, \|u^*-u^\dagger_{\mathcal{V}}\|_{\mathcal{U}}\leq \delta} |\mathcal{V}|. 
\end{align*}
\end{problem}
}

 {\color{black}
\begin{figure}
    \centering
    \begin{tikzpicture}[scale=1.5]
    \draw (-1.7,-1.2) rectangle (1.9, 1.2);
    \draw (0,0) ellipse (1.5 and 1);
    \node at (-1.2,0.4) {$\widetilde{\mathcal{U}}$};
    \draw (0.1, 0.1) ellipse (1.1 and 0.7);
    \node at (-0.6,0.3) {$\mathcal{U}_c$};
    \node at (-1.4,0.9) {$\mathcal{U}$};
    \filldraw (1.3,-0.6) circle (0.03) node[anchor=west] {$u^*$};
    \draw[opacity=0.3, thick, dashed, blue] (1.3, -0.6) circle (0.5);
    \draw[opacity=0.3, thick, dashed, red] (1.3, -0.6) circle (0.3);
    \coordinate (a) at (1.3, -0.6);
    \coordinate (b) at (1,-0.62);
    \coordinate (c) at (0.9,-0.3);
    \draw[red, opacity=0.3] (a) -- (b) node[midway, below, text opacity=1]{$\epsilon$};
    \draw[blue, opacity=0.3] (a) -- (c) node[midway, above, text opacity=1]{$\delta$};
    \filldraw (1,-0.62) circle (0.03) node[anchor=east] {${u^\dagger_{\widetilde{\mathcal{U}}}}$};
    \filldraw (0.9,-0.3) circle (0.03) node[anchor=south] {$u^\dagger_{\mathcal{U}_c}$};
\end{tikzpicture}
    \caption{\color{black}The illustration of Problem \ref{main_prob}. $u^*$ is the solution of a PDE in the space $\mathcal{U}$. Given an approximation $u^\dagger_{\widetilde{\mathcal{U}}}$ of $u^*$ in the subspace $\widetilde{\mathcal{U}}$ of $\mathcal{U}$ such that $\|{u^\dagger_{\widetilde{\mathcal{U}}}}-u^*\|_{\mathcal{U}}\leq \epsilon$ for some $\epsilon>0$, and  given a positive number $\delta$ such that $\epsilon\leq \delta$, we ``compress" $\widetilde{\mathcal{U}}$ to find a minimum subspace $\mathcal{U}_c$  of $\widetilde{\mathcal{U}}$ such that the approximation $u^\dagger_{{\mathcal{U}}_c}$ of $u^*$ in $\mathcal{U}_c$ satisfies $\|u^* - u^\dagger_{{\mathcal{U}}_c}\|_{\mathcal{U}}\leq \delta$. }
    \label{fig:compress_prob}
\end{figure}
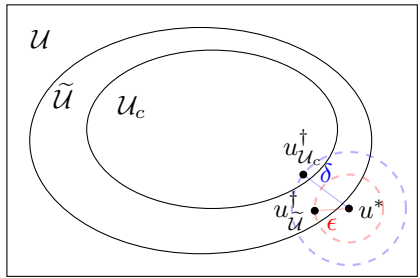
}
{\color{black}
Problem \ref{main_prob} is concerned about the efficient information in the space $\widetilde{\mathcal{U}}$ used to approximate the true solution $u^*$, which is illustrated in Figure \ref{fig:compress_prob}.   Using a ``condensed" space not only saves the storage but reduces the computational costs. Merely decreasing the number of base functions of $\mathcal{U}$ may not sufficiently handle Problem \ref{main_prob}, which is implied by the numerical experiments in Section \ref{secNumResults} (see discussions in Subsection \ref{subsecEliip}). A thorough general discussion of Problem \ref{main_prob} for all the numerical methods of solving nonlinear PDEs is out of the scope of this paper. Here,  we propose to study Problem \ref{main_prob} in terms of the  framework in \cite{chen2021solving},} where a numerical approximation of the solution to a nonlinear PDE is viewed as the maximum \textit{a posteriori} (MAP) estimator of a GP conditioned on solving the PDE at a finite set of sample points. {\color{black}By using the representer theorem \cite[Section 17.8]{owhadi2019operator}, the approximated solution lies in the span of base functions which are obtained by evaluating kernel functions or derivatives of kernels at sample points. Thus, base functions generated by random samples may provide redundant information about the approximated solution. To address Problem \ref{main_prob} with respect to the GP method in \cite{chen2021solving}, we propose a} low-rank approximation method based on SGPs, which leverages a few inducing points to summarize the information from observations. In {\color{black}the} SGP method, we replace the Gaussian prior by an alternative with a zero mean and a low-rank kernel induced from inducing points. The approximation to the solution of a PDE can be viewed as a MAP estimator of a SGP conditioned on a noisy observation about the values of linear operators of the true solution, where the observation satisfies the PDE at the samples.  Meanwhile, the numerical approximation can also be viewed as a function with the minimum norm in the reproducing kernel Hilbert space (RKHS) associated with the low-rank kernel. {\color{black}By Theorem \ref{thmexist} in Section \ref{secGF}, a solution given by the SGP method lies in the span of base functions generated by evaluating kernels or derivatives of kernels at inducing points. Thus, if a small number of inducing points are sufficient to capture the information in the sample points, the RKHS generated by the low-rank kernel is sufficient to serve as a candidate of the minimum subspace in Problem \ref{main_prob} for the GP method \cite{chen2021solving}}.  {\color{black}By} a stable Woodbury matrix identity stated in Subsection \ref{subsecComGram}, the dimension of the matrix to be inverted in the SGP method depends only on the number of inducing points selected. {\color{black}In the numerical experiments in Section \ref{secNumResults}, the SGP method reduces the computational cost by sacrificing only negligible amounts of accuracy by using half of the samples as inducing points.}

This inducing-point approach is initially proposed in the deterministic training conditional (DTC) approximation \cite{Seeger_2003}, and then it has been fully studied in several works. \cite{quinonero2005unifying} provides a unifying view for the approximation of GP regressions. This framework includes the DTC, the fully independent training conditional (FITC) approximation \cite{Snelson_2006}, and the partial independent training conditional (PITC) approximation. All {\color{black}the} above methods modify the joint prior using a conditional {\color{black}independence} assumption between training and test data given inducing variables. Alternatively,  \cite{titsias2009variational,Hensman_2013} leverage the inducing-point approach into the computing of the evidence lower bound of the log marginal likelihood. They retain exact priors but approximate the posteriors via variational inference. On the other hand, the estimation of inducing inputs has been generalized into an augmented feature space in \cite{lazaro2009inter}. In particular, SGPs are extended in the spectral domain with and without variational inference \cite{lazaro2010sparse,hensman2017variational}. Moreover, some works {\color{black} directly} speed up the computation of GP regressions through fast matrix-vector multiplication and pre-conditional conjugate gradients \cite{gardner2018gpytorch}, and through structured kernel interpolation \cite{wilson2015kernel}. 

Our contributions are summarized as follows:
\begin{enumerate}
	\item We introduce a general framework based on SGPs to solve nonlinear PDEs. {\color{black}Our method seeks approximated solutions in RKHSs associated with positive semi-definite kernels generated by inducing points. The SGP method provides a way of ``compressing" function spaces such that ``condensed" subspaces provide sufficient information about  approximated solutions to PDEs. Meanwhile, the framework in this paper can be viewed as an extension of the one in \cite{chen2021solving} to  take into consideration of more general positive semi-definite kernels};
	{\color{black} \item In the SGP method, the inversion of the covariance matrix requires only the inversion of a matrix whose size is proportional to the number of inducing points which is much less than the number of samples.} The numerical experiments in  Section \ref{secNumResults} show that the SGP method{\color{black},  which leverages less than half of the uniform samples as inducing points,}  reduces the computational cost {\color{black}by sacrificing only negligible amounts of accuracy} compared to the GP method in \cite{chen2021solving} using the same number of uniform samples;
	\item  We give a rigorous existence proof for the numerical approximation to the solution of a general PDE and analyze error bounds in Section \ref{secGF}. Our analysis bases on the arguments in \cite{chen2021solving},  learning theory \cite{smale2005shannon, smale2009geometry}, and Nystr\"{o}m approximations \cite{jin2013improved}. The error bound given in Theorem \ref{error_analy} provides qualitative hints to improve the accuracy of our method;
	\item The numerical experiments show that the choice of the positions of inducing points and parameters of kernels has a profound impact on the accuracy of the numerical approximations. In Section \ref{secSummary}, we give a probabilistic interpretation for the SGP method in the setting of \cite{titsias2009variational}, which motivates  future work for hyperparameter learning.
\end{enumerate}
The paper is organized as follows. Section \ref{secSummary}  summarizes our SGP method by solving a nonlinear elliptic PDE. In Section \ref{secGF}, we provide a framework of the SGP method to solve general PDEs. 
Numerical experiments are presented in Section \ref{secNumResults}. 
Further discussion and future work appear in Section \ref{secConclu}. 

\begin{notations*}
Vectors are in bold font. For a real-valued vector $\boldsymbol{v}$, we represent by $|\boldsymbol{v}|$ the Euclidean norm of $\boldsymbol{v}$ and by $\boldsymbol{v}^T$ its transpose. For a matrix $A$, we denote by $\|A\|$ the spectral norm of $A$. Let $\langle \boldsymbol{u}, \boldsymbol{v}\rangle$ or $\boldsymbol{u}^T\boldsymbol{v}$ be the inner product of vectors $\boldsymbol{u}$ and $\boldsymbol{v}$.  For a normed vector space $V$, let $\|\cdot\|_V$ be the norm of $V$. 
Let $\mathcal{U}$ be a Banach space associated with a quadratic norm $\|\cdot\|_{\mathcal{U}}$ and let $\mathcal{U}^*$ be the dual of $\mathcal{U}$. Denote by $[\cdot, \cdot]$ the duality pairing between $\mathcal{U}^*$ and $\mathcal{U}$. We assume that there exists a linear, bijective, symmetric ($[\mathcal{K}_{\mathcal{U}}\phi, \psi]=[\mathcal{K}_{\mathcal{U}}\psi, \phi]$), and positive ($[\mathcal{K}_{\mathcal{U}}\phi, \phi]>0$ for $\phi\not=0$) covariance operator $\mathcal{K}_{\mathcal{U}}: \mathcal{U}^*\mapsto \mathcal{U}$, such that 
\begin{align*}
	\|u\|_{\mathcal{U}}^2 = [\mathcal{K}^{-1}u, u], \forall u\in \mathcal{U}. 
\end{align*}
Let $\{\phi_i\}_{i=1}^P$ be a set of $P\in\mathbb{N}$ elements in $\mathcal{U}^*$ and let  $\boldsymbol{\phi}:=(\phi_1, \dots, \phi_P)$ be in the product space ${(\mathcal{U}^*)}^{\bigotimes P}$. Then, for $u\in \mathcal{U}$, we denote the pairing $[\boldsymbol{\phi}, u]$ by
\begin{align*}
	[\boldsymbol{\phi}, u]:=([{\phi}_1, u], \dots, [{\phi}_P, u]).
\end{align*}
Furthermore, for $\boldsymbol{u}:=(u_1,\dots,u_S)\in \mathcal{U}^{\bigotimes S}$, $S\in \mathbb{N}$, we represent by $[\boldsymbol{\phi}, \boldsymbol{u}]\in\mathbb{R}^{P\times S}$ the matrix with entries $[\phi_i, u_j]$. We write $\|\cdot\|$ as the induced norm for linear operators on $\mathcal{U}$, i.e., for any $\mathcal{F}: \mathcal{U}\mapsto\mathcal{U}$, we define $\|\mathcal{F}\|=\max_{\|u\|_{\mathcal{U}}=1}\|\mathcal{F}(u)\|_{\mathcal{U}}$. Given any $R\in \mathbb{N}$, we denote by $[R]$ the set of nonnegative integers less equal to $R$. 
Finally, we represent by $C$ a positive real number whose value  may change line by line. 
\end{notations*}
\section{A Recapitulation of the SGP Method}
\label{secSummary}

 {\color{black} To briefly summarize the SGP method and compare it to the GP method proposed in \cite{chen2021solving}, in Subsection \ref{subsecSGP}, we demonstrate the SGP method by solving a nonlinear elliptic PDE \eqref{illpde}.   Both the GP and the SGP methods can be applied to more general forms of PDEs, which is discussed in Section \ref{secGF}. Furthermore, we can naturally extend what we discuss below to solve PDE systems. We show a numerical example about solving a PDE system in Subsection \ref{subsecMFG}.} The SGP method approximates the solution of the PDE by a minimizer of an optimal recovery problem, which {\color{black}searches} a solution with the minimum norm in the RKHS associated with a low-rank kernel generated by inducing points and {\color{black}where} the values of linear operators of the solution are close to those satisfying the PDE at finite samples. From a probabilistic perspective,  the optimal recovery problem can also be interpreted as the MAP estimation for a SGP constrained by a noisy observation of {\color{black}linear operator values of the true solution. The SGP method also requires the inversion of a covariance matrix.} In Subsection \ref{subsecComGram}, we describe a robust way to compute the inverse of the covariance matrix, which is the cornerstone of the SGP method to reduce computational complexity.  Finally, in Subsection \ref{subsecProb}, we give a probabilistic explanation for the SGP method, which motivates a way to choose hyperparameters in future work.  

\subsection{The SGP Method}
\label{subsecSGP}
Let $\Omega$ be a bounded open subset of {\color{black}$\mathbb{R}^2$} with a Lipschitz boundary $\partial \Omega$. Given continuous functions $f:\Omega \mapsto \mathbb{R}$ and $g: \partial\Omega\mapsto\mathbb{R}$, we {\color{black}seek a function} $u^*$ solving 
\begin{align}
\label{illpde}
\begin{cases}
{\color{black}\Delta u^*(x) = u^*(x)(\partial_{x_1}u^*(x) + \partial_{x_2}u^*(x))+ f(x), \forall x:=(x_1, x_2)\in \Omega},\\
u^*(x) = g(x), \forall {\color{black}x:=(x_1, x_2)}\in \partial \Omega,
\end{cases}
\end{align}
{\color{black}where we assume that \eqref{illpde} admits a unique strong solution.} We propose to approximate the solution to \eqref{illpde} by functions in a RKHS generated by inducing points. First, we show the construction of such RKHS. We assume that the solution of \eqref{illpde} is in the space  $\mathcal{U}$, where $\mathcal{U}$ is {\color{black}a} RKHS associated with  a nondegenerate, symmetric, and positive definite kernel $K:\overline{\Omega}\times\overline{\Omega}\mapsto\mathbb{R}$. We view $\mathcal{U}$ as the ambient space. Let $\mathcal{U}^*$ be the dual space of $\mathcal{U}$ and let $[\cdot, \cdot]: \mathcal{U}^*\times\mathcal{U}\mapsto \mathbb{R}$ be the duality pairing. 
Next, we sample $M$ inducing points $\{\widehat{\boldsymbol{x}}_j\}_{j=1}^M$ in $\overline{\Omega}$ such that $\{\widehat{\boldsymbol{x}}_j\}_{j=1}^{M_\Omega}\subset\Omega$ and $\{\widehat{\boldsymbol{x}}_j\}_{j=M_{\Omega}+1}^M\subset\partial\Omega$ for $1\leq M_{\Omega}\leq M$. Then, we define the functionals $\phi_j^{(1)}=\delta_{\widehat{\boldsymbol{x}}_j}$ for $1\leq j\leq M$, {\color{black} $\phi_k^{(2)}=\delta_{\widehat{\boldsymbol{x}}_k}\circ (\partial_{x_1} + \partial_{x_2})$ and  $\phi_k^{(3)}=\delta_{\widehat{\boldsymbol{x}}_k}\circ\Delta$} for $1\leq k\leq M_\Omega$. Denote by $\boldsymbol{\phi}^{(1)}$,  $\boldsymbol{{\phi}}^{(2)}$, {\color{black}$\boldsymbol{{\phi}}^{(3)}$} the vectors with entries $\phi_j^{(1)}$ $\phi_k^{(2)}$, {\color{black}$\phi_k^{(3)}$}, respectively. Let $\boldsymbol{\phi}$ be the vector concatenating $\boldsymbol{\phi}^{(1)}$ $\boldsymbol{\phi}^{(2)}$, {\color{black}and $\boldsymbol{\phi}^{(3)}$}. For simplicity, we denote by $\phi_j$ the $j^{\textit{th}}$ component of $\boldsymbol{\phi}$. Define $K(\cdot, \boldsymbol{\phi})$ as the vector with entries $\int K(\cdot, \boldsymbol{x}')\phi_j(\boldsymbol{x}')\dif \boldsymbol{x}'$ and define $K(\boldsymbol{\phi}, \boldsymbol{\phi})$ as the matrix with elements $\int\int K(\boldsymbol{x}, \boldsymbol{x}')\phi_i(\boldsymbol{x})\phi_j(\boldsymbol{x}')\dif \boldsymbol{x}\dif \boldsymbol{x}'$. Suppose that the inducing points are chosen such that $K(\boldsymbol{\phi}, \boldsymbol{\phi})$ is {\color{black}positive definite}. Then, we define a new kernel $Q:\overline{\Omega}\times\overline{\Omega}\mapsto\mathbb{R}$ as
\begin{align*}
Q(\boldsymbol{x}, \boldsymbol{x}')=K(\boldsymbol{x}, \boldsymbol{\phi})(K(\boldsymbol{\phi}, \boldsymbol{\phi}))^{-1}K(\boldsymbol{\phi}, \boldsymbol{x}'). 
\end{align*}
Using the positivity of the kernel $K$, we see that $Q$ is {\color{black}positive semi-definite} (see Lemma \ref{lmKIcov} in Section \ref{secGF}). Thus, by the Moore--Aronszajn theorem, $Q$ induces a unique RKHS, denoted by $\mathcal{U}_{Q}$.  We call $\mathcal{U}_{{Q}}$ the RKHS generated by inducing points. 

Next, we approximate the solution $u$ in $\mathcal{U}_{{Q}}$. We take $N$ samples $\{\boldsymbol{x}_i\}_{i=1}^N$ such that $\{\boldsymbol{x}_i\}_{i=1}^{N_\Omega}\subset \Omega$ and $\{\boldsymbol{x}_i\}_{i=N_\Omega+1}^{N}\subset \partial \Omega$ for $1\leq N_\Omega\leq N$. Let $\|\cdot\|_{\mathcal{U}_{Q}}$ be the norm of $\mathcal{U}_{Q}$.  {\color{black}To find an approximated solution in $\mathcal{U}_{{Q}}$, for the first attempt, we consider
	a similar formulation to that of the GP method in \cite{chen2021solving} and solve the optimal recovery problem
	\begin{align}
		\label{nonregprob}
		\begin{cases}
			\min\limits_{u\in \mathcal{U}_{Q}} \|u\|_{\mathcal{U}_Q}^2 \\
			\text{s.t. } \Delta u(\boldsymbol{x}_j) =  u(\boldsymbol{x}_j)(\partial_{x_1}u(\boldsymbol{x}_j) + \partial_{x_2}u(\boldsymbol{x}_j)) + f(\boldsymbol{x}_j), \forall j=1,\dots, N_\Omega,\\
			\quad\quad u(\boldsymbol{x}_j) = g(\boldsymbol{x}_j), \forall j=N_{\Omega}+1,\dots,  N.
		\end{cases}
	\end{align}
	However, since $Q$ is only positive semi-definite, the representer theorem \cite[Section 17.8]{owhadi2019operator} does not apply. Furthermore, since $\mathcal{U}_Q$ is a subspace of $\mathcal{U}$, the solution $u^*$ to \eqref{illpde} may not lie in  $\mathcal{U}_Q$. Hence, we may not find a function in $\mathcal{U}_Q$ satisfying the constraints in \eqref{nonregprob}.     Thus, instead of seeking a function in $\mathcal{U}_{{Q}}$ satisfying the constraints of \eqref{nonregprob} exactly, we consider a relaxed version of \eqref{nonregprob}.
}

{\color{black}More precisely, our SGP method introduces a slack vector $\boldsymbol{z}$ and} approximates the solution $u^*$ of \eqref{illpde} with a minimizer of the following regularized optimal recovery problem
\begin{align}
\label{regoptprobs1}
\begin{cases}
\min\limits_{\boldsymbol{z}\in \mathbb{R}^{2N_{\Omega}+N}, u\in \mathcal{U}_{Q}} \sum\limits_{j=1}^{N}|z^{(1)}_j-u(\boldsymbol{x}_j)|^2{\color{black}+\sum\limits_{j=1}^{N_{\Omega}}|z^{({2})}_j-(\partial_{x_1}+\partial_{x_2}) u(\boldsymbol{x}_j)}|^2+\sum\limits_{j=1}^{N_{\Omega}}|z^{({\color{black}3})}_j-\Delta u(\boldsymbol{x}_j)|^2+\gamma\|u\|_{\mathcal{U}_Q}^2 \\
\text{s.t. } {\color{black}z_j^{(3)}=z_j^{(1)}z_j^{(2)}  +}  f(\boldsymbol{x}_j), \forall j=1,\dots, N_\Omega,\\
\quad\quad{z}^{(1)}_j = g(\boldsymbol{x}_j), \forall j=N_{\Omega}+1,\dots,  N,
\end{cases}
\end{align}
where $\gamma>0$ is a given small regularization parameter and
\begin{align*}
{\color{black}
	\boldsymbol{z}:=(z_1^{(1)},\dots, z^{(1)}_{N_\Omega}, z_{N_\Omega+1}^{(1)}, \dots, z_{N}^{(1)}, z^{(2)}_1,\dots, z^{(2)}_{N_\Omega}, z^{(3)}_1,\dots, z^{(3)}_{N_\Omega}).
}
\end{align*} Let $(u^\dagger, \boldsymbol{z}^\dagger)$ be a minimizer of \eqref{regoptprobs1}.  A minimizer $u^\dagger$ of \eqref{regoptprobs1} can be viewed as a MAP estimator of a SGP $\xi\sim\mathcal{N}(0, Q)$ conditioned on a noisy observation $\boldsymbol{z}^\dagger$ about values of the linear operators of $u^\dagger$, where $\boldsymbol{z}^\dagger$ satisfies the PDE at the sample
points. A probabilistic interpretation of \eqref{regoptprobs1} is given in Subsection \ref{subsecProb}.  

{\color{black}To solve the infinite dimensional minimization problem \eqref{regoptprobs1}, we derive a representer formula for a minimizer $u$.}
Define $\boldsymbol{\psi}$ as the vector consisting of {\color{black}$\delta_{{\boldsymbol{x}}_j}$, $\delta_{{\boldsymbol{x}}_k}\circ(\partial_{x_1}+\partial_{x_2})$, and  $\delta_{{\boldsymbol{x}}_k}\circ\Delta$,  for $1\leq j\leq N$ and $1\leq k\leq N_\Omega$.} We represent by $\psi_j$ the $j^{\textit{th}}$ component of $\boldsymbol{\psi}$. Let $Q(\cdot, \boldsymbol{\psi})$ be the vector containing entries $\int Q(\cdot, \boldsymbol{x}')\psi_j(\boldsymbol{x}')\dif \boldsymbol{x}'$ and let $Q(\boldsymbol{\psi}, \boldsymbol{\psi})$ be the matrix consisting of elements $\int\int Q(\boldsymbol{x}, \boldsymbol{x}')\psi_i(\boldsymbol{x})\psi_j(\boldsymbol{x}')\dif \boldsymbol{x}\dif \boldsymbol{x}'$. Theorem \ref{thmexist} in Section \ref{secGF} shows that \eqref{regoptprobs1} admits a minimizer $u^\dagger$ such that
\begin{align}
\label{udagform}
u^\dagger(x) = Q(x, \boldsymbol{\psi})^T(\gamma I + Q(\boldsymbol{\psi}, \boldsymbol{\psi}))^{-1}\boldsymbol{z}^\dagger,
\end{align}
where $I$ is the identity matrix and  $\boldsymbol{z}^\dagger$ is a minimizer of 
\begin{align}
\label{regoptzot}
\begin{cases}
\min\limits_{\boldsymbol{z}\in \mathbb{R}^{2N_{\Omega}+N}} \boldsymbol{z}^T(\gamma I + Q(\boldsymbol{\psi}, \boldsymbol{\psi}))^{-1}\boldsymbol{z}\\
\text{s.t. } {\color{black}{z}^{(3)}_j=z_{j}^{(1)}z_{j}^{(2)}+}f(\boldsymbol{x}_j), \forall j=1,\dots, N_\Omega,\\
\quad\quad{z}^{(1)}_j = g(\boldsymbol{x}_j), \forall j=N_{\Omega}+1,\dots,  N. 
\end{cases}
\end{align}
Next, we use the technique of eliminating variables in \cite{chen2021solving} to remove the constraints of \eqref{regoptzot}. More precisely, observing $z_j^{(1)}=g(\boldsymbol{x}_j)$ and {\color{black}$z_j^{(3)}=z_j^{(1)}z_j^{(2)}-f(\boldsymbol{x}_j)$},  we rewrite \eqref{regoptzot} as
\begin{align}
\label{regoptzotmi}
{\color{black}\min\limits_{(\boldsymbol{z}^{(1)}, \boldsymbol{z}^{(2)})\in \mathbb{R}^{2N_\Omega}}(\boldsymbol{z}^{(1)}, g(\boldsymbol{x}_{\partial\Omega}), 
 \boldsymbol{z}^{(2)}, \boldsymbol{z}^{(1)}\boldsymbol{z}^{(2)}-f(\boldsymbol{x}_{\Omega}))(\gamma I + Q(\boldsymbol{\psi}, \boldsymbol{\psi}))^{-1}\begin{pmatrix}
\boldsymbol{z}^{(1)}\\ g(\boldsymbol{x}_{\partial\Omega})\\ 
\boldsymbol{z}^{(2)}\\
\boldsymbol{z}^{(1)}\boldsymbol{z}^{(2)}-f(\boldsymbol{x}_\Omega)
\end{pmatrix},}
\end{align}
where {\color{black}$\boldsymbol{z}^{(1)}=(z_j^{(1)})_{j=1}^{N_\Omega}$, $\boldsymbol{z}^{(2)}=(z_j^{(2)})_{j=1}^{N_\Omega}$,  $g(\boldsymbol{x}_{\partial\Omega})=(g(\boldsymbol{x}_j))_{j=N_\Omega+1}^N$, $f(\boldsymbol{x}_{\Omega})=(f(\boldsymbol{x}_j))_{j=1}^{N_\Omega}$, and $\boldsymbol{z}^{(1)}\boldsymbol{z}^{(2)}=(z_j^{(1)}z_j^{(2)})_{j=1}^{N_\Omega}$.} Then, we obtain  $\boldsymbol{z}^\dagger$ by solving \eqref{regoptzotmi} using the Gauss--Newton method and get $u^\dagger$ by \eqref{udagform}.

\subsection{Computing $(\gamma I + Q(\boldsymbol{\psi}, \boldsymbol{\psi}))^{-1}$}
\label{subsecComGram} 
The bottleneck of solving  \eqref{regoptzotmi} is to calculate the inverse of $\gamma I + Q(\boldsymbol{\psi}, \boldsymbol{\psi})$, whose size is proportional to the number of sample points. The computational complexity of standard Cholesky decomposition algorithms  is $O(N^3)$ in general. {\color{black}Some fast algorithms \cite{schafer2021sparse, cao2023variational} are developed by considering sparse inverse Cholesky factorizations. However,  their works only consider the case where the covariance matrix contains pointwise evaluations  but not pointwise values of differential operators, which are required in the SGP method.} {\color{black}In a recent study by \cite{chen2023sparse}, a sparse Cholesky factorization algorithm was proposed for kernel matrices, leveraging the sparsity of the Cholesky factor through a novel ordering of Diracs and derivative measurements. Our approach, along with the algorithm described in \cite{chen2023sparse}, focuses on effectively solving nonlinear PDEs. Our method accomplishes this objective by approximating representation functions using a smaller set of bases. In contrast, their approach achieves the same goal by efficiently performing computations with dense kernel matrices through a novel sparse Cholesky decomposition technique.

In this paper, by using} the structure of $Q(\boldsymbol{\psi}, \boldsymbol{\psi})$, we show that the computational complexity can be reduced and depends only on the number of inducing points.  The technique is  from Section 3.4.3 of \cite{rasmussen2006gaussian} and  is well explained in the documents of GPflows\footnote{{\color{black}https://gpflow.github.io/GPflow/2.4.0/notebooks/theory/SGPR\_notes.html}}. We provide the full details below for the sake of completeness. 

Using the Woodbury identity, we get
\begin{align}
\label{wdbrid}
\begin{split}
    (\gamma I + Q(\boldsymbol{\psi}, \boldsymbol{\psi}))^{-1}=&(\gamma I + K(\boldsymbol{\psi}, \boldsymbol{\phi})(K(\boldsymbol{\phi}, \boldsymbol{\phi}))^{-1}K(\boldsymbol{\phi}, \boldsymbol{\psi}))^{-1}\\
=&\gamma^{-1}I-\gamma^{-2}K(\boldsymbol{\psi}, \boldsymbol{\phi})(K(\boldsymbol{\phi}, \boldsymbol{\phi})+\gamma^{-1}K(\boldsymbol{\phi}, \boldsymbol{\psi})K(\boldsymbol{\psi}, \boldsymbol{\phi}))^{-1}K(\boldsymbol{\phi}, \boldsymbol{\psi}). 
\end{split}
\end{align}
The equation \eqref{wdbrid} is not numerically stable when $\gamma$ is small.
To get a better conditioned matrix, we perform the Cholesky decomposition on $K(\boldsymbol{\phi}, \boldsymbol{\phi})$ and get $K(\boldsymbol{\phi}, \boldsymbol{\phi})=LL^T$ (see Remark \ref{etanug} below). Then, we consider
\begin{align}
\label{beform}
\begin{split}
    (\gamma I + Q(\boldsymbol{\psi}, \boldsymbol{\psi}))^{-1}=&\gamma^{-1}I-\gamma^{-2}K(\boldsymbol{\psi}, \boldsymbol{\phi})(K(\boldsymbol{\phi}, \boldsymbol{\phi})+\gamma^{-1}K(\boldsymbol{\phi}, \boldsymbol{\psi})K(\boldsymbol{\psi}, \boldsymbol{\phi}))^{-1}K(\boldsymbol{\phi}, \boldsymbol{\psi})\\
    =&\gamma^{-1}I-\gamma^{-2}K(\boldsymbol{\psi}, \boldsymbol{\phi})L^{-T}L^{T}(K(\boldsymbol{\phi}, \boldsymbol{\phi})+\gamma^{-1}K(\boldsymbol{\phi}, \boldsymbol{\psi})K(\boldsymbol{\psi}, \boldsymbol{\phi}))^{-1}LL^{-1}K(\boldsymbol{\phi}, \boldsymbol{\psi})\\
    =&\gamma^{-1}I-\gamma^{-2}K(\boldsymbol{\psi}, \boldsymbol{\phi})L^{-T}(L^{-1}(K(\boldsymbol{\phi}, \boldsymbol{\phi})+\gamma^{-1}K(\boldsymbol{\phi}, \boldsymbol{\psi})K(\boldsymbol{\psi}, \boldsymbol{\phi}))L^{-T})^{-1}L^{-1}K(\boldsymbol{\phi}, \boldsymbol{\psi})\\
    =&\gamma^{-1}I-\gamma^{-2}K(\boldsymbol{\psi}, \boldsymbol{\phi})L^{-T}(I+\gamma^{-1}L^{-1}K(\boldsymbol{\phi}, \boldsymbol{\psi})K(\boldsymbol{\psi}, \boldsymbol{\phi})L^{-T})^{-1}L^{-1}K(\boldsymbol{\phi}, \boldsymbol{\psi}).
\end{split}
\end{align}
Let $A:=\gamma^{-1/2}L^{-1}K(\boldsymbol{\phi}, \boldsymbol{\psi})$. Then, \eqref{beform} gives
\begin{align}
\label{rIQpsiphi}
 (\gamma I + Q(\boldsymbol{\psi}, \boldsymbol{\psi}))^{-1}=\gamma^{-1}I- \gamma^{-1}A^T(I+AA^T)^{-1}A.
\end{align}
The identity \eqref{rIQpsiphi} is the cornerstone for the SGP method to reduce the computational complexity. We note that the dimension of the matrix needed to be inverted at the right-hand side of \eqref{rIQpsiphi} depends only on the length of $\boldsymbol{\phi}$, which is proportional to the number of inducing points. Hence, if $M<<N$, using \eqref{rIQpsiphi}, we are able to save computational time and memory storage. The numerical results in Section \ref{secNumResults} show that a small number of inducing points is sufficient for the SGP approach to achieve comparable accuracy to the GP method in \cite{chen2021solving}. 
\begin{remark}
\label{etanug}
In general, $K(\boldsymbol{\phi}, \boldsymbol{\phi})$ is ill-conditioned. Thus, to deal with the inverse of $K(\boldsymbol{\phi}, \boldsymbol{\phi})$, we perform the Cholesky decomposition on $K(\boldsymbol{\phi}, \boldsymbol{\phi})+\eta \mathcal{R}$, where $\mathcal{R}$ is a block diagonal nugget built using the approach in \cite{chen2021solving} and $\eta>0$ is a small regularization parameter. 
\end{remark}

\subsection{A Probabilistic Perspective of the SGP Method}
\label{subsecProb}
Here, we give a probabilistic interpretation for the SGP method. First, we put a Gaussian prior on the solution $u^*$ to \eqref{illpde}, i.e., let $u^*\sim\mathcal{N}(0, {K})$, where ${K}$ is {\color{black} the} kernel associated with the RKHS $\mathcal{U}$ in Subsection \ref{subsecSGP}. Let $\{\boldsymbol{x}_i\}_{i=1}^N$ be the set of sample points in Subsection \ref{subsecSGP}. Define  
\begin{align*}
\boldsymbol{u}=(u^*(\boldsymbol{x}_1),\dots, u^*(\boldsymbol{x}_N), {\color{black}(\partial_{x_1}+\partial_{x_2})u^*(\boldsymbol{x}_1),\dots,(\partial_{x_1}+\partial_{x_2})u^*(\boldsymbol{x}_{N_\Omega}),} \Delta u^*(\boldsymbol{x}_1),\dots, \Delta u^*(\boldsymbol{x}_{N_\Omega})). 
\end{align*}
Then, $\boldsymbol{u}$ is a multivariate Gaussian random variable following the distribution $\mathcal{N}(0, K(\boldsymbol{\psi}, \boldsymbol{\psi}))$. Let $\boldsymbol{z}$ be a noisy observation of $\boldsymbol{u}$, i.e., 
\begin{align*}
\boldsymbol{z} = \boldsymbol{u} + \boldsymbol{\epsilon},
\end{align*}
where $\boldsymbol{\epsilon}$ is a noise following the distribution $\mathcal{N}(0, {\gamma}I)$, which is independent with $u^*$.  Thus, $\boldsymbol{z}\sim \mathcal{N}(0, K(\boldsymbol{\psi}, \boldsymbol{\psi})+\gamma I)$ and the probability density of $\boldsymbol{z}$ is 
\begin{align*}
p(\boldsymbol{z})=\frac{1}{\sqrt{(2\pi)^{2(N_\Omega+N)}\operatorname{det}(K(\boldsymbol{\psi}, \boldsymbol{\psi})+\gamma I)}}e^{-\frac{1}{2}\boldsymbol{z}^T(K(\boldsymbol{\psi}, \boldsymbol{\psi})+\gamma I)^{-1}\boldsymbol{z}}.
\end{align*}
Then, solving \eqref{illpde} by the GP method in \cite{chen2021solving} is equivalent to finding $\boldsymbol{z}$ with {\color{black}the maximum likelihood while satisfying the PDE system.} More precisely, the GP method seeks to find $\boldsymbol{z}$ solving  
\begin{align}
\label{illoptprob}
\begin{cases}
\max\limits_{\boldsymbol{z}} \ln p(\boldsymbol{z})\\
\text{s.t. } {\color{black}{z}^{(3)}_j=z_{j}^{(1)}z_{j}^{(2)}+}f(\boldsymbol{x}_j), \forall j=1,\dots, N_\Omega,\\
\quad\quad{z}^{(1)}_j = g(\boldsymbol{x}_j), \forall j, N_{\Omega} + 1 \leq j \leq N.
\end{cases}
\end{align}
In our SGP method, we replace the log-likelihood in \eqref{illoptprob} with its lower bound. The  derivation follows  \cite{titsias2009variational}.  More precisely, let $\{\widehat{\boldsymbol{x}}_i\}_{i=1}^M$ be the inducing points  in Subsection \ref{subsecSGP} and define  $\boldsymbol{v}^{(1)}=(u(\widehat{\boldsymbol{x}}_1), \dots, u(\widehat{\boldsymbol{x}}_{M}))$, {\color{black}$\boldsymbol{v}^{(2)}=((\partial_{x_1}+\partial_{x_2})u(\widehat{\boldsymbol{x}}_1), \dots, (\partial_{x_1}+\partial_{x_2})u(\widehat{\boldsymbol{x}}_{M_\Omega}))$, and} $\boldsymbol{v}^{( {\color{black}3} )}=(\Delta u(\widehat{\boldsymbol{x}}_1), \dots, \Delta u(\widehat{\boldsymbol{x}}_{M_\Omega}))$. Let $\boldsymbol{v}=(\boldsymbol{v}^{(1)}, {\color{black}\boldsymbol{v}^{(2)}}, \boldsymbol{v}^{( {\color{black}3} )})$. Then, $\boldsymbol{v}$ and $\boldsymbol{z}$ are joint Gaussian random variables. Let $q(\boldsymbol{v})$ be the variational distribution of $\boldsymbol{v}$ and let $q(\boldsymbol{u}, \boldsymbol{v})=p(\boldsymbol{u}|\boldsymbol{v})q(\boldsymbol{v})$ be the probability density function of the variational distribution of $(\boldsymbol{u}, \boldsymbol{v})$. Then, by Jensen's inequality, we have
\begin{align}
\label{illKL}
\begin{split}
\ln p(\boldsymbol{z})=& \ln \int\int p(\boldsymbol{z}, \boldsymbol{u}, \boldsymbol{v})\dif \boldsymbol{u}\dif \boldsymbol{v} = \ln \int\int  q(\boldsymbol{u},  \boldsymbol{v})  \frac{p(\boldsymbol{z}, \boldsymbol{u},  \boldsymbol{v})}{q( \boldsymbol{u},  \boldsymbol{v})}\dif \boldsymbol{u}\dif \boldsymbol{v} \\
\geq & \int\int q(\boldsymbol{u}, \boldsymbol{v})\ln \frac{p(\boldsymbol{z}, \boldsymbol{u}, \boldsymbol{v})}{q(\boldsymbol{u}, \boldsymbol{v})}\dif \boldsymbol{u}\dif \boldsymbol{v}=\int\int  p(\boldsymbol{u}| \boldsymbol{v})q(\boldsymbol{v})\ln \frac{p(\boldsymbol{z}|\boldsymbol{u})p(\boldsymbol{v})}{q(\boldsymbol{v})}\dif \boldsymbol{u}\dif \boldsymbol{v}\\
=:& \mathcal{F}(\boldsymbol{z}, q).
\end{split}
\end{align}
Thus, $\mathcal{F}(\boldsymbol{z}, q)$ provides a lower bound for $\ln p(\boldsymbol{z})$. To get an optimal lower bound, we take the derivative of $\mathcal{F}$ w.r.t. $q$ and get
\begin{align*}
\frac{\delta \mathcal{F}(\boldsymbol{z}, q(\boldsymbol{v}))}{\delta q(\boldsymbol{v})}=\int p(\boldsymbol{u}|\boldsymbol{v})(\ln \frac{p(\boldsymbol{z}|\boldsymbol{u})p(\boldsymbol{v})}{q(\boldsymbol{v})} -1)\dif \boldsymbol{u}. 
\end{align*}
Letting $\delta \mathcal{F}(q(\boldsymbol{v}))/ \delta q(\boldsymbol{v})=0$, we get the optimal variational distribution
\begin{align}
\label{illqv}
q^*(\boldsymbol{v})=\frac{p(\boldsymbol{v})}{Z}\operatorname{exp}\bigg(\int p(\boldsymbol{u}|\boldsymbol{v})\ln p(\boldsymbol{z}|\boldsymbol{u})\dif  \boldsymbol{u}\bigg),
\end{align}
where ${Z}$ is a real number guaranteeing the unit mass of $q^*$. 
Then, plugging \eqref{illqv} into \eqref{illKL}, we get
\begin{align*}
\mathcal{F}(\boldsymbol{z}, q)=-\frac{N+N_\Omega}{2}\ln 2\pi -\frac{1}{2}\ln\operatorname{det}(Q(\boldsymbol{\psi}, \boldsymbol{\psi})+\gamma{I})-\boldsymbol{z}^T(Q(\boldsymbol{\psi}, \boldsymbol{\psi})+\gamma{I})^{-1}\boldsymbol{z}-\frac{1}{2\gamma}\operatorname{Tr}(K(\boldsymbol{\psi}, \boldsymbol{\psi})-{Q}(\boldsymbol{\psi}, \boldsymbol{\psi})).
\end{align*}
Hence, if we replace $\ln p(\boldsymbol{z})$ by $\mathcal{F}(\boldsymbol{z}, q^*)$ in \eqref{illoptprob}, we note that \eqref{regoptzot} is equivalent to 
\begin{align}
\label{illsgpprob}
\begin{cases}
\max\limits_{\boldsymbol{z}} \mathcal{F}(z, q^*)\\
\text{s.t. } {\color{black}{z}^{(3)}_j=z_{j}^{(1)}z_{j}^{(2)}+}f(\boldsymbol{x}_j), \forall j=1,\dots, N_\Omega,\\
\quad\quad{z}^{(1)}_j = g(\boldsymbol{x}_j), \forall j=N_{\Omega}+1,\dots,  N. 
\end{cases}
\end{align}
\begin{remark}
\label{hyperpara}
The probabilistic interpretation of the SGP method above suggests a way for hyperparameter learning. Let $\boldsymbol{\theta}$ be the parameter vector consisting of the positions of the inducing points and the kernel's parameters. Write $q^*_{\boldsymbol{\theta}}$ to emphasize the dependence of $q^*$ in \eqref{illqv} on $\boldsymbol{\theta}$. Then, the SGP method with hyperparameter learning for solving \eqref{illpde} considers the maximization problem
\begin{align}
\label{illsgpprobhyp}
\begin{cases}
\max\limits_{\boldsymbol{z}, \boldsymbol{\theta}} \mathcal{F}(z, q^*_{\boldsymbol{\theta}})\\
\text{s.t. } {\color{black}{z}^{(3)}_j=z_{j}^{(1)}z_{j}^{(2)}+}f(\boldsymbol{x}_j), \forall j=1,\dots, N_\Omega,\\
\quad\quad{z}^{(1)}_j = g(\boldsymbol{x}_j), \forall j=N_{\Omega}+1,\dots,  N.
\end{cases}
\end{align}
However, the objective function in \eqref{illsgpprobhyp} is highly nonlinear. Thus, the computation of \eqref{illsgpprobhyp} is very costly. In the numerical experiments of Section \ref{secNumResults}, we fix the positions of inducing points. Then, we use the Gauss--Newton method to solve \eqref{illsgpprobhyp} for different values of the parameters of the kernel on the uniform grid in the domain of parameters and accept the result that achieves the largest value of the objective function in \eqref{illsgpprobhyp}. If we choose  Gaussian kernels for solving PDEs, which contain only one or two lengthscale parameters, the above grid search method for choosing hyperparameters is efficient. {\color{black}When the initial search interval is appropriately chosen, the combination of the maximum likelihood approach and grid search typically produces favorable outcomes. The optimal choice of hyperparameters depends on the underlying PDE's regularity, as a larger lengthscale tends to result in a smoother solution, while a smaller lengthscale tends to favor solutions with lower regularity. Given our understanding of the regularity of PDEs in our experiments, we search a lengthscale within the range of $0.01$ to $1$ using a step size of $0.01$.} We leave the full optimization over $\boldsymbol{z}$ and $\boldsymbol{\theta}$ in \eqref{illsgpprobhyp} to future work. 
\end{remark}
{\color{black}
\begin{remark}
The log-determinant term in $\mathcal{F}$ will not introduce extra computational costs if the Cholesky decomposition of $I+AA^T$ in \eqref{rIQpsiphi} is computed beforehand.  To see this, we first recall the Weinstein--Aronszajn identity, which states that for $n$-by-$m$ matrices $U$ and $V$,  and for an invertible $m$-by-$m$ matrix $W$, we have 
\begin{align*}
\operatorname{det}(I + UWV^T)=\operatorname{det}(W^{-1}+V^TU)\operatorname{det}(W),
\end{align*}
where $I$ is the identity matrix.  In the following presentation of this remark, for ease of presentation, we denote by $n$ and by $m$ the dimensions of $Q(\boldsymbol{\psi}, \boldsymbol{\psi})$ and $AA^T$, separately, where $A$ appears in  \eqref{rIQpsiphi}.   
Thus, using \eqref{rIQpsiphi} and the  Weinstein--Aronszajn identity, we have
\begin{align}
	\operatorname{det}\big((\gamma I + Q(\boldsymbol{\psi}, \boldsymbol{\psi}))^{-1}\big)=&\operatorname{det}\big(\gamma^{-1}I- \gamma^{-1}A^T(I+AA^T)^{-1}A\big)\nonumber\\
	=& \gamma^{-n}\operatorname{det}\big(-(I+AA^T)+AA^T\big)\operatorname{det}\big(-(I+AA^T)\big)\nonumber \\
	=& \gamma^{-n}\operatorname{det}\big(I+AA^T\big). \label{weinsteinar}
\end{align}
Hence, if the Cholesky decomposition of $I+AA^T$ is computed in advance, i.e. $I+AA^T=JJ^T$, we conclude from \eqref{weinsteinar} that
\begin{align}
\log \operatorname{det}\big(\gamma I + Q(\boldsymbol{\psi}, \boldsymbol{\psi})\big) =& - \log \operatorname{det}\big((\gamma I + Q(\boldsymbol{\psi}, \boldsymbol{\psi}))^{-1}\big)\nonumber\\
=& n \log \gamma - \log \operatorname{det}(JJ^T) = n\log \gamma - 2\log \operatorname{det}(J)\nonumber\\
=& n\log\gamma - 2\sum_{i=1}^m \log J_{ii}.  \label{logdetnocst}
\end{align}
Therefore, \eqref{logdetnocst} implies that the computational cost of the log-determinant in $\mathcal{F}$ is not high once we get the Cholesky decompositon of $I+AA^T$, which is required for solving  \eqref{regoptzot} in our current SGP method.
\end{remark}
}

\section{A General Framework for the SGP method}
\label{secGF}
In this section, we provide a framework of the SGP method for solving a general PDE. Then, we give the existence argument for a minimizer of the optimal recovery problem used to approximate the solution of the PDE. Finally, we  analyze the approximation errors. 

\subsection{A RKHS Generated by Linear Operators}
\label{subsecRKHSIP}
In this subsection, we show the construction of a RKHS generated by linear operators based on the abstract theory of RKHSs (see \cite{owhadi2019operator}). The RKHS forms a space where we seek approximations for solutions to general PDEs.

Let $\mathcal{U}$ be a Banach space, let $\mathcal{U}^*$ be its dual, and let $[\cdot, \cdot]$ be their duality paring. We assume further that there exists a covariance operator  $\mathcal{K}:\mathcal{U}^*\mapsto\mathcal{U}$ that is linear, bijective, symmetric ($[\mathcal{K}\phi, \varphi]=[\mathcal{K}\varphi, \phi]$), and positive ($[\mathcal{K}\phi, \phi]>0$ for $\phi\not=0$), and suppose that the norm of $\mathcal{U}$ is given by 
\begin{align*}
\|u\|_{\mathcal{U}}=\sqrt{[\mathcal{K}^{-1}u, u]}, \forall u \in \mathcal{U}. 
\end{align*}
Meanwhile, the inner products of $\mathcal{U}$ and $\mathcal{U}^*$ are given by 
\begin{align*}
\begin{split}
\langle u, v\rangle :=& [\mathcal{K}^{-1}u, v], \forall u, v\in \mathcal{U},\\
\langle \phi, \varphi\rangle_* :=& [\phi, \mathcal{K}\varphi], \forall \phi, \varphi\in \mathcal{U}^*.
\end{split}
\end{align*}
Then, $\mathcal{U}$ coincides with the RKHS space of the kernel $K$ defined by
\begin{align}
\label{defkernelK}
    K(\boldsymbol{x}, \boldsymbol{x'})=[\delta_{\boldsymbol{x}}, \mathcal{K}\delta_{\boldsymbol{x'}}], \forall  \boldsymbol{x}, \boldsymbol{x'}\in \overline{\Omega}, 
\end{align}
where $\delta_{\boldsymbol{x}}$ is the Dirac delta function centered at $\boldsymbol{x}$. 
Let $\xi\sim \mathcal{N}(0, \mathcal{K})$ be the \textit{canonical} GP on $\mathcal{U}$ \cite[Chapter 17.6]{owhadi2019operator} and let $[\phi, \xi]$ be the image of $\phi\in \mathcal{U}^*$ under $\xi$. Then, we have
\begin{align*}
    E[\phi, \xi]=0 \text{ and } E[\phi, \xi][\varphi, \xi] = [\phi, \mathcal{K}\varphi], \forall \phi, \varphi\in \mathcal{U}^*. 
\end{align*}
The GP method proposed in \cite{chen2021solving} approximates a solution of a nonlinear PDE by a MAP point for the GP $\xi\sim\mathcal{N}(0,\mathcal{K})$ conditioned on PDE constraints at the collocation points. To reduce the computation time, we consider a Gaussian prior with a different covariance operator generated by selected linear operators. 

Let $\boldsymbol{\phi}:=\{\phi_i\}_{i=1}^R$ be the collection of $R$ non-trivial elements of $\mathcal{U}^*$.  For the \textit{canonical} GP $\xi\sim\mathcal{N}(0, \mathcal{K})$, $[\boldsymbol{\phi}, \xi]$ is a $\mathbb{R}^R$-valued Gaussian vector and $[\boldsymbol{\phi}, \xi]\sim\mathcal{N}(0, \Theta)$, where
\begin{align}
\label{eq:defTt}
    \Theta\in \mathbb{R}^{R\times R}, \Theta_{i, n}=[\phi_i, \mathcal{K}\phi_n], \forall 1\leq i, n\leq R. 
\end{align}
Suppose that $\Theta$ is invertible. We define the induced covariance operator $\mathcal{Q}$ as
\begin{align}
\label{inducingK}
[\varphi, \mathcal{Q}\varphi] = \mathcal{K}(\varphi, \boldsymbol{\phi})\Theta^{-1}\mathcal{K}(\boldsymbol{\phi}, \varphi), \forall \varphi\in \mathcal{U}^*. 
\end{align}
The next lemma guarantees that $\mathcal{Q}$ is a covariance operator and associates with a RKHS. 
\begin{lemma}
\label{lmKIcov}
Let $\boldsymbol{\phi}:=\{\phi_i\}_{i=1}^R$ be the collection of $R$ linearly independent non-trivial elements of $\mathcal{U}^*$. Then, the operator $\mathcal{Q}$ defined in \eqref{inducingK} is linear, symmetric, and nonnegative. Hence, there exists a unique RKHS $\mathcal{U}_{{Q}}$ associated with the kernel $Q$ defined by
\begin{align}
\label{defQ}
Q(\boldsymbol{x}, \boldsymbol{x'})=[\delta_{\boldsymbol{x}}, \mathcal{Q}\delta_{\boldsymbol{x'}}], \forall \boldsymbol{x}, \boldsymbol{x'}\in \overline{\Omega}. 
\end{align}
Furthermore, $\mathcal{U}_{{Q}}\subset \mathcal{U}$ and 
\begin{align}
\label{eququ}
\|u\|_{\mathcal{U}_{{Q}}} = \|u\|_{\mathcal{U}}, \forall u\in \mathcal{U}_{\mathcal{Q}}. 
\end{align}
\end{lemma}
\begin{proof}
It is easy to see that $\mathcal{Q}$ is linear and symmetric. Since $\mathcal{K}$ is positive, $\Theta$ is positive definite. Then, $[\varphi, \mathcal{Q}\varphi] = \mathcal{K}(\varphi, \boldsymbol{\phi})\Theta^{-1}\mathcal{K}(\boldsymbol{\phi}, \varphi)\geq 0, \forall \varphi\in \mathcal{U}^*$. The equality holds if and only if $\varphi$ is perpendicular to $\boldsymbol{\phi}$ w.r.t. $\mathcal{K}$.  Thus, $\mathcal{Q}$ is nonnegative. Hence, the kernel $Q$ in \eqref{defQ} is symmetric and positive semi-definite. According to the Moore--Aronszajn theorem, there exists a unique RKHS $\mathcal{U}_{{Q}}$ associated with the kernel $Q$. Next, we show the relation between $\mathcal{U}_{{Q}}$ and $\mathcal{U}$. 

Since ${Q}(\boldsymbol{x}, \cdot)\in \mathcal{U}$ for any $\boldsymbol{x}\in \overline{\Omega}$,   the linear span of $\{{Q}(\boldsymbol{x},\cdot), \boldsymbol{x}\in \overline{\Omega}\}$, denoted by $\mathcal{U}_{\mathcal{Q}_0}$, is a subset of $\mathcal{U}$. By the Moore--Aronszajn theorem, $\mathcal{U}_{{Q}}$ is the completion of $\mathcal{U}_{{Q}_0}$. Thus, $\mathcal{U}_{{Q}}\subset\mathcal{U}$.  For any $u\in \mathcal{U}_{{Q}_0}$, there exist sequences $\{\boldsymbol{x}_{i}\}_{i=1}^m$ and $\{\alpha_i\}_{i=1}^m$, $m\in \mathbb{N}$, such that
\begin{align*}
\|u\|_{\mathcal{U}_{{Q}}}^2=\sum_{i=1}^m\sum_{j=1}^m\alpha_i\alpha_jQ(\boldsymbol{x}_i, \boldsymbol{x}_j)=\sum_{i=1}^m\sum_{j=1}^m\alpha_i\alpha_j\mathcal{K}(\delta_{\boldsymbol{x}_i}, \boldsymbol{\phi})\Theta^{-1}\mathcal{K}(\boldsymbol{\phi}, \delta_{\boldsymbol{x}_j})=\|u\|_{\mathcal{U}}^2, 
\end{align*}
where the last two equities follows by \eqref{inducingK} and \eqref{defQ}. Thus, we conclude \eqref{eququ}. 
\end{proof}

\subsection{The SGP method}
In this subsection, we present a general framework of the SGP method to solve nonlinear PDEs. 
Let $\Omega\subset\mathbb{R}^d$ be a bounded open domain with the boundary $\partial \Omega$ for $d\geq 1$. Let {\color{black}$L_1,\dots, L_{D_b}\in \mathcal{L}(\mathcal{U};C(\partial\Omega))$ and $L_{D_b+1},\dots, L_D\in \mathcal{L}(\mathcal{U};C(\Omega))$} be bounded linear operators for $1\leq D_b\leq D$. We seek to find a function $u^*$ solving the nonlinear PDE
\begin{align}
\label{pdegrf}
\begin{cases}
{P}(L_{D_b+1}(u^*)(\boldsymbol{x}),\dots, L_D(u^*)(\boldsymbol{x})) = f(\boldsymbol{x}), \forall \boldsymbol{x}\in \Omega,\\
{B}(L_1(u^*)(\boldsymbol{x}),\dots, L_{D_b}(u^*)(\boldsymbol{x})) = g(\boldsymbol{x}), \forall \boldsymbol{x}\in \partial \Omega,
\end{cases}
\end{align}
where $P, B$ represent nonlinear operators, and $f$, $g$ are given data. Throughout this section, we assume that \eqref{pdegrf} admits a unique strong solution in a quadratic Banach space $\mathcal{U}$ associated with the covariance operator $\mathcal{K}$, which means that $u^*$ has enough regularity for the linear operators to be well defined pointwisely in \eqref{pdegrf}.
    
We propose to approximate the solution to \eqref{pdegrf} in a RKHS generated by linear operators associated with inducing points. To do that, we take $M$ inducing points $\{\widehat{\boldsymbol{x}}_j\}_{j=1}^M$ in $\overline{\Omega}$ such that $\{\widehat{\boldsymbol{x}}_j\}_{j=1}^{M_\Omega}\subset\Omega$ and $\{\widehat{\boldsymbol{x}}_j\}_{j=M_{\Omega}+1}^M\subset\partial\Omega$. Then, we define the functionals $\phi_j^{(i)}\in \mathcal{U}^*$ as
\begin{align*}
\phi^{(i)}_j:=\delta_{\widehat{\boldsymbol{x}}_j}\circ L_i, \text{ and } \begin{cases}
M_{\Omega}+1\leq j\leq M, \text{ if } 1\leq i\leq D_b,\\
1\leq j\leq M_\Omega, \text{ if } D_{b}+1\leq i\leq D. 
\end{cases}
\end{align*}
Let $\boldsymbol{\phi}^{(i)}$ be the vector concatenating the functionals $\phi_j^{(i)}$ for fixed $i$ and define
\begin{align}
\label{defindiline}
\boldsymbol{\phi}:=(\boldsymbol{\phi}^{(1)}, \dots, \boldsymbol{\phi}^{(D)})\in (\mathcal{U}^*)^{\otimes R}, \text{ where } R=(M-M_\Omega)D_b + M_{\Omega}(D-D_b). 
\end{align}
The next corollary gives a RKHS generated by $\boldsymbol{\phi}$. 
\begin{corollary}
\label{exitrkhs}
Let $\boldsymbol{\phi}$ be as in \eqref{defindiline}. Assume that the elements of $\boldsymbol{\phi}$ are linearly independent in $\mathcal{U}^*$. Define {\color{black}$\Theta$ as in \eqref{eq:defTt}.}
Then, there exists 
a RKHS $\mathcal{U}_{{Q}}$ associated with the kernel $Q$ defined by
\begin{align*}
Q(\boldsymbol{x}, \boldsymbol{x'})=K(x,\boldsymbol{\phi})\Theta^{-1}K(\boldsymbol{\phi}, \boldsymbol{x'}), \forall \boldsymbol{x}, \boldsymbol{x'}\in \overline{\Omega},
\end{align*}
such that $\mathcal{U}_{Q}\subset\mathcal{U}$. Denote by $\|\cdot\|_{\mathcal{U}_{Q}}$ the norm of $\mathcal{U}_{Q}$. Then, 
\begin{align*}
\|u\|_{\mathcal{U}_{{Q}}} = \|u\|_{\mathcal{U}}, \forall u\in \mathcal{U}_{{Q}}. 
\end{align*}
\end{corollary}
\begin{proof}
The claim is a direct result of Lemma \ref{lmKIcov}. 
\end{proof}
{\color{black} Corollary \ref{exitrkhs} implies that  $\mathcal{U}_{Q}$ is a ``condensed" subspace of $\mathcal{U}$ and can be treated as a candidate solution for Problem \ref{main_prob}.}
Next, we propose to approximate the solution $u^*$ of \eqref{pdegrf} in the space $\mathcal{U}_{Q}$. More precisely, we take $N$ samples $\overline{\boldsymbol{x}}:=\{\boldsymbol{x}_i\}_{i=1}^N$ in $\overline{\Omega}$ such that $\{\boldsymbol{x}_i\}_{i=1}^{N_\Omega}\subset \Omega$ and $\{\boldsymbol{x}_i\}_{i=N_{\Omega}+1}^{N}\subset \partial\Omega$. 
Given a regularization parameter $\gamma>0$, we approximate $u^*$ by the minimizer of the following optimization problem
\begin{align}
\label{pdeminprob}
\begin{cases}
\min\limits_{\boldsymbol{z}\in \mathbb{R}^{\overline{R}}, u\in \mathcal{U}_Q}& \sum\limits_{i=1}^{D_b}\sum\limits_{j=N_\Omega+1}^N|L_i(u)(\boldsymbol{x}_j)-z_{j}^{(i)}|^2 + \sum\limits_{i=D_b+1}^{D}\sum\limits_{j=1}^{N_\Omega}|L_i(u)(\boldsymbol{x}_j)-z_{j}^{(i)}|^2 + \gamma\|u\|_{\mathcal{U}_Q}^2\\
\operatorname{s.t.}& {P}(z_j^{(D_b+1)},\dots, z_j^{(D)}) = f(\boldsymbol{x}_j), \text{ for } j = 1,\dots, N_{\Omega},\\
&{B}(z^{(1)}_j,\dots, z^{(D_b)}_j) = g(\boldsymbol{x}_j), \text{ for } j=N_{\Omega}+1,\dots, N,
\end{cases}
\end{align}
where 
\begin{align}
\label{defRbar}
\overline{R}=(N-N_\Omega)D_b + N_{\Omega}(D-D_b)
\end{align}
and 
\begin{align*}
\boldsymbol{z}:=(z^{(1)}_{N_\Omega+1},\dots,z^{(1)}_{N}, \dots, z^{(D_b)}_{N_\Omega+1},\dots,z^{(D_b)}_{N}, z^{(D_b+1)}_{1},\dots,z^{(D_b+1)}_{N_\Omega}, \dots, z^{(D)}_{1},\dots,z^{(D)}_{N_\Omega}).
\end{align*}
For ease of presentation, we rewrite \eqref{pdeminprob} into a compact form. To do that, we define the functionals $\psi_j^{(i)}\in \mathcal{U}^*$ as
\begin{align*}
\psi^{(i)}_j:=\delta_{\boldsymbol{x}_j}\circ L_i, \text{ and } \begin{cases}
N_{\Omega}+1\leq j\leq N, \text{ if } 1\leq i\leq D_b,\\
1\leq j\leq N_\Omega, \text{ if } D_{b}+1\leq i\leq D. 
\end{cases}
\end{align*}
Meanwhile, let $\boldsymbol{\psi}^{(i)}$ be the vector consisting of $\psi_j^{(i)}$ for fixed $i$ and define
\begin{align}
\label{defpsi}
\boldsymbol{\psi}:=(\boldsymbol{\psi}^{(1)}, \dots, \boldsymbol{\psi}^{(D)})\in (\mathcal{U}^*)^{\otimes \overline{R}}.
\end{align}
Next, we define the data vector $\boldsymbol{y}\in \mathbb{R}^N$ by
\begin{align*}
y_i = \begin{cases}
f(\boldsymbol{x}_i), \text{ if } i\in \{1,\dots, N_{\Omega}\},\\
g(\boldsymbol{x}_i), \text{ if } i\in \{N_{\Omega}+1, \dots, N\}. 
\end{cases}
\end{align*}
Furthermore, we define the nonlinear map
\begin{align*}
(F(\boldsymbol{z}))_j:=\begin{cases}
{P}(z_j^{(D_b+1)},\dots, z_j^{(D)}), \text{ for } j = 1,\dots, N_{\Omega},\\
{B}(z^{(1)}_j,\dots, z^{(D_b)}_j), \text{ for } j=N_{\Omega}+1,\dots, N. 
\end{cases}
\end{align*}
Thus, we reformulate \eqref{pdeminprob} as
\begin{align}
\label{pderegprob}
\begin{cases}
\min\limits_{\boldsymbol{z}\in \mathbb{R}^{\overline{R}}, u\in \mathcal{U}_Q}& |[\boldsymbol{\psi}, u]-\boldsymbol{z}|^2+\gamma\|u\|_{\mathcal{U}_Q}^2\\
\operatorname{s.t.}& F(\boldsymbol{z})=\boldsymbol{y}. 
\end{cases}
\end{align}

{\color{black}We observe that the minimization on $u$ in \eqref{pderegprob} is over the infinite dimensional function space $\mathcal{U}_Q$.  Fortunately, we can derive a representer formula for $u$ and transform \eqref{pderegprob} into a finite dimensional minimization problem. However, we are not able to use the representer theorem in \cite[Section 17.8]{owhadi2019operator} directly, since the inducing kernel $Q$ in Corollary \ref{exitrkhs} is not positive definite. Instead,} the arguments are based on the framework in \cite{smale2005shannon}. {\color{black}The basic idea is to define a sampling operator $S_{\overline{\boldsymbol{x}}}$ (explained in more details below) such that $S_{\overline{\boldsymbol{x}}}u = [\boldsymbol{\psi}, u], \forall u\in \mathcal{U}_Q$. Then, the objective function in \eqref{pderegprob} is reformulated as $\langle u, S_{\overline{\boldsymbol{x}}}^*S_{\overline{\boldsymbol{x}}}u - S_{\overline{\boldsymbol{x}}}^*\boldsymbol{z}\rangle+\gamma \|u\|_{\mathcal{U}_Q}^2$, where $S_{\overline{\boldsymbol{x}}}^*$ is the adjoint operator of $S_{\overline{\boldsymbol{x}}}$ and will be defined later. Then, by taking a functional derivative of the objective function and setting the derivative to be zero,  we obtain $u=(S_{\overline{\boldsymbol{x}}}^*S_{\overline{\boldsymbol{x}}}+\gamma I)^{-1}S_{\overline{\boldsymbol{x}}}^*\boldsymbol{z}$ for fixed $\boldsymbol{z}$. Then, we plug $u=(S_{\overline{\boldsymbol{x}}}^*S_{\overline{\boldsymbol{x}}}+\gamma I)^{-1}S_{\overline{\boldsymbol{x}}}^*\boldsymbol{z}$ back into \eqref{pderegprob} and get a finite dimensional minimization problem over $\boldsymbol{z}$. Next, we explain rigorously the definitions and the properties of $S_{\overline{\boldsymbol{x}}}$ and its adjoint in more details. 
}

Let $\ell^2(\overline{\boldsymbol{x}})$ be the set of sequences $\boldsymbol{a}=(a_{\boldsymbol{x}})_{\boldsymbol{x}\in \overline{\boldsymbol{x}}}$ with {\color{black}$\langle \boldsymbol{a}, \boldsymbol{b}\rangle_{\ell^2(\overline{\boldsymbol{x}})}=\sum_{\boldsymbol{x}\in \overline{\boldsymbol{x}}}a_{\boldsymbol{x}}b_{\boldsymbol{x}}$} defining the inner product. 
Define the sampling operator $S_{\overline{\boldsymbol{x}}}:\mathcal{U}_{Q}\mapsto \ell^2(\overline{\boldsymbol{x}})$ by
\begin{align}
\label{defS}
S_{\overline{\boldsymbol{x}}}(u)=[\boldsymbol{\psi}, u], \forall u\in \mathcal{U}_Q.
\end{align}
One can view $S_{\overline{\boldsymbol{x}}}$ as a generalization of the sampling operator in \cite{smale2005shannon}, which only {\color{black}evaluates} a function at the sample points.  Denote by $\langle \cdot, \cdot\rangle_{\mathcal{U}_{Q}}$ the inner product of $\mathcal{U}_Q$ and by $S_{\overline{\boldsymbol{x}}}^*$ the adjoint of $S_{\overline{\boldsymbol{x}}}$.  Then, for each $c\in \ell^2(\overline{\boldsymbol{x}})$ and $u\in \mathcal{U}_Q$, we have
\begin{align*}
\langle S^*_{\overline{\boldsymbol{x}}}c, u\rangle_{\mathcal{U}_{Q}} = \langle S_{\overline{\boldsymbol{x}}}u, c\rangle_{\ell^2(\overline{\boldsymbol{x}})} = c^T[\boldsymbol{\psi}, u] = \langle c^T\mathcal{Q}\boldsymbol{\psi}, u\rangle_{\mathcal{U}_{Q}}. 
\end{align*}
Thus, we {\color{black}obtain}
\begin{align}
\label{defSxT}
S_{\overline{\boldsymbol{x}}}^*c = c^T\mathcal{Q}\boldsymbol{\psi}, \forall c\in \ell^2(\overline{\boldsymbol{x}}). 
\end{align}
The next lemma gives some basic properties of $S_{\overline{\boldsymbol{x}}}$ and $S_{\overline{\boldsymbol{x}}}^*$. For a little abuse of notations, in the rest of this section, we denote by $I$ the identity map or the identity matrix, whose meaning is easy to recognize from the context. 
\begin{lemma}
\label{sxprop}
Let $S_{\overline{\boldsymbol{x}}}$ and $S_{\overline{\boldsymbol{x}}}^*$ be as in \eqref{defS} and \eqref{defSxT}. 
Then, for any $\gamma>0$, ${S}_{\overline{\boldsymbol{x}}}^*{S}_{\overline{\boldsymbol{x}}}+\gamma I$ and ${S}_{\overline{\boldsymbol{x}}}{S}_{\overline{\boldsymbol{x}}}^*+\gamma I$ are bijective. Meanwhile, 
\begin{align}
\label{transidt}
{S}_{\overline{\boldsymbol{x}}}^*({S}_{\overline{\boldsymbol{x}}}{S}_{\overline{\boldsymbol{x}}}^*+\gamma I)^{-1} = ({S}_{\overline{\boldsymbol{x}}}^*{S}_{\overline{\boldsymbol{x}}}+\gamma I)^{-1}{S}_{\overline{\boldsymbol{x}}}^*. 
\end{align}
For any $c\in \ell^2(\overline{\boldsymbol{x}})$, we have
\begin{align}
\label{eqre0}
({S}_{\overline{\boldsymbol{x}}}^*{S}_{\overline{\boldsymbol{x}}}+\gamma I)^{-1}{S}_{\overline{\boldsymbol{x}}}^*c = (\mathcal{Q}\boldsymbol{\psi})^T(\gamma I + \mathcal{Q}(\boldsymbol{\psi}, \boldsymbol{\psi}))^{-1}c. 
\end{align}
Furthermore, 
\begin{align}
\label{bdImus}
I-S_{\overline{\boldsymbol{x}}}(S_{\overline{\boldsymbol{x}}}^*S_{\overline{\boldsymbol{x}}}+\gamma I)^{-1}S_{\overline{\boldsymbol{x}}}^*=\gamma(\mathcal{Q}(\boldsymbol{\psi}, \boldsymbol{\psi})+\gamma I)^{-1}. 
\end{align}
\end{lemma}
{\color{black}
We refer the reader to Appendix \ref{profMt} for the proof of Lemma \ref{sxprop}. We note that \eqref{eqre0} and \eqref{bdImus} establishes the relations between the operations of sampling operators $S_{\overline{\boldsymbol{x}}}$ and the covariance matrix $\mathcal{Q}(\boldsymbol{\psi}, \boldsymbol{\psi})$. With the properties of the sampling operator at hand, we are able to prove the existence of minimizers to \eqref{pderegprob}. The following theorem is similar to Proposition 2.3 in \cite{chen2021solving}. A key difference is that the arguments of Proposition 2.3 in \cite{chen2021solving} rely on the representer theorem \cite[Section 17.8]{owhadi2019operator}, which requires the covariance operator to be positive definite, but our proof here holds for general positive semi-definite covariance operators. }
\begin{theorem}
\label{thmexist}
Let $S_{\overline{\boldsymbol{x}}}$ and $S_{\overline{\boldsymbol{x}}}^*$ be defined in \eqref{defS} and \eqref{defSxT}. {\color{black}Let $\overline{R}$ be as in \eqref{defRbar}.} Given $\gamma>0$, 
the system \eqref{pderegprob} admits a solution 
\begin{align}
\label{uoptda}
u^\dagger=(S_{\overline{\boldsymbol{x}}}^*S_{\overline{\boldsymbol{x}}}+\gamma I)^{-1}S_{\overline{\boldsymbol{x}}}^*\boldsymbol{z}^\dagger,  
\end{align}
where $\boldsymbol{z}^\dagger$ is a solution to 
\begin{align}
\label{zsys}
\begin{split}
\begin{cases}
\min\limits_{\boldsymbol{z}\in \mathbb{R}^{\overline{R}}}& \boldsymbol{z}^T(\mathcal{Q}(\boldsymbol{\psi}, \boldsymbol{\psi})+\gamma I)^{-1}\boldsymbol{z}\\
\operatorname{s.t.}& F(\boldsymbol{z})=\boldsymbol{y}.
\end{cases}
\end{split}
\end{align}
\end{theorem}
\begin{proof}
In \eqref{pderegprob}, given $\boldsymbol{z}\in \mathbb{R}^{\overline{R}}$, we first consider the minimization problem over $u$
\begin{align}
\label{1lvmin}
\min_{u\in \mathcal{U}_{Q}}|[\boldsymbol{\psi}, u]-\boldsymbol{z}|^2+\gamma \|u\|^2_{\mathcal{U}_{Q}}. 
\end{align}
By the definition of ${S}_{\overline{\boldsymbol{x}}}$ in \eqref{defS}, we get \begin{align}
\label{objfu1lv}
\begin{split}
|[\boldsymbol{\psi}, u]-\boldsymbol{z}|^2+\gamma \|u\|^2_{\mathcal{U}_{Q}}=&\langle {S}_{\overline{\boldsymbol{x}}}u - \boldsymbol{z}, {S}_{\overline{\boldsymbol{x}}}u - \boldsymbol{z}\rangle_{\ell^2(\overline{\boldsymbol{x}})} + \gamma \langle u, u\rangle_{\mathcal{U}_{Q}}\\
=& \langle {S}_{\overline{\boldsymbol{x}}}^*{S}_{\overline{\boldsymbol{x}}}u-2{S}_{\overline{\boldsymbol{x}}}^*\boldsymbol{z}+\gamma u, u\rangle_{\mathcal{U}_{Q}}+|\boldsymbol{z}|^2. 
\end{split}
\end{align}
Taking the function derivative w.r.t. $u$, we obtain from \eqref{objfu1lv} that
\begin{align*}
    {S}_{\overline{\boldsymbol{x}}}^*{S}_{\overline{\boldsymbol{x}}}u-{S}_{\overline{\boldsymbol{x}}}^*\boldsymbol{z}+\gamma u = 0,
\end{align*}
which implies that the minimizer $u^\dagger$ of \eqref{1lvmin} satisfies
\begin{align}
\label{uoptform}
    u^\dagger = ({S}_{\overline{\boldsymbol{x}}}^*{S}_{\overline{\boldsymbol{x}}}+\gamma I)^{-1}{S}_{\overline{\boldsymbol{x}}}^*\boldsymbol{z}. 
\end{align}
Thus, using \eqref{objfu1lv} and \eqref{uoptform}, we have
\begin{align}
\label{tgtfncnm}
\begin{split}
|[\boldsymbol{\psi}, u^\dagger]-\boldsymbol{z}|^2+\gamma \|u^\dagger\|^2_{\mathcal{U}_{Q}}
=&\langle ({S}_{\overline{\boldsymbol{x}}}^*{S}_{\overline{\boldsymbol{x}}}-2({S}_{\overline{\boldsymbol{x}}}^*{S}_{\overline{\boldsymbol{x}}}+\gamma I)+\gamma I)({S}_{\overline{\boldsymbol{x}}}^*{S}_{\overline{\boldsymbol{x}}}+\gamma I)^{-1}{S}_{\overline{\boldsymbol{x}}}^*\boldsymbol{z}, u^\dagger\rangle_{\mathcal{U}_{Q}} + |\boldsymbol{z}|^2\\
=&-\langle {S}_{\overline{\boldsymbol{x}}}^*\boldsymbol{z}, u^\dagger\rangle_{\mathcal{U}_{Q}} + |\boldsymbol{z}|^2= - \langle \boldsymbol{z}, {S}_{\overline{\boldsymbol{x}}}u^\dagger\rangle_{\ell^2(\overline{\boldsymbol{x}})} + |\boldsymbol{z}|^2\\
=& \langle (I-{S}_{\overline{\boldsymbol{x}}}({S}_{\overline{\boldsymbol{x}}}^*{S}_{\overline{\boldsymbol{x}}}+\gamma I)^{-1}{S}_{\overline{\boldsymbol{x}}}^*)\boldsymbol{z}, \boldsymbol{z}\rangle_{\ell^2(\overline{\boldsymbol{x}})}=\gamma\boldsymbol{z}^T(\gamma I + \mathcal{Q}(\boldsymbol{\psi}, \boldsymbol{\psi}))^{-1}\boldsymbol{z},
\end{split}
\end{align}
where the last equality results from \eqref{bdImus}. Thus, \eqref{tgtfncnm} implies that \eqref{pderegprob} is equivalent to \eqref{zsys}. Using a similar compactness argument as in the proof of Theorem 1.1 in \cite{chen2021solving}, one can show that \eqref{zsys} admits a minimizer $\boldsymbol{z}^\dagger$. Therefore, we conclude \eqref{uoptda} by \eqref{uoptform}. 
\end{proof}
\begin{remark}
\label{dlwcsts}
Theorem \ref{thmexist} implies that in order to find the solution $u^*$ of the PDE system \eqref{pdegrf}, we need to solve \eqref{zsys}. To handle the constraints of \eqref{zsys}, we use the methods of eliminating variables or relaxation described in Subsection 3.3 of \cite{chen2021solving} when necessary. 
\end{remark}
\begin{remark}
\label{dlwinvgram}
To compute the inversion of $\mathcal{Q}(\boldsymbol{\psi}, \boldsymbol{\psi})+\gamma I$, we use the techniques introduced in Subsection \ref{subsecComGram} and Remark \ref{etanug}. 
\end{remark}
Next, we estimate the error between the approximation solution $u^\dagger$ given in \eqref{uoptda} and the strong solution $u^*$ of \eqref{pdegrf}. {\color{black}The proof appears in Appendix \ref{profMt}. }
\begin{theorem}
\label{error_analy}
Let $u^*$ be the strong solution of \eqref{pdegrf}. Given $\gamma>0$,  let $(u^\dagger, \boldsymbol{z}^\dagger)$ be a solution of \eqref{pderegprob}, where $u^\dagger$ is given in \eqref{uoptda}, and  $\boldsymbol{z}^\dagger$ solves \eqref{zsys}. 
Then, 
{\color{black}
\begin{align}
\label{bderror}
\|u^\dagger - u^*\|_{\mathcal{U}}\leq & \inf_{\psi\in \operatorname{span}(\boldsymbol{\psi})}\|u^* -  \mathcal{K}\psi\|_{\mathcal{U}}+ \gamma \sqrt{[\boldsymbol{\psi}, u^*]^T\mathcal{K}(\boldsymbol{\psi}, \boldsymbol{\psi})^{-3}[\boldsymbol{\psi}, u^*]}\nonumber\\
+&  \sqrt{3}|\boldsymbol{z}^\dagger|\sqrt{\|(\gamma I + \mathcal{Q}(\boldsymbol{\psi}, \boldsymbol{\psi}))^{-1}-(\gamma I + \mathcal{K}(\boldsymbol{\psi}, \boldsymbol{\psi}))^{-1}\|}\nonumber \\
&+\sqrt{(\boldsymbol{z}^\dagger - [\boldsymbol{\psi}, u^*])^T\mathcal{K}(\boldsymbol{\psi}, \boldsymbol{\psi})^{-1}(\boldsymbol{z}^\dagger-[\boldsymbol{\psi}, u^*])}.
\end{align}
}
\end{theorem}
{\color{black}
The first term at the right-hand side (RHS) of \eqref{bderror} measures the ability of $\operatorname{span}(\mathcal{K}\boldsymbol{\psi})$ to approximate $u^*$. 
Since $u^*\in \mathcal{U}$, we conclude that as the samples become dense enough, $\inf_{\psi\in \operatorname{span}(\boldsymbol{\psi})}\|u^* -  \mathcal{K}\psi\|_{\mathcal{U}}$ is small. The second term at the RHS of \eqref{bderror} measures the effect of the nugget term $\gamma I$, which converges to $0$ as $\gamma\rightarrow 0$ for a fixed number of samples. The magnitudes of the last two terms of \eqref{bderror} depend on how well $\mathcal{Q}(\boldsymbol{\psi}, \boldsymbol{\psi})$ approximates $\mathcal{K}(\boldsymbol{\psi}, \boldsymbol{\psi})$.  Indeed, when $\mathcal{Q}(\boldsymbol{\psi}, \boldsymbol{\psi})=\mathcal{K}(\boldsymbol{\psi}, \boldsymbol{\psi})$, the SGP method recovers the GP algorithm in \cite{chen2021solving}. Hence, $|\boldsymbol{z}^\dagger|$ is bounded in this case and the third term at the RHS of \eqref{bderror} equals $0$. Besides, $\boldsymbol{z}^\dagger$ converges to $[\boldsymbol{\psi}, u^*]$ as the number of samples, $N$, goes to infinity and $\gamma\rightarrow 0$ by Theorem 3.3 of \cite{chen2021solving}.  Thus, when $\mathcal{Q}=\mathcal{K}$,  as the samples gets dense and $\gamma\rightarrow 0$,  $\boldsymbol{z}^\dagger \mathcal{K}(\boldsymbol{\psi}, \boldsymbol{\psi})^{-1}\boldsymbol{z}^\dagger$ and $\boldsymbol{z}^\dagger \mathcal{K}(\boldsymbol{\psi}, \boldsymbol{\psi})^{-1}[\boldsymbol{\psi}, u^*]$ converge to $\|u^*\|_{\mathcal{U}}^2$. Hence, the last term of \eqref{bderror} converges to $0$. 
}

\subsection{The Nystr\"{o}m Approximation Error}
As {\color{black}mentioned above}, the upper bound in \eqref{bderror} {\color{black}decreases} when $\mathcal{Q}(\boldsymbol{\psi}, \boldsymbol{\psi})$ approximates $\mathcal{K}(\boldsymbol{\psi}, \boldsymbol{\psi})$ well. In this subsection, we {\color{black}study} the spectral norm of $\mathcal{K}(\boldsymbol{\psi}, \boldsymbol{\psi}) - \mathcal{Q}(\boldsymbol{\psi}, \boldsymbol{\psi})$, denoted by $\|\mathcal{K}(\boldsymbol{\psi}, \boldsymbol{\psi}) - \mathcal{Q}(\boldsymbol{\psi}, \boldsymbol{\psi})\|$. 

We recall that $\mathcal{Q}(\boldsymbol{\psi}, \boldsymbol{\psi})=\mathcal{K}(\boldsymbol{\psi}, \boldsymbol{\phi})(\mathcal{K}(\boldsymbol{\phi}, \boldsymbol{\phi}))^{-1}\mathcal{K}(\boldsymbol{\phi}, \boldsymbol{\psi})$. Thus, we can view $\mathcal{Q}(\boldsymbol{\psi}, \boldsymbol{\psi})$ as the Nystr\"{o}m approximation to $\mathcal{K}(\boldsymbol{\psi}, \boldsymbol{\psi})$. Hence, we call $\|\mathcal{K}(\boldsymbol{\psi}, \boldsymbol{\psi}) - \mathcal{Q}(\boldsymbol{\psi}, \boldsymbol{\psi})\|$ the Nystr\"{o}m approximation error in the SGP method.  A typical Nystr\"{o}m approximation  approximates a symmetric positive semi-definite matrix $G\in \mathbb{R}^{n\times n}$, $n\in \mathbb{N}$, by sampling $m$ columns from $G$ and does not assume that columns of $G$ have close relations with each other. The authors in \cite{drineas2005nystrom} show that for any $m$ uniformly sampled columns, with a high probability, the approximation error is $O(n/\sqrt{m})$. Later, \cite{jin2013improved} improves the bound from $O(n/\sqrt{m})$ to $O(n/m^{1-\rho})$ for $\rho\in (0, 1/2)$ under the assumption that the eigenvalues of $G$ have a big eigengap. In our setting, since $\boldsymbol{\psi}$ has linear operators corresponding to the same sample points, the columns of $\mathcal{K}(\boldsymbol{\psi}, \boldsymbol{\psi})$ are highly correlated. Thus, we approximate $\mathcal{K}(\boldsymbol{\psi}, \boldsymbol{\psi})$ using a small number of inducing points instead of using columns from $\mathcal{K}(\boldsymbol{\psi}, \boldsymbol{\psi})$.  In general, instead of being taken from the samples, inducing points can be selected anywhere in the input space \cite{fowlkes2004spectral}. The error analysis has been thoroughly studied in \cite{zhang2008improved}, which indicates that the approximation error in the matrix Frobenious norm has an upper bound that is influenced by the quantization error defined in \cite{zhang2008improved}. Since that quantization error is the objective function in the k-means clustering \cite{gersho2012vector}, {\color{black}the} k-means algorithm is typically used for the initialization of inducing inputs. In our settings, we assume that the inducing points are sampled from the set of samples. We adapt the framework in \cite{jin2013improved} to our settings and show that the Nystr\"{o}m approximation error has an upper bound that depends only on the numbers of samples and inducing points. In other words, though the dimension of $\mathcal{K}(\boldsymbol{\psi}, \boldsymbol{\psi})$ is proportional to the product of the size of samples and the number of linear operators, the error bound is not influenced by the size of linear operators. 

{\color{black}Our arguments are an adaptation to those in  \cite{jin2013improved} to take into consideration of covariance matrices with values of differential operators. The strategy is to turn}  $\|\mathcal{K}(\boldsymbol{\psi}, \boldsymbol{\psi})-\mathcal{K}(\boldsymbol{\psi}, \boldsymbol{\phi})(\mathcal{K}(\boldsymbol{\phi}, \boldsymbol{\phi}))^{-1}\mathcal{K}(\boldsymbol{\phi}, \boldsymbol{\psi})\|$ into a functional approximation problem{\color{black}, which is easier to bound.} Define the sets
\begin{align}
	\label{defsetsns}
	\mathcal{H}_a=\operatorname{span}\{\mathcal{K}\mathcal{\boldsymbol{\phi}}\} \text{ and } \mathcal{H}_b=\{\boldsymbol{\beta^T\mathcal{K}\boldsymbol{\psi}}, |\boldsymbol{\beta}|\leq 1\}. 
\end{align}
For $h\in \mathcal{H}_b$, we define
\begin{align*}
	\mathcal{E}(h, \mathcal{H}_a)=\min_{v\in \mathcal{H}_a}\|v-h\|^2_{\mathcal{U}}.
\end{align*}
Then, $\mathcal{E}(h, \mathcal{H}_a)$ is the minimum error for approximating a function $h\in \mathcal{H}_b$ by functions in $\mathcal{H}_a$. Meanwhile, define $\mathcal{E}(\mathcal{H}_a)$ as the worst error in approximating any function $h\in \mathcal{H}_b$ by functions in $\mathcal{H}_a$, i.e.
\begin{align}
	\label{defeha}
	\mathcal{E}(\mathcal{H}_a)=\max_{h\in \mathcal{H}_b}\mathcal{E}(h, \mathcal{H}_a). 
\end{align}

The next proposition shows the equivalence of $\|\mathcal{K}(\boldsymbol{\psi}, \boldsymbol{\psi})-\mathcal{K}(\boldsymbol{\psi}, \boldsymbol{\phi})(\mathcal{K}(\boldsymbol{\phi}, \boldsymbol{\phi}))^{-1}\mathcal{K}(\boldsymbol{\phi}, \boldsymbol{\psi})\|$ and $\mathcal{E}(\mathcal{H}_a)$. 
\begin{pro}
\label{nystroapprox}
Let $\mathcal{E}(\mathcal{H}_a)$, $\boldsymbol{\psi}$, and $\boldsymbol{\phi}$ be as in \eqref{defeha}, \eqref{defpsi}, and \eqref{defindiline}, {\color{black}respectively}. 
Then, 
\begin{align}
\label{pronyser}
\|\mathcal{K}(\boldsymbol{\psi}, \boldsymbol{\psi})-\mathcal{K}(\boldsymbol{\psi}, \boldsymbol{\phi})(\mathcal{K}(\boldsymbol{\phi}, \boldsymbol{\phi}))^{-1}\mathcal{K}(\boldsymbol{\phi}, \boldsymbol{\psi})\|=\mathcal{E}(\mathcal{H}_a). 
\end{align}
\end{pro}
\begin{proof}
Let $h\in \mathcal{H}_b$ and $v\in \mathcal{H}_a$. Then, there exists $\boldsymbol{\alpha}$ and $\boldsymbol{\beta}$ such that
$v=\boldsymbol{\alpha}^T\mathcal{K}\boldsymbol{\phi}$ and $h=\boldsymbol{\beta}^T\mathcal{K}\boldsymbol{\psi}$. We have
\begin{align*}
\begin{split}
\mathcal{E}(h, \mathcal{H}_a)=&\min_{\boldsymbol{\alpha}}\boldsymbol{\alpha}^T\mathcal{K}(\boldsymbol{\phi}, \boldsymbol{\phi})\boldsymbol{\alpha}-2\boldsymbol{\alpha}^T\mathcal{K}(\boldsymbol{\phi}, \boldsymbol{\psi})\boldsymbol{\beta}+\boldsymbol{\beta}^T\mathcal{K}(\boldsymbol{\psi}, \boldsymbol{\psi})\boldsymbol{\beta}\\
=& \boldsymbol{\beta}^T(\mathcal{K}(\boldsymbol{\psi}, \boldsymbol{\psi})-\mathcal{K}(\boldsymbol{\psi}, \boldsymbol{\phi})(\mathcal{K}(\boldsymbol{\phi}, \boldsymbol{\phi}))^{-1}\mathcal{K}(\boldsymbol{\phi}, \boldsymbol{\psi}))\boldsymbol{\beta}.
\end{split}
\end{align*}
Hence, we obtain 
\begin{align*}
\mathcal{E}(\mathcal{H}_a)=&\max_{h\in \mathcal{H}_b}\mathcal{E}(h, \mathcal{H}_b)\\
=& \max_{|\boldsymbol{\beta}|\leq 1}\boldsymbol{\beta}^T(\mathcal{K}(\boldsymbol{\psi}, \boldsymbol{\psi})-\mathcal{K}(\boldsymbol{\psi}, \boldsymbol{\phi})(\mathcal{K}(\boldsymbol{\phi}, \boldsymbol{\phi}))^{-1}\mathcal{K}(\boldsymbol{\phi}, \boldsymbol{\psi}))\boldsymbol{\beta}\\
=& \|\mathcal{K}(\boldsymbol{\psi}, \boldsymbol{\psi})-\mathcal{K}(\boldsymbol{\psi}, \boldsymbol{\phi})(\mathcal{K}(\boldsymbol{\phi}, \boldsymbol{\phi}))^{-1}\mathcal{K}(\boldsymbol{\phi}, \boldsymbol{\psi})\|,
\end{align*}
which concludes \eqref{pronyser}. 
\end{proof}
{\color{black}Thus, Proposition \ref{nystroapprox} implies that estimating {\color{black}an} upper bound for $\|\mathcal{K}(\boldsymbol{\psi}, \boldsymbol{\psi})-\mathcal{K}(\boldsymbol{\psi}, \boldsymbol{\phi})(\mathcal{K}(\boldsymbol{\phi}, \boldsymbol{\phi}))^{-1}\mathcal{K}(\boldsymbol{\phi}, \boldsymbol{\psi})\|$ is equivalent to bounding $\mathcal{E}(\mathcal{H}_a)$.} To proceed with the argument, we make the following assumptions about the inducing points and the kernel associated with the RKHS $\mathcal{U}$.

{\color{black}The first assumption assumes that the inducing points are uniformly sampled from the collection of sample points.}
\begin{hyp}
	\label{hypunisp}
	Let  $\overline{\boldsymbol{x}}:=(\overline{\boldsymbol{x}}_1, \overline{\boldsymbol{x}}_2)$, where $\overline{\boldsymbol{x}}_1$ is the set of interior sample points in $\Omega$ and $\overline{\boldsymbol{x}}_2$ is the collection of boundary samples, and let $\widehat{\boldsymbol{x}}:=(\widehat{\boldsymbol{x}}_1, \widehat{\boldsymbol{x}}_2)$ be the set of inducing points. Assume that  $\widehat{\boldsymbol{x}}_1$ and $\widehat{\boldsymbol{x}}_2$ are uniformly sampled from $\overline{\boldsymbol{x}}_1$ and  $\overline{\boldsymbol{x}}_2$, separately. 
\end{hyp}
{\color{black}The next assumption supposes that the covariance operator $\mathcal{K}$ admits enough regularity.}
\begin{hyp}
	\label{hypliopt}
	Let $\boldsymbol{\psi}$ be as in \eqref{defpsi}. {\color{black}Let $\boldsymbol{\psi}^{\boldsymbol{x}}{\subset}(\delta_{\boldsymbol{x}}\circ L_{1}, \dots, \delta_{\boldsymbol{x}}\circ L_{D})$ be} the vector of linear operators in $\boldsymbol{\psi}$ that is  associated with the sample point at $\boldsymbol{x}\in \overline{\boldsymbol{x}}$.   Assume that there exists a constant $C>0$ such that {\color{black} 
	\begin{align*}
		\begin{split}
\operatorname{trace}(\mathcal{K}(\boldsymbol{\psi}^{\boldsymbol{x}}, \boldsymbol{\psi}^{\boldsymbol{x}})) \leq C.
		\end{split}
	\end{align*}
 }
\end{hyp}
\begin{remark}
	We note that if the kernel $K$ in \eqref{defkernelK} generated by $\mathcal{K}$ is smooth enough and the domain $\overline{\Omega}$ is compact,  Assumption \ref{hypliopt} holds. 
\end{remark}
{\color{black}
We provide an upper bound for $\mathcal{E}(\mathcal{H}_a)$ using arguments based on the eigenvalue decomposition. To do this, we define some notation. We denote $\boldsymbol{\psi}:=(\boldsymbol{\psi}_\Omega, \boldsymbol{\psi}_{\partial \Omega})$, where $\boldsymbol{\psi}_\Omega$ represents linear operators for interior points, and $\boldsymbol{\psi}_{\partial \Omega}$ contains linear operators for boundary samples. We recall that $\overline{\boldsymbol{x}}:=(\overline{\boldsymbol{x}}_1, \overline{\boldsymbol{x}}_2)$ contains all sample points, with $\overline{\boldsymbol{x}}_1$ being the set of interior points in $\Omega$ and $\overline{\boldsymbol{x}}_2$ representing the collection of boundary points. We define $N_\Omega$ as the size of $\overline{\boldsymbol{x}}_1$ and $N_{\partial \Omega}$ as the length of $\overline{\boldsymbol{x}}_2$. We also define $R_1$ and $R_2$ as $R_1=N_\Omega(D-D_b)$ and $R_2=(N-N_\Omega)D_b$, respectively, where $D$ and $D_b$ are given in \eqref{pdegrf}. We use $\{\lambda_j\}_{j=1}^{R_1}$ and $\{\tau_j\}_{j=1}^{R_2}$ to represent the set of eigenvalues of $\mathcal{K}(\boldsymbol{\psi}_\Omega, \boldsymbol{\psi}_\Omega)$ and $\mathcal{K}(\boldsymbol{\psi}_{\partial \Omega}, \boldsymbol{\psi}_{\partial \Omega})$, respectively, sorted in descending order. We denote $[i]$ as the set of all non-negative integers less than or equal to $i$. Using these notations, we establish an upper bound for $\mathcal{E}(\mathcal{H}_a)$ in the following theorem. The proof is presented in Appendix \ref{prfNSE}.}  
\begin{theorem}
	\label{corobdls}
	Suppose that Assumptions \ref{hypunisp} and \ref{hypliopt} hold. Let $\mathcal{E}(\mathcal{H}_a)$ be as in \eqref{defeha}. With a probability at least $1-\delta$, $\delta\in (0, 1)$, for any $\boldsymbol{r}:=(r_1, r_2)\in [R_1]\times [R_2]$, {\color{black}there exists a constant $C$ such that} 
	\begin{align}
		\label{bdE}
		\mathcal{E}(\mathcal{H}_a) 
		\leq & \max\bigg\{ \frac{C N_{\Omega}^2\ln^2(2/\delta)}{\lambda_{r_1}M_\Omega}, \frac{CN_{\partial\Omega}^2\ln^2(2/\delta)}{\tau_{r_2}M_{\partial\Omega}} \bigg\}+4\lambda_{r_1+1}+4\tau_{r_2+1},
	\end{align}
	with the convention that $\lambda_{R_1+1}=0$ and $\tau_{R_2+1}=0$.  
\end{theorem}
	Proposition \ref{nystroapprox} and Theorem \ref{corobdls} imply that  $\|\mathcal{K}(\boldsymbol{\psi}, \boldsymbol{\psi})-\mathcal{K}(\boldsymbol{\psi}, \boldsymbol{\phi})(\mathcal{K}(\boldsymbol{\phi}, \boldsymbol{\phi}))^{-1}\mathcal{K}(\boldsymbol{\phi}, \boldsymbol{\psi})\|$ admits an upper bound that depends only on the size of sample points and is not influenced by the number of linear operators in the PDE.

\section{Numerical Results}
\label{secNumResults}
{\color{black}In this section, we study numerically how much information the SGP method preserves for different choices of numbers of inducing points.}
We demonstrate the efficacy of the SGP method by solving a nonlinear elliptic equation in Subsection \ref{subsecEliip}, {\color{black} a mean-field game system in Subsection \ref{subsecMFG}},  Burger's equation in Subsection \ref{subsecBurgers}, and a parabolic equation in Subsection \ref{subsecParabo}. For all the experiments, we take $N$ samples in the domain in a way such that $N_\Omega$ samples are in the interior. Besides, we randomly choose $M$ points from the samples and treat them as inducing points. To show the performance of the SGP method, we first plot samples, inducing points, loss histories of the Gauss--Newton iterations, and contours of pointwise solution errors for fixed values of $N$ and $M$. Though we solve  at a collection of uniform random samples, we  
compute the solution errors by reevaluating the numerical approximation function at a {\color{black}$60\times 60$} equally spaced grid using \eqref{uoptda} and by comparing the results with the values of {\color{black}reference solutions}. Then, we record the $L^\infty$ errors of  the GP method in  \cite{chen2021solving} and the SGP algorithm for different values of $N$ and $M$.  Results
are averaged over $10$ realizations of  random sample points.

 Our implementation is based on the code\footnote{https://github.com/yifanc96/NonlinearPDEs-GPsolver.git} of  \cite{chen2021solving}, which leverages Python with the JAX package for automatic differentiation. Our experiments are conducted only on CPUs. Additional performance speedups can be obtained by considering accelerated hardware such as
Graphics Processing Units (GPU).

\subsection{A Nonlinear Elliptic Equation}
\label{subsecEliip}
In this example, we reconsider the nonlinear elliptic equation {\color{black}in Subsection \ref{subsecSGP}}. More precisely, we consider \eqref{illpde} with  $\Omega=(0, 3)^2$ and $g(\boldsymbol{x})=0$. We prescribe the solution $u$ to be $\sin(\pi x_1)\sin(\pi x_2)+4\sin(4\pi x_1)\sin(4\pi x_2)$ for $\boldsymbol{x}:=(x_1, x_2)\in \overline{\Omega}$ and compute $f$ accordingly. We use the Gaussian kernel
$K(\boldsymbol{x}, \boldsymbol{y}) =\operatorname{exp}(-\frac{|\boldsymbol{x}-\boldsymbol{y}|^2}{2\sigma^2})$ with the lengthscale parameter $\sigma=0.2$.   Given $N$ and $M$, we set $N_\Omega = 0.75 \times N$ and $M_\Omega = 0.75\times M$. In this example,  we choose  $\gamma=\eta=10^{-12}$ (see Subsection \ref{subsecComGram} and Remark \ref{etanug}). The algorithm stops once the error between two successive steps is less than $10^{-5}$.  

Figure \ref{fig:NonlinearElliptic} shows the numerical results for $N=9600$ and $M=1200$. The uniform sample points used are plotted in Figure \ref{fig:NonlinearElliptic:sample_points}, from which we randomly choose $M$ including points shown in Figure \ref{fig:NonlinearElliptic:inducing_points}.  We solve the unconstrained minimization problem in \eqref{regoptzotmi}. The convergence history of the Gauss--Newton iteration in {\color{black}Figure} \ref{fig:NonlinearElliptic:loss_hist}  shows that the SGP method converges in $5$ iterations. In Figure \ref{fig:NonlinearElliptic:errors}, we plot the contour of pointwise errors between the numerical approximation and the true solution on a {\color{black}$60\times 60$} equally spaced grid.

\begin{figure}[!hbtbp]
	\centering          
		\begin{subfigure}[b]{0.3\textwidth}
			\includegraphics[width=\textwidth]{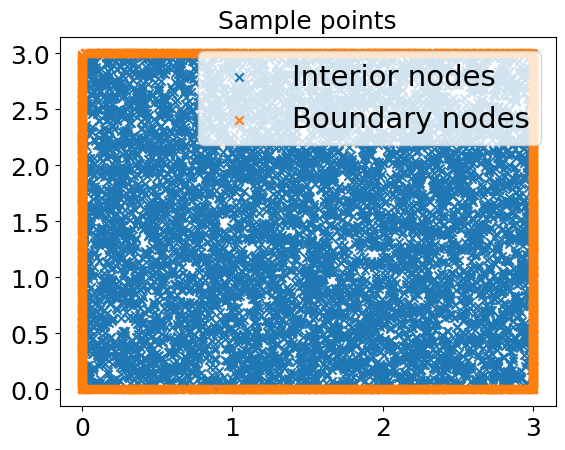}
			\caption{Sample points.}
			\label{fig:NonlinearElliptic:sample_points}
		\end{subfigure} 
		\begin{subfigure}[b]{0.3\textwidth}
			\includegraphics[width=\textwidth]{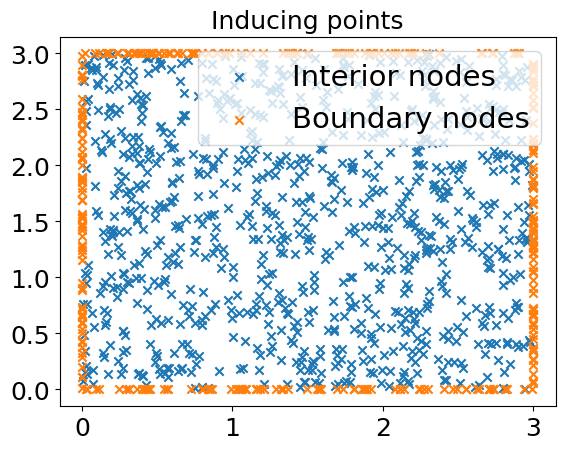}
			\caption{Inducing points.}
			\label{fig:NonlinearElliptic:inducing_points}
		\end{subfigure} \\
		\begin{subfigure}[b]{0.3\textwidth}
			\includegraphics[width=\textwidth]{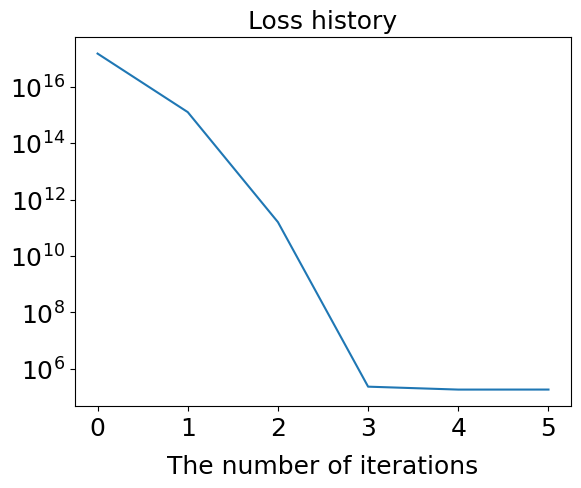}
			\caption{Loss history.}
			\label{fig:NonlinearElliptic:loss_hist}
		\end{subfigure} 
		\begin{subfigure}[b]{0.3\textwidth}
			\includegraphics[width=\textwidth]{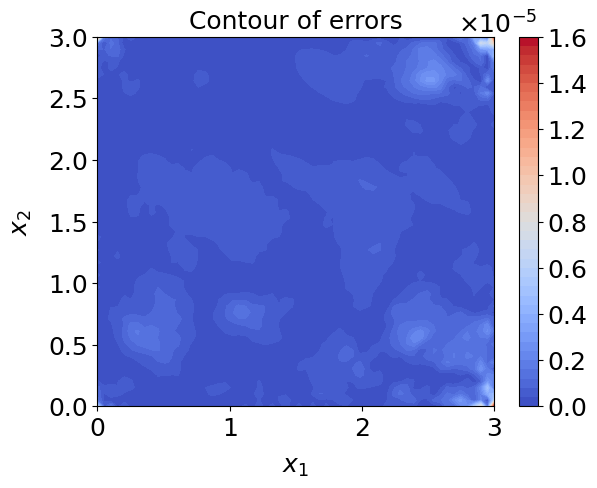}
			\caption{Pointwise errors.}
			\label{fig:NonlinearElliptic:errors}
		\end{subfigure}
	\caption{\color{black}Numerical results for the nonlinear elliptic equation: (a) a set of sample points; (b) a collection of inducing points; (c) the convergence history of the Gauss--Newton iterations; (d) contours of pointwise errors. The regularization parameters $\gamma=\eta=10^{-12}$. $N=9600, N_\Omega=0.75\times N$, $M=1200, M_\Omega=0.75\times M$.}
	\label{fig:NonlinearElliptic}
\end{figure}

To further explore the efficacy of our algorithm. We compare the $L^\infty$ errors of the SGP method and  the GP algorithm in \cite{chen2021solving} for different numbers of sample points and inducing points in Table \ref{tb:error_comp}. We see that when $N=M$, both algorithms achieve accuracy of the same magnitude. Meanwhile, when $N>>M$, the SGP method with $N$ sample points and $M$ inducing points is more accurate than the GP method with $M$ sample points. {\color{black}This phenomenon implies that for answering Problem \ref{main_prob},  purely reducing the number of base functions in a function space may not result in an efficient subspace.} Furthermore, {\color{black}Table \ref{tb:error_comp} implies that} the SGP method with $N$ sample points, half of which are used as inducing points, achieves comparable accuracy to the GP method with the same number of sample points. {\color{black}This is also confirmed in Figure \ref{fig:NE:Accuracy}, where we plot the means and the standard variances of the maximum pointwise errors over $10$ different realizations against different numbers of inducing points when the number of samples is fixed to be $4800$. Figure \ref{fig:NE:Accuracy} shows that the accuracy is not deteriorated if the number of inducing points is larger than $2000$, which is smaller than half the number of samples.}  Hence, by using \eqref{rIQpsiphi}, the size of the matrix needed to be inverted in the SGP method can be significantly reduced. In other words, the SGP approach consumes much less computational time while retaining desirable accuracy compared to the GP method.  {\color{black}Figure \ref{fig:NE:Accuracy} also implies that if the number of inducing points is too small, the induced space $\mathcal{U}_Q$ cannot maintain sufficient information about the approximated solution. Therefore, there is a trade-off between accuracy and efficiency.} {\color{black}If the number of inducing points is chosen properly, the subspace generated by the SGP method serves as a candidate answer to Problem \ref{main_prob} in terms of the GP method.}

\begin{table}[ht!]
	\centering
 {\color{black}
	\begin{tabular}{c  c  c  c c}
		\hline
		$N$ & $1200$ &  $2400$ &  $4800$ & $9600$ \\
		\hline
		GP, $L^\infty$ error & $1.34\times 10^{-1}$  & $1.01\times 10^{-3}$ & $6.78\times 10^{-5}$  &  $4.46\times 10^{-6}$  \\
		SGP($M=600$), $L^\infty$ error& $1.46\times 10^{-1}$  &  $9.34\times 10^{-3}$ &  $2.55\times 10^{-3}$ & $2.91\times 10^{-3}$  \\
		SGP($M=1200$), $L^\infty$ error &  $1.40\times 10^{-1}$  &   $1.37\times 10^{-3}$ &  $7.58\times 10^{-5}$ & $2.16\times 10^{-5}$  \\
		SGP($M=2400$), $L^\infty$ error &  N/A & $1.92\times 10^{-3}$  & $2.62\times 10^{-5}$ & $6.20\times 10^{-6}$ \\
		SGP($M=4800$), $L^\infty$ error &  N/A & N/A & $3.15\times 10^{-5}$ & $5.13\times 10^{-6}$ \\
		\hline
	\end{tabular}
	\caption{\color{black}The nonlinear elliptic equation: $L^\infty$ errors of the SGP method and the GP algorithm for different numbers of sample points and inducing points. The N/A means not available. We set $\gamma=\eta=10^{-12}$. }
	\label{tb:error_comp}
 }
\end{table}

{
\color{black}
\begin{figure}[!hbtbp]
	\centering          
		\begin{subfigure}[b]{0.3\textwidth}
			\includegraphics[width=\textwidth]{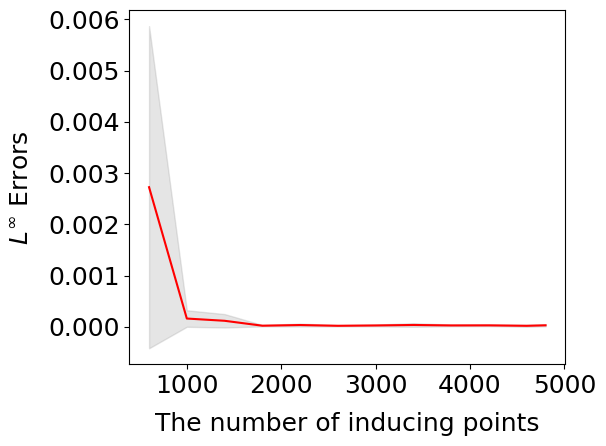}
		\end{subfigure} 
	\caption{\color{black}The nonlinear elliptic equation: $L^\infty$ errors for different numbers of inducing points when the number of samples $N$ equals to $4800$.}
	\label{fig:NE:Accuracy}
\end{figure}
}

{
\color{black}
\subsection{A Stationary Mean Field Game}
\label{subsecMFG}
In this subsection, we explore the behavior of the SGP method by solving  a stationary mean field game (MFG) system, which  consists of two coupled PDEs, a Hamilton--Jacobi equation and a Fokker--Planck equation. 
MFGs model the behaviors of agents in a large population. To learn more about the use of finite differences and machine learning methods for solving MFGs, we suggest referring to the references \cite{achdou2010mean, briceno2018proximal, ruthotto2020machine, lin2020apac}. The GP method and a random Fourier feature algorithm for solving MFG  are presented in \cite{mou2022numerical}. Following the same framework as in \cite{mou2022numerical}, our SGP method can be naturally extended to solve PDE systems.  In this example, we seek $(u, m, \lambda)\in C^\infty(\mathbb{T}^2)\times C^\infty(\mathbb{T}^2)\times\mathbb{R}$ solving 
\begin{align}
\label{mfgsys}
\begin{cases}
-\nu\Delta u + \frac{|\nabla u|^2}{2} +\frac{ \sin(4\pi x_1) + \cos(4\pi x_1) +\sin(4\pi x_2)}{2} = m^2 + \lambda,  \forall x:=(x_1, x_2)\in \mathbb{T}^2,\\
-\nu\Delta m - \div(m\nabla u) = 0, \forall x:=(x_1, x_2)\in \mathbb{T}^2, \\
\int_{\mathbb{T}^2}m\dif x = 1, \int_{\mathbb{T}^2} u\dif x = 0,
\end{cases}
\end{align}
where $\nu=0.1$ and $\mathbb{T}^2$ is the two-dimensional torus.  In the computations, we identify $\mathbb{T}^2$ with $[-0.5, 0.5]^2$. Meanwhile, to deal with the periodic condition, we choose the  kernel
\begin{align}
\label{pok}
K((x_1, x_2), (x_1', x_2'))=\operatorname{exp}(\cos(2\pi(x_1-x_1'))+\cos(2\pi(x_2-x_2'))-2).
\end{align}
Then, we randomly select $M$ inducing points from $N$ uniform random samples. For simplicity, we approximate both $u$ and $m$ in the same RKHS $\mathcal{U}_{Q}$ generated by inducing points as described in Subsection \ref{subsecRKHSIP}. More precisely, we consider the following  minimization problem
\begin{align}
\label{mfgopt}
\begin{cases}
\min_{u, m\in \mathcal{U}_{Q}, \lambda\in \mathbb{R}, \boldsymbol{z}, \boldsymbol{\rho}\in \mathbb{R}^N} \gamma\|u\|_{\mathcal{U}_{Q}}^2+\gamma\|m\|_{\mathcal{U}_Q}^2 + \gamma\lambda^2 + \sum_{j=1}^N|z_j^{(1)}-u(\boldsymbol{x}_j)|^2
 + \sum_{j=1}^N|z_j^{(2)}-\partial_{x_1} u(\boldsymbol{x}_j)|^2\\
\qquad +\sum_{j=1}^N|z_j^{(3)}-\partial_{x_2}u(\boldsymbol{x}_j)|^2+\sum_{j=1}^N|z_j^{(4)}-\Delta u(\boldsymbol{x}_j)|^2 + \sum_{j=1}^N|\rho_j^{(1)}-m(\boldsymbol{x}_j)|^2\\
\quad\quad + \sum_{j=1}^N|\rho_j^{(2)}-\partial_{x_1} m(\boldsymbol{x}_j)|^2 +\sum_{j=1}^N|\rho_j^{(3)}-\partial_{x_2}m(\boldsymbol{x}_j)|^2+\sum_{j=1}^N|\rho_j^{(4)}-\Delta m(\boldsymbol{x}_j)|^2 \\
\quad\quad +\sum_{j=1}^N |\nu z_j^{(4)} - \frac{(z_j^{(2)} + z_j^{(3)})^2}{2}-\frac{\sin(4\pi x_{1,j})+\cos(4\pi x_{1,j})+\sin(4\pi x_{2,j})}{2}+(\rho_j^{(1)})^2+\lambda|^2\\
\quad\quad +\sum_{j=1}^N|-\nu \rho^{(4)}_j - \rho^{(2)}_jz_{j}^{(2)}   - \rho^{(3)}_jz_{j}^{(3)}  - \rho^{(1)}_jz^{(4)}|^2,\\
\text{s.t. } \sum_{j=1}^Nz_j^{(1)}=0,  \frac{1}{N}\sum_{j=1}^N\rho_j^{(1)}=1.  
\end{cases}
\end{align}
Then, we derive the representer formulas for $u$ and $m$ as what we did in Section \ref{secGF}. To study the efficacy of our SGP method, for clarification, we only compare the pointwise errors between a minimizer $\boldsymbol{m}^\dagger$ of \eqref{mfgopt} and a reference solution $m^*$ on a $60\times 60$ uniform grid. To get $m^*$, we use the proximal method proposed in  \cite{briceno2018proximal} to compute \eqref{mfgsys} on a $120\times 120$ uniform grid and linearly interpolate the solution on the coarser $60\times 60$ grid.  

Figure \ref{fig:mfg} illustrates the numerical results when we take $N=800$ samples and $M=50$ inducing points. Here, we set $\gamma=10^{-10}$ and $\eta=10^{-4}$. Figures \ref{fig:mfg:sample_points} and \ref{fig:mfg:inducing_points} plot the samples and inducing points, respectively. Figure \ref{fig:mfg:loss_hist} shows the loss history of the Gauss--Newton iterations and Figure \ref{fig:mfg:errors} presents the contour of pointwise errors. Table \ref{tb:mfg:error_comp} compares the $L^\infty$ errors of $m^*$ and $m^\dagger$ for different values of $N$ and $M$. We see that for solving the MFG system \eqref{mfgsys}, the SGP method exhibits similar efficacy as in the previous experiments for solving a single nonlinear elliptic equation. That is, the SGP method which takes half of the samples as inducing points, achieves comparable accuracy to the GP method using the same number of samples.

\begin{figure}[!hbtbp]
	\centering          
	\begin{subfigure}[b]{0.3\textwidth}
		\includegraphics[width=\textwidth]{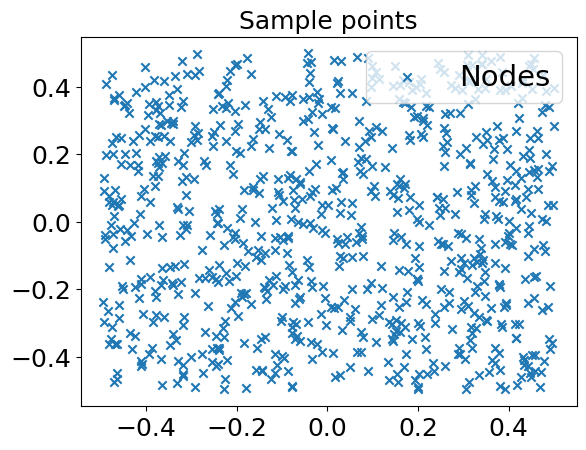}
		\caption{Sample points.}
		\label{fig:mfg:sample_points}
	\end{subfigure} 
	\begin{subfigure}[b]{0.3\textwidth}
		\includegraphics[width=\textwidth]{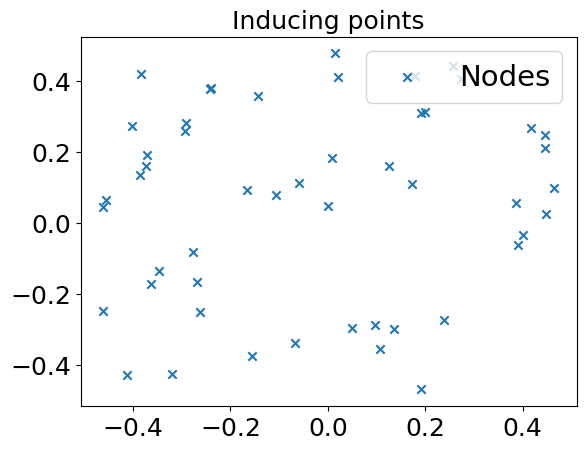}
		\caption{Inducing points.}
		\label{fig:mfg:inducing_points}
	\end{subfigure} \\
	\begin{subfigure}[b]{0.3\textwidth}
		\includegraphics[width=\textwidth]{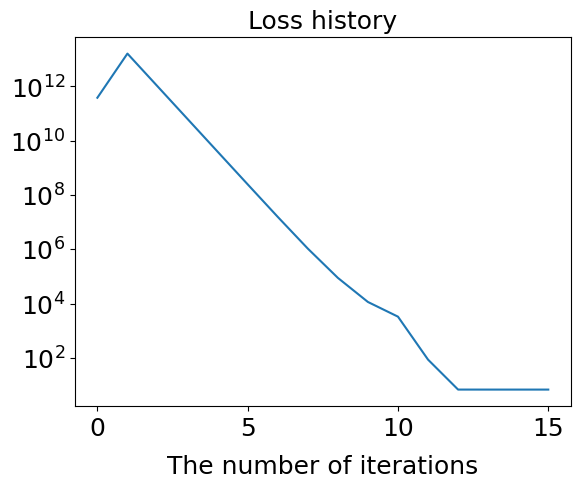}
		\caption{Loss history.}
		\label{fig:mfg:loss_hist}
	\end{subfigure} 
	\begin{subfigure}[b]{0.3\textwidth}
		\includegraphics[width=\textwidth]{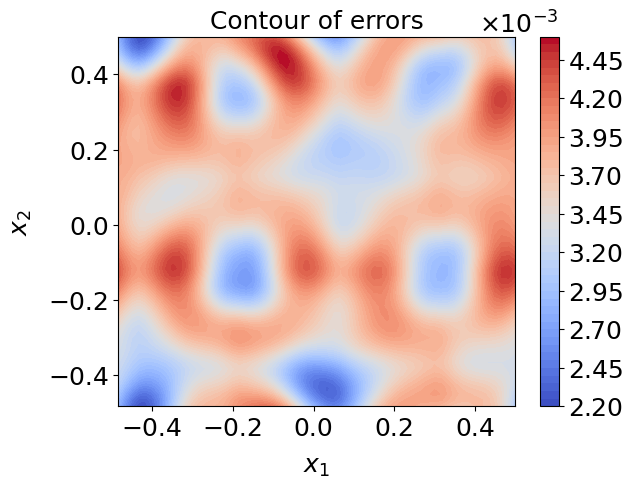}
		\caption{Pointwise errors.}
		\label{fig:mfg:errors}
	\end{subfigure}
	\caption{\color{black}Numerical results for the stationary MFG: (a) a set of sample points; (b) a collection of inducing points; (c) the convergence history of the Gauss--Newton iterations; (d) contours of pointwise errors. The regularization parameters $\gamma=10^{-10}$ and $\eta=10^{-4}$. $N=800$ and $M=50$.}
	\label{fig:mfg}
\end{figure}

\begin{table}[ht!]
	\centering
 {\color{black}
	\begin{tabular}{c  c  c  c c}
		\hline
		$N$ & $100$ &  $200$ &  $400$ & $800$ \\
		\hline
		GP, $L^\infty$ error &  $1.09\times 10^{-1}$ & $2.31\times 10^{-2}$ &  $1.05 \times 10^{-2}$ & $4.36\times 10^{-3}$ \\
		SGP($M=50$), $L^\infty$ error & $1.16\times 10^{-1}$  & $2.61\times 10^{-2}$  & $1.10\times 10^{-2}$  & $6.20\times 10^{-3}$ \\
		SGP($M=100$), $L^\infty$ error & $1.18\times 10^{-1}$   & $2.31\times 10^{-2}$  & $1.25\times 10^{-2}$  & $5.32\times 10^{-3}$  \\
		SGP($M=200$), $L^\infty$ error &  N/A & $2.26\times 10^{-2}$  & $1.15\times 10^{-2}$ & $6.97\times 10^{-3}$ \\
		SGP($M=400$), $L^\infty$ error &  N/A & N/A   & $1.12\times 10^{-2}$ & $4.04\times 10^{-3}$ \\
		\hline
	\end{tabular}
	\caption{\color{black}The stationary MFG: The $L^\infty$ errors of the SGP method and the GP algorithm are analyzed under varying sample point and inducing point configurations. The N/A means not available. We set $\gamma=10^{-10}$ and $\eta=10^{-4}$.}
	\label{tb:mfg:error_comp}
 }
\end{table}

}

\subsection{Burgers' equation}
\label{subsecBurgers}
In this example, we consider the same Burger's equation as in \cite{chen2021solving}. More precisely, given $\nu=0.02$, we seek to find $u$ solving 
\begin{align*}
\begin{cases}
\partial_t u + u \partial_x u + \nu \partial_x^2 u = 0, \forall (t, x)\in (0, 1]\times (-1, 1),\\
u(0, x) = -\sin(\pi x), \forall x\in [-1, 1],\\
u(t, -1) = u(t, 1)=0, \forall t\in [0, 1].
\end{cases}
\end{align*}
We uniformly sample $N$ points in the space-time domain, from which $M$ samples are randomly chosen as inducing points. All experiments in this example use a fixed ratio of interior points, $N_\Omega/N = M_\Omega/M=5/6$. To take into consideration the space and time variability of the solution to Burger's equation, as in \cite{chen2021solving}, we choose the anisotropic kernel
\begin{align}
\label{anisotropickn}
K((t, x), (t', x'))=\operatorname{exp}(-(t-t')^2/\sigma_1^2-(x-x')^2/{\sigma_2^2}),
\end{align}
with $(\sigma_1, \sigma_2)=(0.3, 0.05)$. For comparison, as in \cite{chen2021solving}, the true solution is calculated from the Cole--Hopf transformation and the numerical quadrature. We use the technique of eliminating variables to deal with the constraints in \eqref{zsys} (see Subsection 3.3 in \cite{chen2021solving}) and use the Gauss--Newton method to solve the resulting unconstrained minimization problem. All the experiments in this example stop once the errors between two successive steps are less than $10^{-5}$. 

Figure \ref{fig:burgers} shows the numerical results for $N=2400$ and $M=600$ with parameters $\gamma=\eta=10^{-6}$. The samples and the inducing points are given in Figures \ref{fig:burgers:sample_points} and \ref{fig:burgers:inducing_points}. The history of Gauss--Newton iterations is plotted in Figure \ref{fig:burgers:loss_hist}, which implies the fast convergence of the SGP method. The absolute values of the pointwise errors between the numerical solution and the true solution are shown in Figure \ref{fig:burgers:errors}. Similar to the results of the GP method given in \cite{chen2021solving}, the maximum errors occurred close to the 
shock. In Figures \ref{fig:burgers:0.25.png}-\ref{fig:burgers:0.75.png}, we  compare the
numerical and true solutions at times $t = 0.25, 0.5, 0.75$ to  highlight the accuracy of our SGP method.

In Table \ref{tb:burgers:error_comp},  we present the $L^\infty$ errors of the SGP method and  the GP algorithm in \cite{chen2021solving} for different choices of $N$ and $M$. Similar to the nonlinear elliptic example, the SGP algorithm yields comparable accurate numerical solutions to the GP method by using half of the samples as inducing points. 

\begin{figure}[!hbtbp]
	\centering    
		\begin{subfigure}[b]{0.3\textwidth}
			\includegraphics[width=\textwidth]{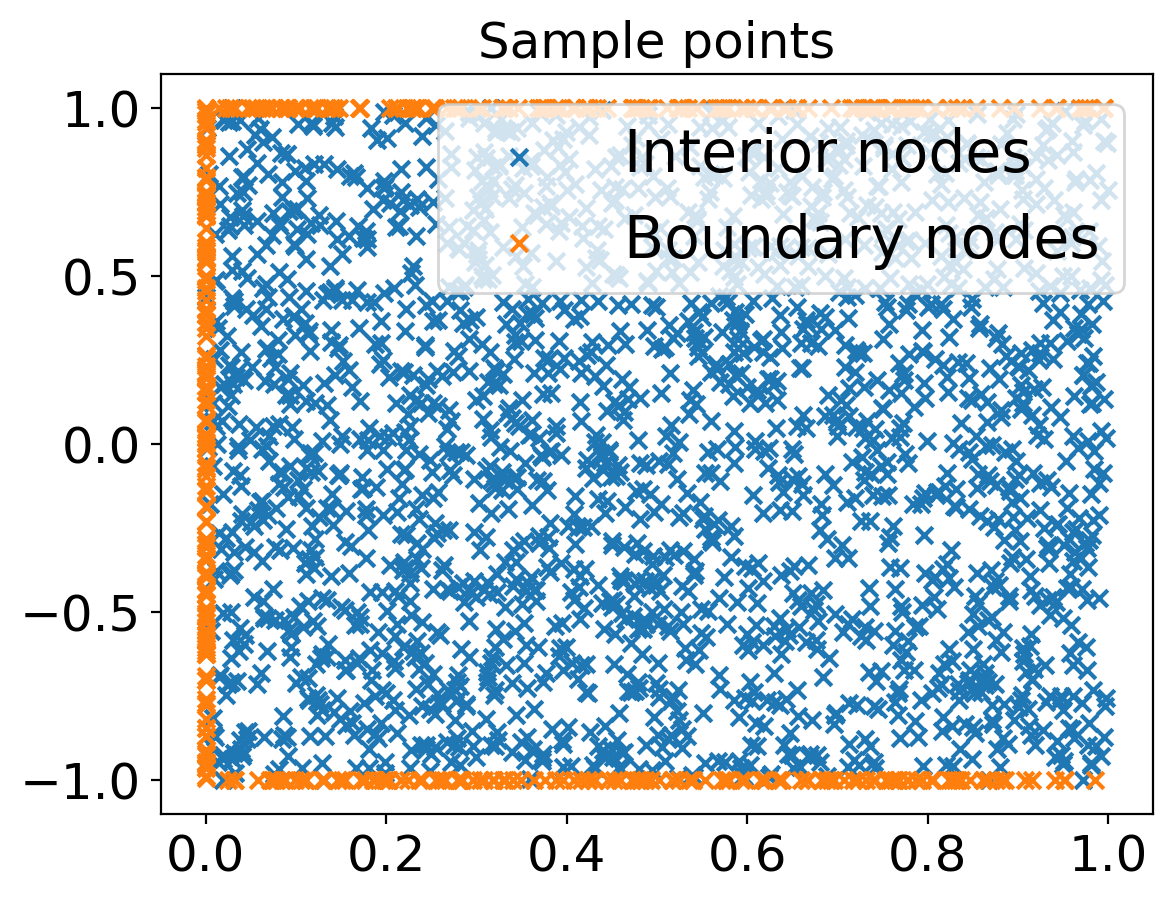}
			\caption{Sample points.}
			\label{fig:burgers:sample_points}
		\end{subfigure}
		\begin{subfigure}[b]{0.3\textwidth}
			\includegraphics[width=\textwidth]{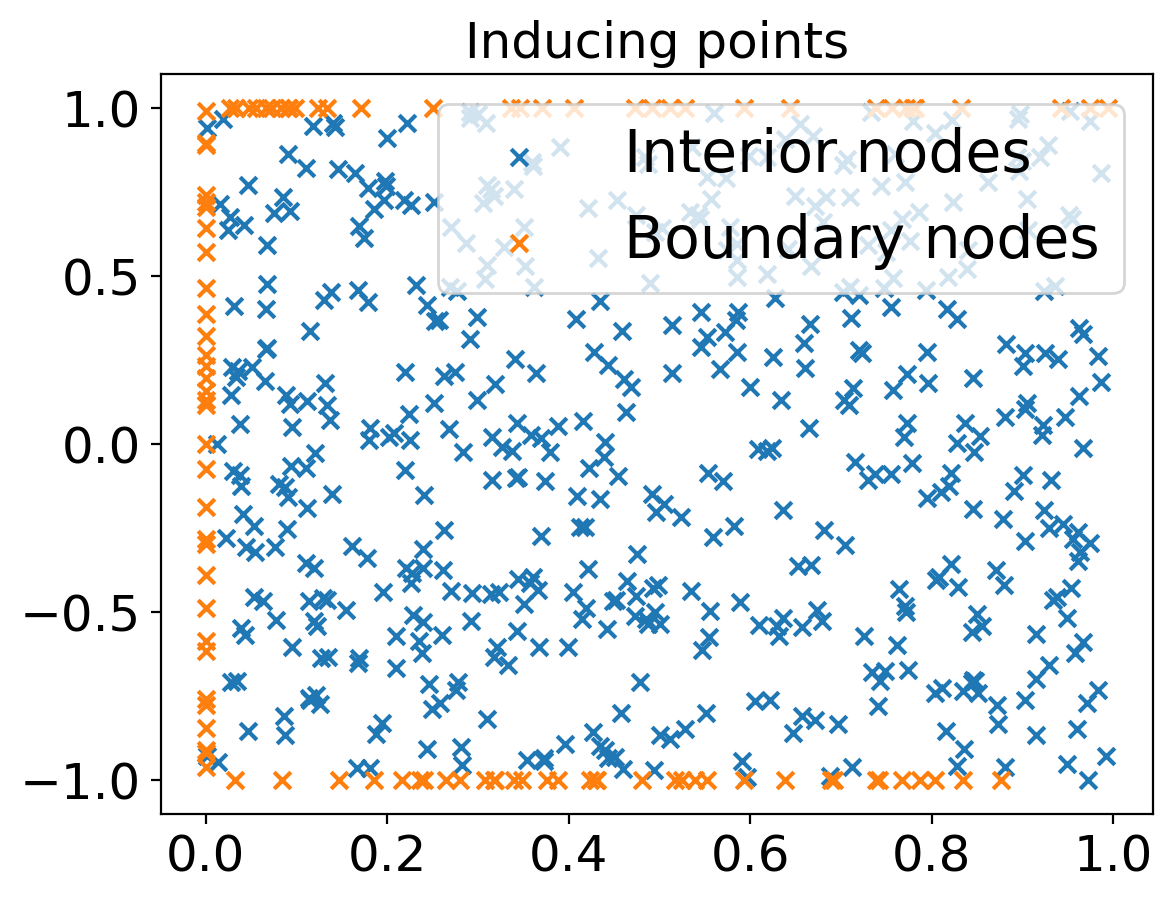}
			\caption{Inducing points.}
			\label{fig:burgers:inducing_points}
		\end{subfigure} \\
		\begin{subfigure}[b]{0.29\textwidth}
			\includegraphics[width=\textwidth]{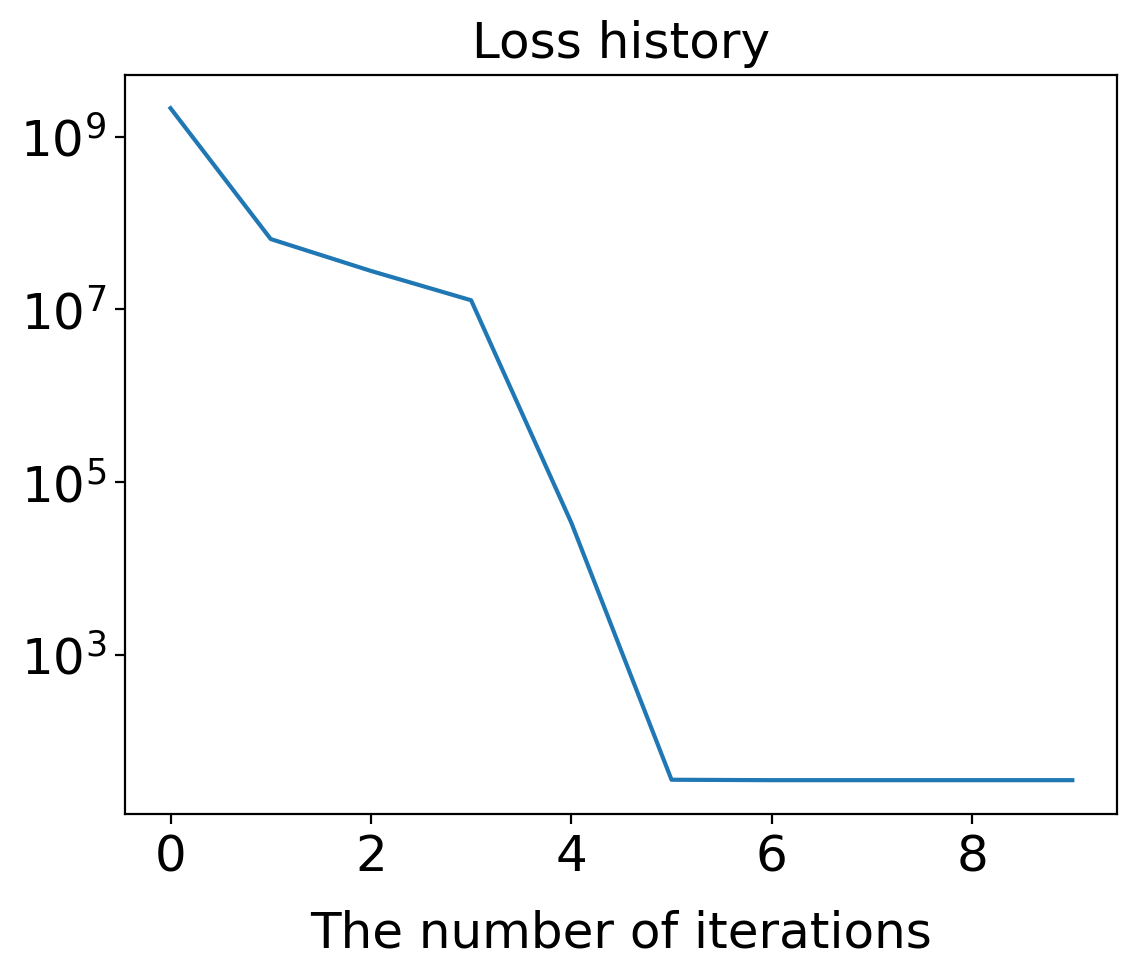}
			\caption{Loss history.}
			\label{fig:burgers:loss_hist}
		\end{subfigure}
		\begin{subfigure}[b]{0.3\textwidth}
			\includegraphics[width=\textwidth]{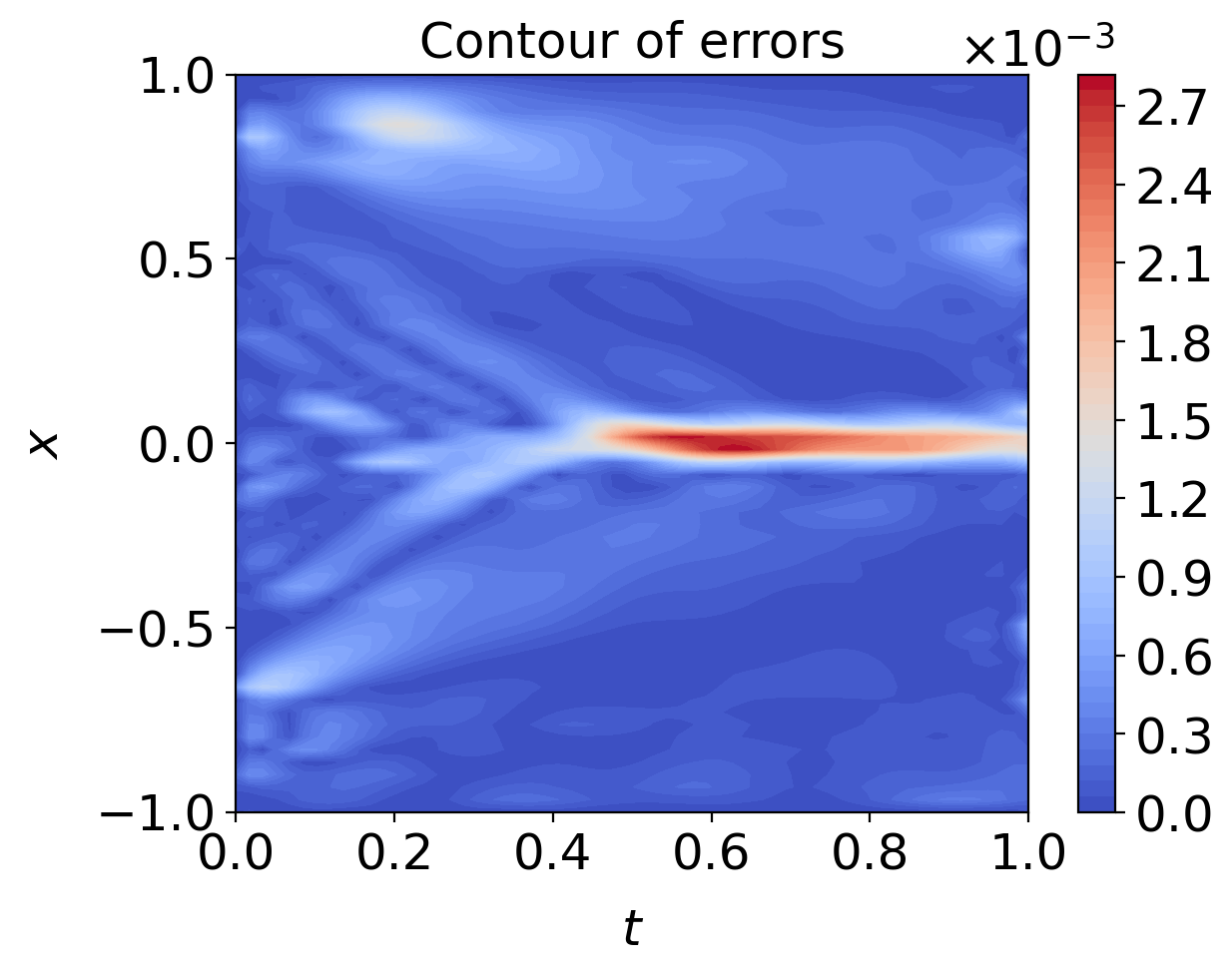}
			\caption{Pointwise errors.}
			\label{fig:burgers:errors}
		\end{subfigure}\\
		\begin{subfigure}[b]{0.3\textwidth}
			\includegraphics[width=\textwidth]{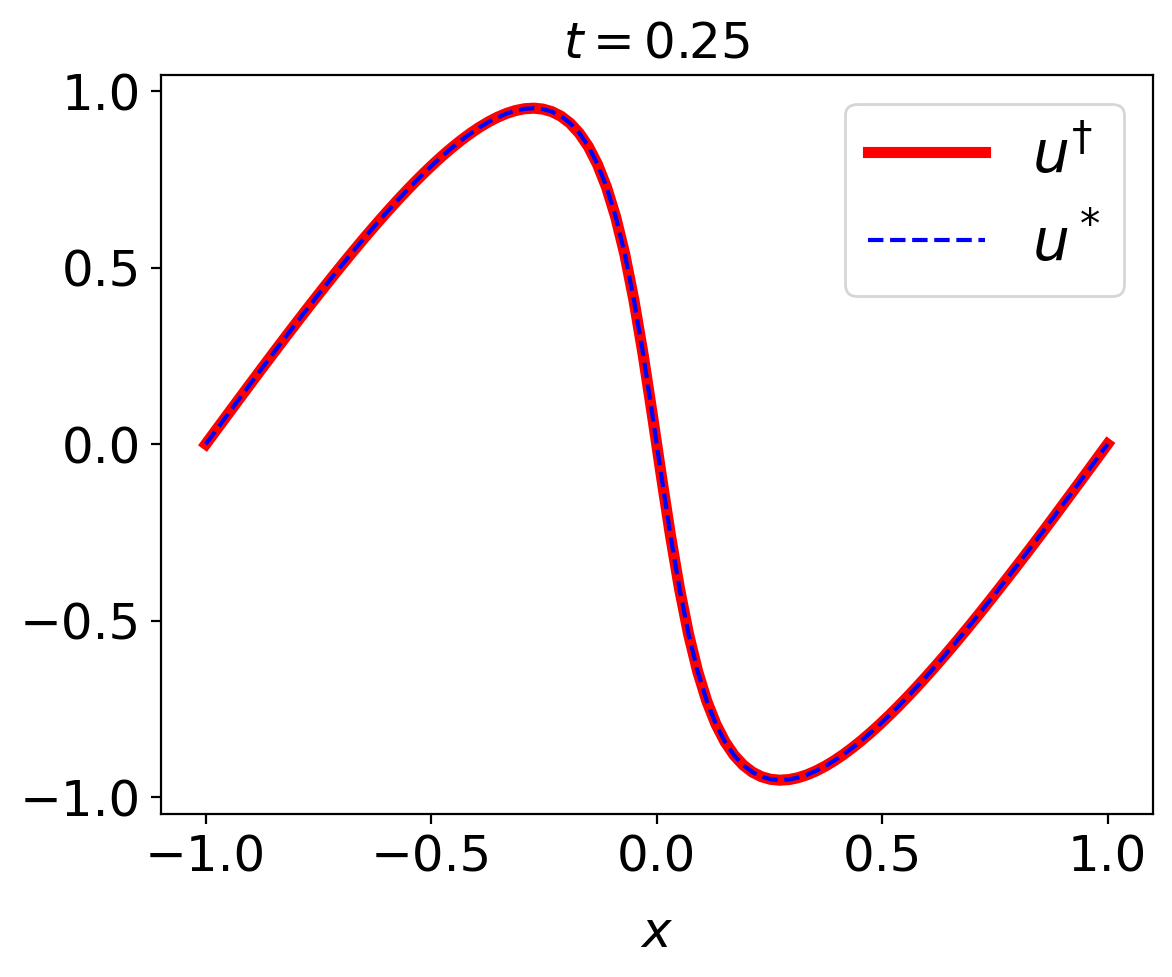}
			\caption{$u^\dagger$ v.s. $u^*$ at $t=0.25$.}
			\label{fig:burgers:0.25.png}
		\end{subfigure}
		\begin{subfigure}[b]{0.3\textwidth}
			\includegraphics[width=\textwidth]{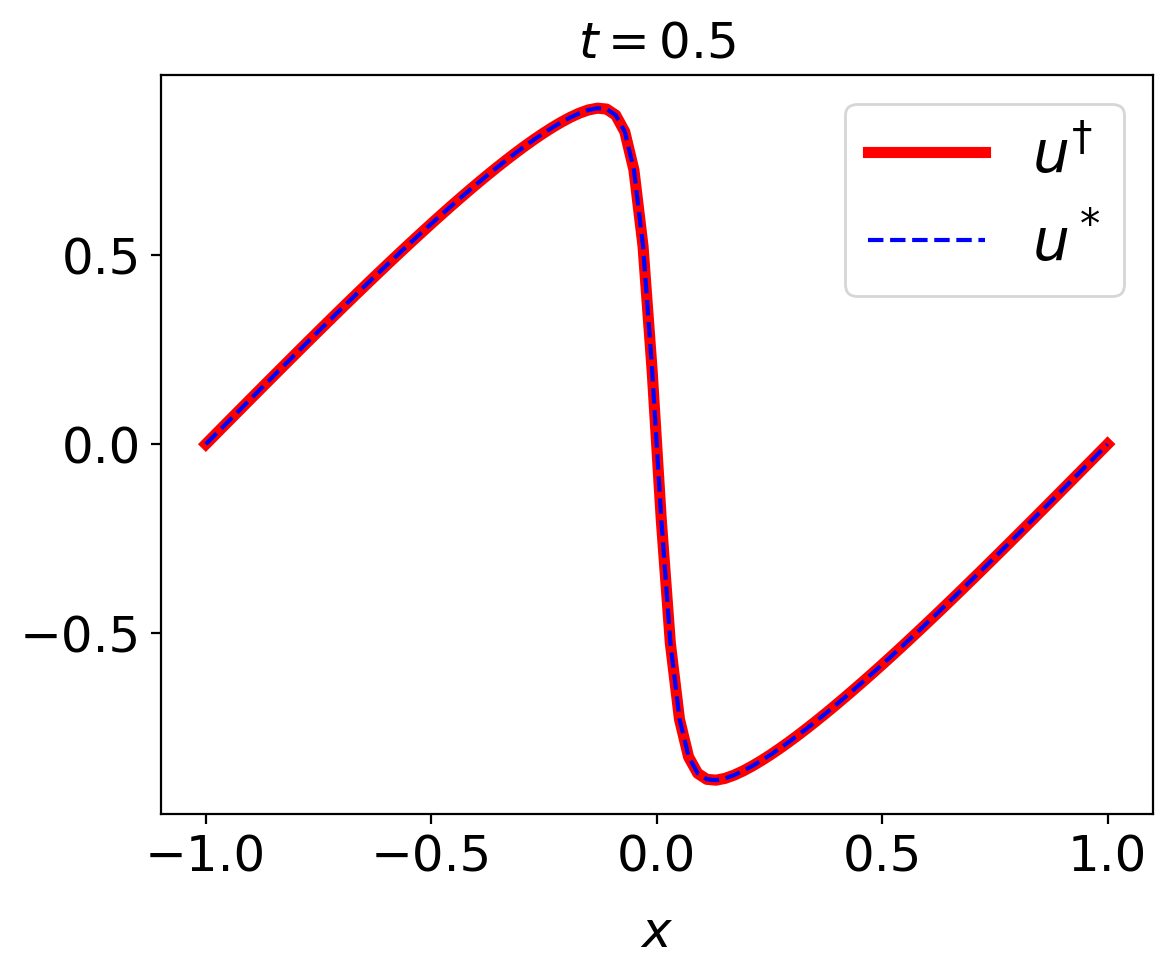}
			\caption{$u^\dagger$ v.s. $u^*$ at $t=0.5$.}
			\label{fig:burgers:0.5.png}
		\end{subfigure}\begin{subfigure}[b]{0.3\textwidth}
		\includegraphics[width=\textwidth]{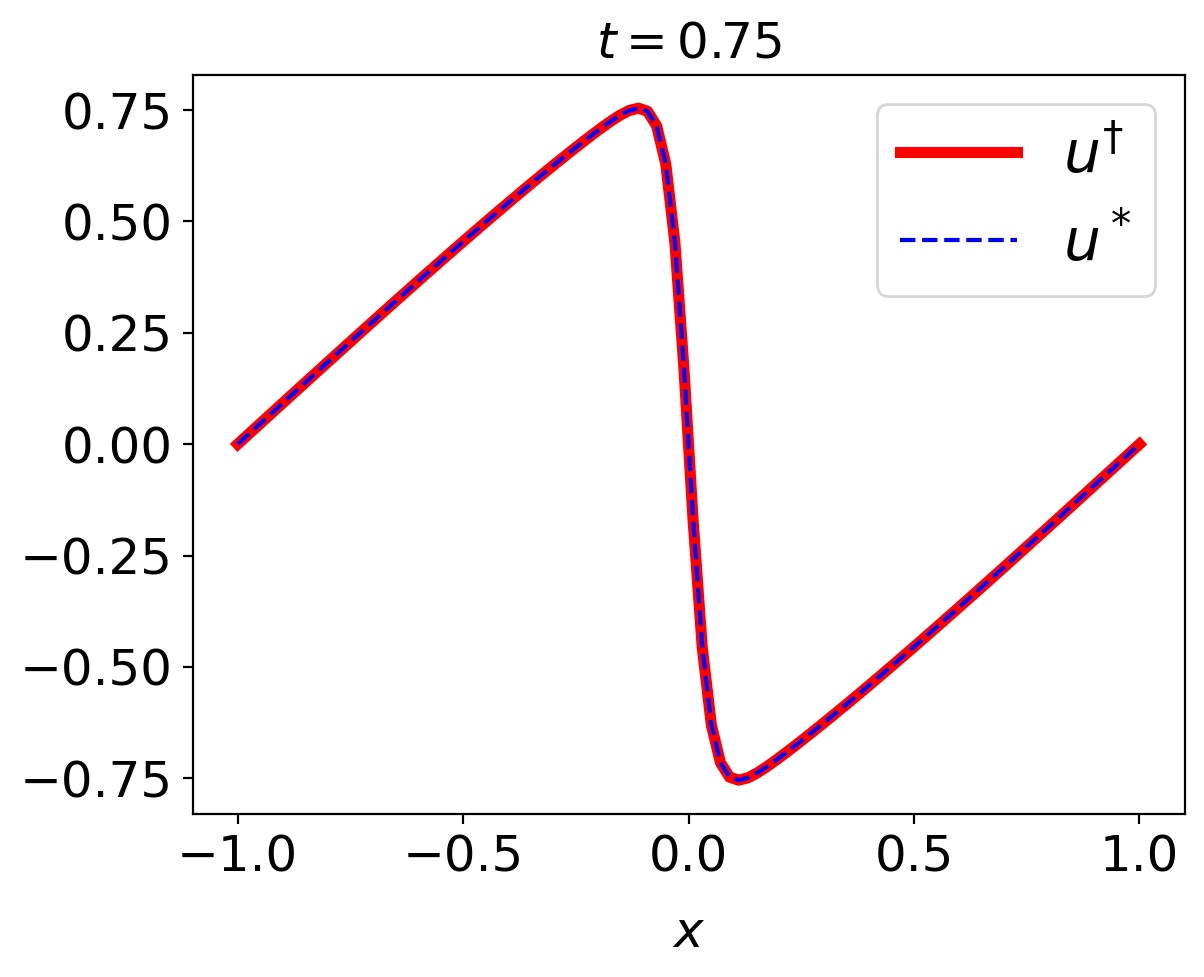}
		\caption{$u^\dagger$ v.s. $u^*$ at $t=0.75$.}
		\label{fig:burgers:0.75.png}
	\end{subfigure}
	\caption{Numerical results for the Burger's  equation: (a) a set of sample points; (b) a collection of inducing points; (c) the convergence history of the Gauss--Newton iterations; (d) contours of pointwise errors; (e)-(g) time slices of the numerical approximation $u^\dagger$ and the true solution $u^*$ at $t = 0.25, 0.5, 0.75$. The parameters $\gamma=\eta=10^{-6}$, $N=2400, N_\Omega=5 N/6$, $M=600$, and $M_\Omega=5 M /6 $.}
	\label{fig:burgers}
\end{figure}

\begin{table}[ht!]
	\centering
	\begin{tabular}{c  c  c  c c}
		\hline
		$N$ & $600$ &  $1200$ &  $2400$ & $4800$ \\
		\hline
		GP, $L^\infty$ error & $6.29\times 10^{-1}$ & $4.66\times 10^{-2}$  &  $4.92\times 10^{-3}$ & $4.60\times 10^{-4}$\\
		SGP($M=300$), $L^\infty$ error &  $6.80\times 10^{-1}$  &  $7.10\times 10^{-2}$ & $1.72\times 10^{-2}$  &  $1.15\times 10^{-2}$ \\
		SGP($M=600$), $L^\infty$ error &  $7.15\times 10^{-1}$ &  $7.56\times 10^{-2}$ & $4.60\times 10^{-3}$ & $1.30\times 10^{-3}$\\
		SGP($M=1200$), $L^\infty$ error &  N/A &  $5.83\times 10^{-2}$ & $4.24\times 10^{-3}$ & $1.13\times 10^{-3}$ \\
		SGP($M=2400$), $L^\infty$ error &  N/A & N/A  & $3.53\times 10^{-3}$  & $4.47\times 10^{-4}$ \\
		\hline
	\end{tabular}
	\caption{Burger's equation: $L^\infty$ errors of the SGP method and the GP algorithm for different numbers of sample points and inducing points. The N/A means not available. We set $\gamma=\eta=10^{-6}$ for $N\leq 1200$ and choose $\gamma=\eta=10^{-8}$ for $N= 2400$.}
	\label{tb:burgers:error_comp}
\end{table}

\subsection{A Nonlinear Parabolic Equation}
\label{subsecParabo}
We consider the numerical solution of the parabolic equation 
\begin{align*}
\begin{cases}
\partial_t u - \partial_x^2u + \frac{1}{2}|\partial_x u|^2 + u + x\partial u_x = f, \forall (t, x)\in (0, 1]\times (0, 3/2),\\
u(0, x)= g(x), \forall x\in (0, 3/2),\\
u(t, -1) = h_1(x), 
u(t, 1) = h_2(x), \forall t \in (0, 1).
\end{cases}
\end{align*}
We prescribe the true solution $u(t, x)=(\sin(\pi x)+2\cos(2\pi x))e^{-t}$ and compute $f$, $g$, $h_1$, and $h_2$ accordingly. 
We fix the ratios $N_\Omega/N=M_\Omega/M=6/7$. Meanwhile, we use the anisotropic kernel in \eqref{anisotropickn} with the same lengthscales, i.e., $(\sigma_1, \sigma_2)=(0.3, 0.05)$. We use the technique of eliminating variables to handle the constraints in \eqref{zsys} (see Subsection 3.3 of \cite{chen2021solving}). All the experiments in this example stop once the errors between two successive steps are less than $10^{-5}$ or the iteration numbers exceed $20$. 

Figure \ref{fig:parabolic} plots the numerical results when $N=2800$, $M=700$, and $\gamma=\eta=10^{-10}$. Figures \ref{fig:parabolic:sample_points} and \ref{fig:parabolic:inducing_points} depict the samples and the inducing points. The history of Gauss--Newton iterations is plotted in Figure \ref{fig:parabolic:loss_hist}, which shows that the algorithm converges within $6$ iterations. The absolute values of the pointwise errors between the numerical solution and the true solution are given in Figure \ref{fig:parabolic:errors}.  Table \ref{tb:parabolic:error_comp} records the $L^\infty$ errors of the SGP method and  the GP algorithm in \cite{chen2021solving} for different values of $N$ and $M$, which justifies the efficacy of the SGP method.

\begin{figure}[!hbtbp]
	\centering          
	\begin{subfigure}[b]{0.3\textwidth}
		\includegraphics[width=\textwidth]{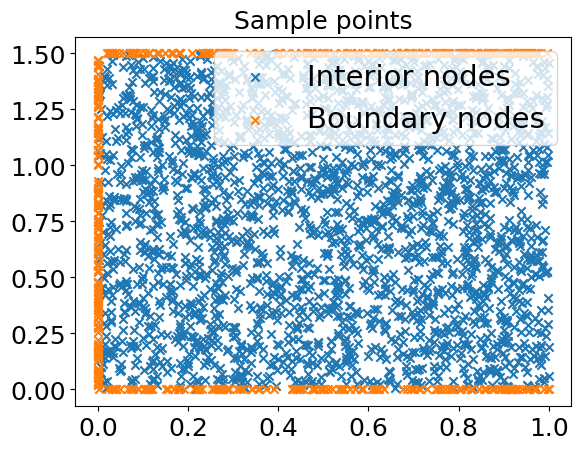}
		\caption{Sample points.}
		\label{fig:parabolic:sample_points}
	\end{subfigure} 
	\begin{subfigure}[b]{0.3\textwidth}
		\includegraphics[width=\textwidth]{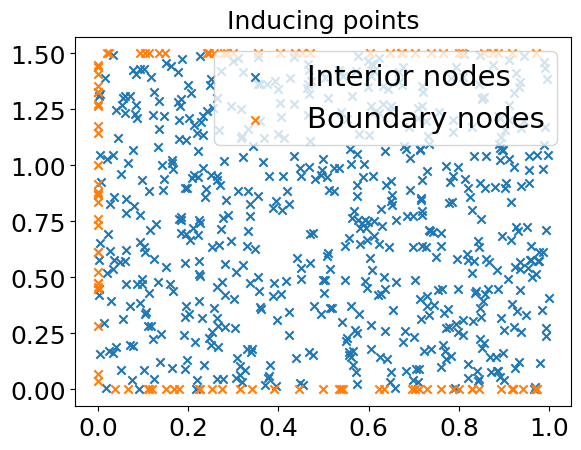}
		\caption{Inducing points.}
		\label{fig:parabolic:inducing_points}
	\end{subfigure} \\
	\begin{subfigure}[b]{0.3\textwidth}
		\includegraphics[width=\textwidth]{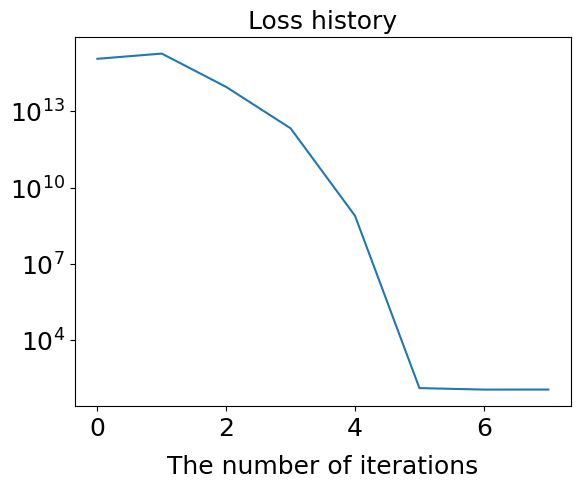}
		\caption{Loss history.}
		\label{fig:parabolic:loss_hist}
	\end{subfigure} 
	\begin{subfigure}[b]{0.3\textwidth}
		\includegraphics[width=\textwidth]{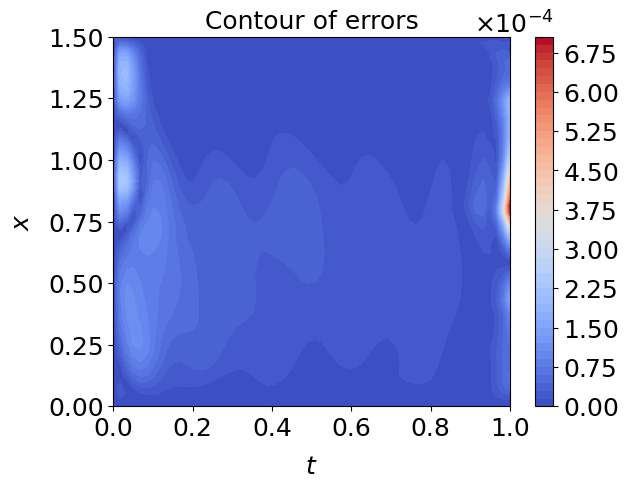}
		\caption{Pointwise errors.}
		\label{fig:parabolic:errors}
	\end{subfigure}
	\caption{Numerical results for the parabolic equation: (a) a set of sample points; (b) a collection of inducing points; (c) the convergence history of the Gauss--Newton iterations; (d) contours of pointwise errors. The regularization parameters $\gamma=\eta=10^{-10}$. $N=2800, N_\Omega=6/7 \times N$, $M=700, M_\Omega=6/7\times M$.}
	\label{fig:parabolic}
\end{figure}

\begin{table}[ht!]
	\centering
	\begin{tabular}{c  c  c  c c}
		\hline
		$N$ & $700$ &  $1400$ &  $2800$ & $5600$ \\
		\hline
		GP, $L^\infty$ error & $2.75\times 10^{-1}$  & $6.42\times 10^{-3}$  &  $5.13\times 10^{-4}$ & $7.25\times 10^{-5}$ \\
		SGP($M=350$), $L^\infty$ error & $3.09\times 10^{-1}$   & $2.92\times 10^{-2}$  &  $2.67\times 10^{-2}$ & $3.63\times 10^{-3}$ \\
		SGP($M=700$), $L^\infty$ error & $2.79\times 10^{-1}$ & $1.50\times 10^{-2}$  & $6.93\times 10^{-4}$  & $1.00\times 10^{-4}$ \\
		SGP($M=1400$), $L^\infty$ error &  N/A & $9.47\times 10^{-3}$  &  $5.22\times 10^{-4}$ & $4.70\times 10^{-5}$\\
		SGP($M=2800$), $L^\infty$ error &  N/A & N/A  & $3.36\times 10^{-4}$  &  $4.32\times 10^{-5}$ \\
		\hline
	\end{tabular}
	\caption{The parabolic PDE: $L^\infty$ errors of the SGP method and the GP algorithm for different numbers of sample points and inducing points. The N/A means not available. We set $\gamma=\eta=10^{-10}$.}
	\label{tb:parabolic:error_comp}
\end{table}

\section{Conclusion and Future Work}
\label{secConclu}
{\color{black}In general, numerical techniques used to solve nonlinear PDEs provide an estimation of the solution within a specific function space. The function space may contain redundant information in terms of PDE solutions if it is not chosen properly. In this paper, we formulate the problem of finding a ``compressed" subspace, in which we seek an approximated solution to a nonlinear PDE with a negligible loss in accuracy compared to the result of using the original function space. In particular, we examine the GP method that approximates functions in the span of base functions formulated by evaluating derivatives of different orders of kernels at random samples. We present a SGP method which seeks solutions for nonlinear PDEs in ``condensed" RKHSs.}
Specifically, the SGP method finds a numerical solution to a PDE in the RKHS associated with a low-rank kernel generated by inducing points. Meanwhile, the approximation solution can be viewed as a maximum \textit{a posteriori} probability estimator of a SGP conditioned a noisy observation of values of linear operators satisfying the PDE at a finite set of samples. {\color{black}Our numerical experiments imply that if the number of inducing points is chosen properly, the SGP algorithm produces a subspace that holds sufficient enough information about the approximated solution. Hence, the SGP method} consumes less computational time than the GP method in \cite{chen2021solving} and achieves comparable accuracy when both methods use uniform samples. We notice that the positions of inducing points greatly influence the performance of our algorithm. Hence, a potential future work is to investigate a better way to put inducing points. Meanwhile, the choice of hyperparameters has a profound impact on the performance of our method. The probabilistic interpretation of the SGP approach in Subsection \ref{subsecProb} provides a way for hyperparameter learning in future work. 

\bibliographystyle{plain}

\clearpage
\appendix

\section{Proof for The Error of The SGP Method}
\label{profMt}
{
	\color{black}
We first give a proof for Lemma \ref{sxprop}. 
\begin{proof}[Proof of Lemma \ref{sxprop}]
We first show that  ${S}_{\overline{\boldsymbol{x}}}^*{S}_{\overline{\boldsymbol{x}}}+\gamma I$ is invertible for any $\gamma>0$. 
Let $u, v\in \mathcal{U}_Q$ and ${S}_{\overline{\boldsymbol{x}}}^*{S}_{\overline{\boldsymbol{x}}}u+\gamma u=v$. Then, according to \eqref{defS} and \eqref{defSxT}, we have
\begin{align}
\label{fmluf}
[\boldsymbol{\psi}, u]^T\mathcal{Q}\boldsymbol{\psi} + \gamma u = v. 
\end{align}
Acting $\boldsymbol{\psi}$ on both sides of \eqref{fmluf}, we get
\begin{align*}
    \mathcal{Q}(\boldsymbol{\psi}, \boldsymbol{\psi})[\boldsymbol{\psi}, u] + \gamma [\boldsymbol{\psi}, u] = [\boldsymbol{\psi}, v].
\end{align*}
Since $\mathcal{Q}(\boldsymbol{\psi}, \boldsymbol{\psi})$ is positive semi-definite, $\mathcal{Q}(\boldsymbol{\psi}, \boldsymbol{\psi})+\gamma I$ is invertible. Thus, 
\begin{align}
\label{psiuform}
    [\boldsymbol{\psi}, u] = (\gamma I + \mathcal{Q}(\boldsymbol{\psi}, \boldsymbol{\psi}))^{-1} [\boldsymbol{\psi}, v].
\end{align}
We get from \eqref{fmluf} and \eqref{psiuform} that
\begin{align*}
\begin{split}
u = \frac{1}{\gamma}v -\frac{1}{\gamma}[\boldsymbol{\psi}, v]^T(\gamma I + \mathcal{Q}(\boldsymbol{\psi}, \boldsymbol{\psi}))^{-1}\mathcal{Q}\boldsymbol{\psi},
\end{split}
\end{align*}
which implies that ${S}_{\overline{\boldsymbol{x}}}^*{S}_{\overline{\boldsymbol{x}}}+\gamma I$ is bijective. Furthermore, 
\begin{align}
\label{uformula}
\begin{split}
({S}_{\overline{\boldsymbol{x}}}^*{S}_{\overline{\boldsymbol{x}}}+\gamma I)^{-1}v = \frac{1}{\gamma}v -\frac{1}{\gamma}[\boldsymbol{\psi}, v]^T(\gamma I + \mathcal{Q}(\boldsymbol{\psi}, \boldsymbol{\psi}))^{-1}\mathcal{Q}\boldsymbol{\psi}, \forall v\in \mathcal{U}_Q. 
\end{split}
\end{align}
Similarly, one can show that ${S}_{\overline{\boldsymbol{x}}}{S}_{\overline{\boldsymbol{x}}}^*+\gamma I$ is bijective. 

To prove \eqref{transidt}, we see that
\begin{align*}
\begin{split}
({S}_{\overline{\boldsymbol{x}}}^*{S}_{\overline{\boldsymbol{x}}}+\gamma I){S}_{\overline{\boldsymbol{x}}}^*({S}_{\overline{\boldsymbol{x}}}{S}_{\overline{\boldsymbol{x}}}^*+\gamma I)^{-1}=&({S}_{\overline{\boldsymbol{x}}}^*{S}_{\overline{\boldsymbol{x}}}{S}_{\overline{\boldsymbol{x}}}^*+\gamma {S}_{\overline{\boldsymbol{x}}}^*)({S}_{\overline{\boldsymbol{x}}}{S}_{\overline{\boldsymbol{x}}}^*+\gamma I)^{-1}\\
=&{S}_{\overline{\boldsymbol{x}}}^*({S}_{\overline{\boldsymbol{x}}}{S}_{\overline{\boldsymbol{x}}}^*+\gamma I)({S}_{\overline{\boldsymbol{x}}}{S}_{\overline{\boldsymbol{x}}}^*+\gamma I)^{-1}={S}_{\overline{\boldsymbol{x}}}^*,
\end{split}
\end{align*}
which yields \eqref{transidt}. 

For any $c\in \ell^2(\overline{\boldsymbol{x}})$, using \eqref{uformula} and \eqref{defSxT}, we get
\begin{align*}
\begin{split}
({S}_{\overline{\boldsymbol{x}}}^*{S}_{\overline{\boldsymbol{x}}}+\gamma I)^{-1}{S}_{\overline{\boldsymbol{x}}}^*c =& \frac{1}{\gamma}c^T\mathcal{Q}\boldsymbol{\psi}-\frac{1}{\gamma}[\boldsymbol{\psi}, c^T\mathcal{Q}\boldsymbol{\psi}](\gamma I + \mathcal{Q}(\boldsymbol{\psi}, ,\boldsymbol{\psi}))^{-1}\mathcal{Q}\boldsymbol{\psi}\\
=& \frac{1}{\gamma}c^T\mathcal{Q}\boldsymbol{\psi}-\frac{1}{\gamma}c^T\mathcal{Q}(\boldsymbol{\psi}, \boldsymbol{\psi})(\gamma I + \mathcal{Q}(\boldsymbol{\psi}, \boldsymbol{\psi}))^{-1}\mathcal{Q}\boldsymbol{\psi}\\
=& \frac{1}{\gamma}c^T\mathcal{Q}\boldsymbol{\psi}-\frac{1}{\gamma}c^T(\mathcal{Q}(\boldsymbol{\psi}, \boldsymbol{\psi})+\gamma I - \gamma I)(\gamma I + \mathcal{Q}(\boldsymbol{\psi}, \boldsymbol{\psi}))^{-1}\mathcal{Q}\boldsymbol{\psi}\\
=& c^T(\gamma I + \mathcal{Q}(\boldsymbol{\psi}, \boldsymbol{\psi}))^{-1}\mathcal{Q}\boldsymbol{\psi},
\end{split}
\end{align*}
which concludes \eqref{eqre0}. 

For any $c\in \ell^2(\overline{\boldsymbol{x}})$, we use \eqref{eqre0} and obtain
\begin{align}
\label{eqImsc}
\begin{split}
(I - {S}_{\overline{\boldsymbol{x}}}({S}_{\overline{\boldsymbol{x}}}^*{S}_{\overline{\boldsymbol{x}}}+\gamma I)^{-1}{S}_{\overline{\boldsymbol{x}}}^*)c =& c - {S}_{\overline{\boldsymbol{x}}}((\mathcal{Q}\boldsymbol{\psi})^T(\gamma I + \mathcal{Q}(\boldsymbol{\psi}, \boldsymbol{\psi}))^{-1}c)\\
=& c - \mathcal{Q}(\boldsymbol{\psi}, \boldsymbol{\psi})(\gamma I + \mathcal{Q}(\boldsymbol{\psi}, \boldsymbol{\psi}))^{-1}c\\
=& c - (\gamma I + \mathcal{Q}(\boldsymbol{\psi}, \boldsymbol{\psi})-\gamma I)(\gamma I + \mathcal{Q}(\boldsymbol{\psi}, \boldsymbol{\psi}))^{-1}c\\
=& \gamma (\gamma I + \mathcal{Q}(\boldsymbol{\psi}, \boldsymbol{\psi}))^{-1}c. 
\end{split}
\end{align}
Since \eqref{eqImsc} holds for any $c\in \ell^2(\overline{\boldsymbol{x}})$, we conclude \eqref{bdImus}. 
\end{proof}

The next lemma is similar to Lemma \ref{sxprop}, in which we establish the properties of a sampling operator leveraging the information of all the samples. 
	\begin{lemma}
		\label{lmsopt}
		Define the sampling operator $\widehat{S}_{\overline{\boldsymbol{x}}}:\mathcal{U}\mapsto \ell^2(\overline{\boldsymbol{x}})$ by
		\begin{align*}
			\widehat{S}_{\overline{\boldsymbol{x}}}=[\boldsymbol{\psi}, u]. 
		\end{align*}
		Denote $\widehat{S}_{\overline{\boldsymbol{x}}}^*$ as the adjoint of $\widehat{S}_{\overline{\boldsymbol{x}}}$ {\color{black}in the sense that for any $c\in \ell^2(\overline{\boldsymbol{x}})$, $\langle \widehat{S}^*_{\overline{\boldsymbol{x}}}c, u\rangle = \langle \widehat{S}_{\overline{\boldsymbol{x}}}u, c\rangle_{\ell^2(\overline{\boldsymbol{x}})}$.} Then, 
		\begin{align}
			\label{defShatxT}
			\widehat{S}_{\overline{\boldsymbol{x}}}^*c = c^T\mathcal{K}\boldsymbol{\psi}, \forall c \in \ell^2(\overline{\boldsymbol{x}}). 
		\end{align}
		Besides, 
		\begin{align}
			\label{transidt1}
			\widehat{S}_{\overline{\boldsymbol{x}}}^*(\widehat{S}_{\overline{\boldsymbol{x}}}\widehat{S}_{\overline{\boldsymbol{x}}}^*+\gamma I)^{-1} = (\widehat{S}_{\overline{\boldsymbol{x}}}^*\widehat{S}_{\overline{\boldsymbol{x}}}+\gamma I)^{-1}\widehat{S}_{\overline{\boldsymbol{x}}}^*,
		\end{align}
		and 
		\begin{align}
			\label{eqre1}
			(\widehat{S}_{\overline{\boldsymbol{x}}}^*\widehat{S}_{\overline{\boldsymbol{x}}}+\gamma I)^{-1}\widehat{S}_{\overline{\boldsymbol{x}}}^*c = (\mathcal{K}\boldsymbol{\psi})^T(\gamma I + \mathcal{K}(\boldsymbol{\psi}, \boldsymbol{\psi}))^{-1}c, \forall c\in \ell^2(\overline{\boldsymbol{x}}).
		\end{align}
	\end{lemma}

	\begin{proof}
		
		Since $\widehat{S}_{\overline{\boldsymbol{x}}}^*$ is the adjoint of $\widehat{S}_{\overline{\boldsymbol{x}}}$, for each $c\in \ell^2(\overline{\boldsymbol{x}})$, we have
		\begin{align*}
			\langle \widehat{S}^*_{\overline{\boldsymbol{x}}}c, u\rangle = \langle \widehat{S}_{\overline{\boldsymbol{x}}}u, c\rangle_{\ell^2(\overline{\boldsymbol{x}})} = c^T[\boldsymbol{\psi}, u] = \langle c^T\mathcal{K}(\boldsymbol{\psi}), u\rangle, \forall u\in \mathcal{U}. 
		\end{align*}
		Thus, we conclude \eqref{defShatxT}. 
		Meanwhile, for $\gamma>0$, by the same arguments as the proof of Lemma \ref{sxprop}, $\widehat{S}_{\overline{\boldsymbol{x}}}$ admits similar properties to those of $S_{\overline{\boldsymbol{x}}}$, i.e., \eqref{transidt1} and \eqref{eqre1} holds. 
	\end{proof}

Next, we give arguments for our main result, Theorem \ref{error_analy}. 
}
\begin{proof}[Proof of Theorem \ref{error_analy}]
	{\color{black}Let $\widehat{S}_{\overline{\boldsymbol{x}}}$ and $\widehat{S}_{\overline{\boldsymbol{x}}}^*$ be defined as in Lemma \ref{lmsopt}. To estimate the norm of $u^\dagger - u^*$, we introduce  intermediate functions,
		\begin{align}
			v:=&(\widehat{S}_{\overline{\boldsymbol{x}}}^*\widehat{S}_{\overline{\boldsymbol{x}}}+\gamma I)^{-1}\widehat{S}_{\overline{\boldsymbol{x}}}^*\boldsymbol{z}^\dagger, \label{defvft}\\
			\widehat{u}:=&(\widehat{S}_{\overline{\boldsymbol{x}}}^*\widehat{S}_{\overline{\boldsymbol{x}}}+\gamma I)^{-1}\widehat{S}_{\overline{\boldsymbol{x}}}^*\widehat{S}_{\overline{\boldsymbol{x}}}u^*. \label{defuft}
		\end{align}
		and 
		\begin{align}
			\label{defubar}
			\overline{u} := (\mathcal{K}\boldsymbol{\psi})^T\mathcal{K}(\boldsymbol{\psi}, \boldsymbol{\psi})^{-1}[\boldsymbol{\psi}, u^*]. 
		\end{align}
		Next, we use the triangle inequality and obtain
		\begin{align}
			\label{bduusttag}
			\|u^\dagger - u^*\|_{\mathcal{U}} \leq \|u^\dagger - v\|_{\mathcal{U}} + \|v - \widehat{u}|_{\mathcal{U}} + \|\widehat{u} - \overline{u}\|_{\mathcal{U}}+\|\overline{u} - u^*\|_{\mathcal{U}}.
		\end{align}
  The accuracy of the approximation of $\mathcal{Q}$ to $\mathcal{K}$ is measured by $\|u^\dagger-v\|_{\mathcal{U}}$, while $\|v-\widehat{u}\|_{\mathcal{U}}$ represents the error between the values of linear operators of $u^*$ and $\boldsymbol{z}^\dagger$. The effect of the nugget term $\gamma I$ is captured by $\|\widehat{u}-\overline{u}\|_{\mathcal{U}}$, and the extent to which the span of $\mathcal{K}\boldsymbol{\psi}$ can approximate the true solution $u^*$ is reflected in $\|\overline{u}-u^*\|_{\mathcal{U}}$.
		Next, we split the arguments for bounding terms at the right-hand side of \eqref{bduusttag} into different steps.

		\textbf{Step 1.} Here, we give {\color{black}an} upper bound for $\|u^\dagger - v\|_{\mathcal{U}}$, where $v$ is defined in \eqref{defvft}.}	Given $\gamma>0$, let $u^\dagger$ be as in \eqref{uoptda} and {\color{black}let} $\boldsymbol{z}^\dagger$ be a solution to \eqref{zsys}.  Then, we have
	\begin{align}
		\label{udaggerfsyss}
		\begin{cases}
			{S}_{\overline{\boldsymbol{x}}}^*{S}_{\overline{\boldsymbol{x}}}u^\dagger + \gamma u^\dagger = {S}_{\overline{\boldsymbol{x}}}^*\boldsymbol{z}^\dagger,\\
			\widehat{S}_{\overline{\boldsymbol{x}}}^*\widehat{S}_{\overline{\boldsymbol{x}}}v+\gamma v = \widehat{S}_{\overline{\boldsymbol{x}}}^*\boldsymbol{z}^\dagger.
		\end{cases}
	\end{align}
	which yields
	\begin{align}
		\label{ufrepre}
		\begin{cases}
			{S}_{\overline{\boldsymbol{x}}}u^\dagger=({S}_{\overline{\boldsymbol{x}}}{S}_{\overline{\boldsymbol{x}}}^*+\gamma I)^{-1}{S}_{\overline{\boldsymbol{x}}}{S}_{\overline{\boldsymbol{x}}}^*\boldsymbol{z}^\dagger,\\
			\widehat{S}_{\overline{\boldsymbol{x}}}v=(\widehat{S}_{\overline{\boldsymbol{x}}}\widehat{S}_{\overline{\boldsymbol{x}}}^*+\gamma I)^{-1}\widehat{S}_{\overline{\boldsymbol{x}}}\widehat{S}_{\overline{\boldsymbol{x}}}^*\boldsymbol{z}^\dagger.
		\end{cases}
	\end{align}
	Taking the difference of the two equations in \eqref{udaggerfsyss} and using \eqref{ufrepre}, we get
	\begin{align}
		\label{bdudaggerf}
		\gamma (u^\dagger - v) =& {S}_{\overline{\boldsymbol{x}}}^*(\boldsymbol{z}^\dagger-{S}_{\overline{\boldsymbol{x}}}u^\dagger)-\widehat{S}_{\overline{\boldsymbol{x}}}^*(\boldsymbol{z}^\dagger - \widehat{S}_{\overline{\boldsymbol{x}}}v)\nonumber\\
		=&\gamma {S}_{\overline{\boldsymbol{x}}}^*({S}_{\overline{\boldsymbol{x}}}{S}_{\overline{\boldsymbol{x}}}^*+\gamma I)^{-1}\boldsymbol{z}^\dagger-\gamma \widehat{S}_{\overline{\boldsymbol{x}}}^*(\widehat{S}_{\overline{\boldsymbol{x}}}\widehat{S}_{\overline{\boldsymbol{x}}}^*+\gamma I)^{-1}\boldsymbol{z}^\dagger.
	\end{align}
	By \eqref{transidt}, \eqref{eqre0}, \eqref{transidt1},  \eqref{eqre1}, and \eqref{bdudaggerf}, we have
	\begin{align*}
		\begin{split}
			u^\dagger - v =& ({S}_{\overline{\boldsymbol{x}}}^*{S}_{\overline{\boldsymbol{x}}}+\gamma I)^{-1}{S}_{\overline{\boldsymbol{x}}}^*\boldsymbol{z}^\dagger-(\widehat{S}_{\overline{\boldsymbol{x}}}^*\widehat{S}_{\overline{\boldsymbol{x}}}+\gamma I)^{-1}\widehat{S}_{\overline{\boldsymbol{x}}}^*\boldsymbol{z}^\dagger\\
			=& (\mathcal{Q}\boldsymbol{\psi})^T(\gamma I + \mathcal{Q}(\boldsymbol{\psi}, \boldsymbol{\psi}))^{-1}\boldsymbol{z}^\dagger - (\mathcal{K}\boldsymbol{\psi})^T(\gamma I + \mathcal{K}(\boldsymbol{\psi}, \boldsymbol{\psi}))^{-1}\boldsymbol{z}^\dagger.
		\end{split}
	\end{align*}
	{\color{black}Thus, we obtain}
	\begin{align}
		\label{bdudft1}
		\|u^\dagger - v\|^2_{\mathcal{U}} =& (\boldsymbol{z}^\dagger)^T(\gamma I + \mathcal{Q}(\boldsymbol{\psi}, \boldsymbol{\psi}))^{-1}\boldsymbol{z}^\dagger-(\boldsymbol{z}^\dagger)^T(\gamma I + \mathcal{K}(\boldsymbol{\psi}, \boldsymbol{\psi}))^{-1}\boldsymbol{z}^\dagger\nonumber\\
		&+2\gamma(\boldsymbol{z}^\dagger)^T(\gamma I + \mathcal{Q}(\boldsymbol{\psi}, \boldsymbol{\psi}))^{-1}(\gamma I + \mathcal{K}(\boldsymbol{\psi}, \boldsymbol{\psi}))^{-1}\boldsymbol{z}^\dagger\nonumber\\
		&-\gamma(\boldsymbol{z}^\dagger)^T(\gamma I + \mathcal{Q}(\boldsymbol{\psi}, \boldsymbol{\psi}))^{-2}\boldsymbol{z}^\dagger-\gamma(\boldsymbol{z}^\dagger)^T(\gamma I + \mathcal{K}(\boldsymbol{\psi}, \boldsymbol{\psi}))^{-2}\boldsymbol{z}^\dagger \nonumber\\
		\leq & 3\|(\gamma I + \mathcal{Q}(\boldsymbol{\psi}, \boldsymbol{\psi}))^{-1}-(\gamma I + \mathcal{K}(\boldsymbol{\psi}, \boldsymbol{\psi}))^{-1}\||\boldsymbol{z}^\dagger|^2,
	\end{align}
	where we use $\|(\gamma I + \mathcal{Q}(\boldsymbol{\psi}, \boldsymbol{\psi}))^{-1}\|\leq 1/\gamma$ and $\|(\gamma I + \mathcal{K}(\boldsymbol{\psi}, \boldsymbol{\psi}))^{-1}\|\leq 1/\gamma$ in the last inequality. 
	
	{\color{black} \textbf{Step 2}. Next, we bound $v-\widehat{u}$, where $\widehat{u}$ is given in \eqref{defuft}. Using \eqref{eqre1} and the definitions \eqref{defvft}-\eqref{defuft}, we get
		\begin{align}
			v =& (\mathcal{K}\boldsymbol{\psi})^T(\gamma I + \mathcal{K}(\boldsymbol{\psi}, \boldsymbol{\psi}))^{-1}\boldsymbol{z}^\dagger, \label{expuvexv}\\
			\widehat{u} = & (\mathcal{K}\boldsymbol{\psi})^T(\gamma I + \mathcal{K}(\boldsymbol{\psi}, \boldsymbol{\psi}))^{-1}[\boldsymbol{\psi}, u^*]. \label{expuvexu}
		\end{align}
		Thus, we have
		\begin{align*}
			\|v - \widehat{u}\|_{\mathcal{U}}^2 =& (\boldsymbol{z}^\dagger-[\boldsymbol{\psi}, u^*])^T(\gamma I + \mathcal{K}(\boldsymbol{\psi}, \boldsymbol{\psi}))^{-1}\mathcal{K}(\boldsymbol{\psi}, \boldsymbol{\psi})(\gamma I + \mathcal{K}(\boldsymbol{\psi}, \boldsymbol{\psi}))^{-1}(\boldsymbol{z}^\dagger - [\boldsymbol{\psi}, u^*])\nonumber\\
			=& (\boldsymbol{z}^\dagger - [\boldsymbol{\psi}, u^*])^T(\gamma I + \mathcal{K}(\boldsymbol{\psi}, \boldsymbol{\psi}))^{-1}(\boldsymbol{z}^\dagger - [\boldsymbol{\psi}, u^*])\nonumber\\
			&- \gamma (\boldsymbol{z}^\dagger-[\boldsymbol{\psi}, u^*])^T(\gamma I + \mathcal{K}(\boldsymbol{\psi}, \boldsymbol{\psi}))^{-2}(\boldsymbol{z}^\dagger - [\boldsymbol{\psi}, u^*])\nonumber\\
			\leq & (\boldsymbol{z}^\dagger - [\boldsymbol{\psi}, u^*])^T\mathcal{K}(\boldsymbol{\psi}, \boldsymbol{\psi})^{-1}(\boldsymbol{z}^\dagger - [\boldsymbol{\psi}, u^*]), 
		\end{align*}
		where the last inequality results from the fact that $\mathcal{K}(\boldsymbol{\psi}, \boldsymbol{\psi})^{-1} - (\gamma I + \mathcal{K}(\boldsymbol{\psi}, \boldsymbol{\psi}))^{-1}$ and $(\gamma I + \mathcal{K}(\boldsymbol{\psi}, \boldsymbol{\psi}))^{-2}$ are positive definite. 
		Hence, we obtain
		\begin{align}
			\label{bdlzlsusf}
			\begin{split}
				\|{\color{black}v-\widehat{u}}\|_{\mathcal{U}}\leq \sqrt{(\boldsymbol{z}^\dagger - [\boldsymbol{\psi}, u^*])^T\mathcal{K}(\boldsymbol{\psi}, \boldsymbol{\psi})^{-1}(\boldsymbol{z}^\dagger-[\boldsymbol{\psi}, u^*])}.
			\end{split}
		\end{align}
	}
	
	{\color{black}\textbf{Step 3}. Then, we estimate the norm of  $\widehat{u} - \overline{u}$. 
		Using \eqref{expuvexu}, we get
		\begin{align}
			\overline{u} - \widehat{u} =& (\mathcal{K}\boldsymbol{\psi})^T(\mathcal{K}(\boldsymbol{\psi}, \boldsymbol{\psi})^{-1}- (\gamma I + \mathcal{K}(\boldsymbol{\psi}, \boldsymbol{\psi}))^{-1})[\boldsymbol{\psi}, u^*]\nonumber \\
			=& \gamma (\mathcal{K}\boldsymbol{\psi})^T\mathcal{K}(\boldsymbol{\psi}, \boldsymbol{\psi})^{-1}(\gamma I + \mathcal{K}(\boldsymbol{\psi}, \boldsymbol{\psi}))^{-1}[\boldsymbol{\psi}, u^*] \label{umubar}
		\end{align}
		where we use $\mathcal{K}(\boldsymbol{\psi}, \boldsymbol{\psi})^{-1} - (\gamma I + \mathcal{K}(\boldsymbol{\psi}, \boldsymbol{\psi}))^{-1} = \gamma \mathcal{K}(\boldsymbol{\psi}, \boldsymbol{\psi})^{-1}(\gamma I + \mathcal{K}(\boldsymbol{\psi}, \boldsymbol{\psi}))^{-1}$ in the last equality. 
		Thus, using \eqref{umubar} and the fact that $\mathcal{K}(\boldsymbol{\psi}, \boldsymbol{\psi})^{-3} - (\gamma I + \mathcal{K}(\boldsymbol{\psi}, \boldsymbol{\psi}))^{-1}\mathcal{K}(\boldsymbol{\psi}, \boldsymbol{\psi})^{-1}(\gamma I + \mathcal{K}(\boldsymbol{\psi}, \boldsymbol{\psi}))^{-1}$ is positive definite, we have
		\begin{align}
			\|\widehat{u} - \overline{u}\|^2_{\mathcal{U}} =& \gamma^2 [\boldsymbol{\psi}, u^*]^T(\gamma I + \mathcal{K}(\boldsymbol{\psi}, \boldsymbol{\psi}))^{-1}\mathcal{K}(\boldsymbol{\psi}, \boldsymbol{\psi})^{-1}(\gamma I + \mathcal{K}(\boldsymbol{\psi}, \boldsymbol{\psi}))^{-1}[\boldsymbol{\psi}, u^*]\nonumber\\
			\leq & \gamma^2 [\boldsymbol{\psi}, u^*]^T\mathcal{K}(\boldsymbol{\psi}, \boldsymbol{\psi})^{-3}[\boldsymbol{\psi}, u^*]. \label{fteq}
		\end{align}

		\textbf{Step 4.} In the last step, we bound $\overline{u} - u^*$, where $\overline{u}$ is given in \eqref{defubar}. We observe that $\overline{u}$ is the projection of $u^*$ onto the span of $\mathcal{K}\boldsymbol{\psi}$ \cite[Theorem 12.3]{owhadi2019operator}. Thus, 
		\begin{align}
			\label{inespest}
			\|u^* - \overline{u}\|_{\mathcal{U}} = \inf_{\psi\in \operatorname{span}(\boldsymbol{\psi})}\|u^* -  \mathcal{K}\psi\|_{\mathcal{U}}
		\end{align}
		Therefore, \eqref{bduusttag}, \eqref{bdudft1}, \eqref{bdlzlsusf}, \eqref{fteq}, and \eqref{inespest} imply that \eqref{bderror} holds. 
	}
\end{proof}

\section{Proofs for the Nystr\"{o}m Approximation Error}
\label{prfNSE}
{\color{black}In this section, we give a detailed proof for estimating the upper bound for $\mathcal{E}(\mathcal{H}_a)$}. For simplification, we write $\boldsymbol{\psi}:=(\boldsymbol{\psi}_\Omega, \boldsymbol{\psi}_{\partial \Omega})$, where $\boldsymbol{\psi}_\Omega$ represents the collection of linear operators for interior points and $\boldsymbol{\psi}_{\partial \Omega}$ consists of linear operators for boundary samples. We write $\overline{\boldsymbol{x}}:=(\overline{\boldsymbol{x}}_1, \overline{\boldsymbol{x}}_2)$, where we recall that $\overline{\boldsymbol{x}}$ contains all the sample points,  $\overline{\boldsymbol{x}}_1$ is the set of interior points in $\Omega$, and $\overline{\boldsymbol{x}}_2$ stands for the collection of boundary points.  
We define the operators $L_{\overline{\boldsymbol{x}}_1}$ and $L_{\overline{\boldsymbol{x}}_2}$  such that for any function $v\in \mathcal{U}$, 
\begin{align}
	\label{defLpsi}
	L_{\overline{\boldsymbol{x}}_1}[v]=\frac{1}{N_\Omega}(\mathcal{K}\boldsymbol{\psi}_{\Omega})^T[\boldsymbol{\psi}_{\Omega}, v] \text{ and } L_{\overline{\boldsymbol{x}}_2}[v]=\frac{1}{N_{\partial\Omega}}(\mathcal{K}\boldsymbol{\psi}_{\partial\Omega})^T[\boldsymbol{\psi}_{\partial\Omega}, v],
\end{align}
where $N_\Omega$ is the size of $\overline{\boldsymbol{x}}_1$ and $N_{\partial \Omega}$ is the length of $\overline{\boldsymbol{x}}_2$. The next proposition connects the eigenvalues and the eigenfunctions of $L_{\overline{\boldsymbol{x}}_1}$ and $L_{\overline{\boldsymbol{x}}_2}$ to the eigenvalues and the eigenvectors of $\mathcal{K}(\boldsymbol{\psi}_\Omega, \boldsymbol{\psi}_\Omega)$ and $\mathcal{K}(\boldsymbol{\psi}_{\partial \Omega}, \boldsymbol{\psi}_{\partial \Omega})$. 
\begin{pro}
	\label{kbasics}
	Let $R_1=N_\Omega(D-D_b)$ and $R_2=(N-N_\Omega)D_b$, where $D$ and $D_b$ are given in \eqref{pdegrf}. 
	Let $\{\lambda_j\}_{j=1}^{R_1}$ be the set of eigenvalues of $\mathcal{K}(\boldsymbol{\psi}_\Omega, \boldsymbol{\psi}_\Omega)$, which are ranked in the descending order.  Define the operator $L_{\overline{\boldsymbol{x}}_1}$ as in \eqref{defLpsi}. Then, the eigenvalues of $L_{\overline{\boldsymbol{x}}_1}$ are $\{\lambda_j/N_\Omega\}_{j=1}^{R_1}$. Moreover, let $\varphi_1, \dots, \varphi_{R_1}$ be the corresponding eigenfunctions of $L_{\overline{\boldsymbol{x}}_1}$ that have normalized functional norms, i.e. $\langle \varphi_i, \varphi_j\rangle = \delta_{ij}$, $1\leq i, j\leq R_1$. Let $V=[\boldsymbol{v}_1, \dots,  \boldsymbol{v}_{R_1}]$ be the orthonormal eigenvector matrix of $\mathcal{K}(\boldsymbol{\psi}_\Omega, \boldsymbol{\psi}_\Omega)$. Then, we have
	\begin{align}
		\label{defvj}
		\sqrt{\lambda_j}\boldsymbol{v}_j=[\boldsymbol{\psi}_{\Omega}, \varphi_j], \forall 1\leq j \leq R_1, 
	\end{align}
	\begin{align}
		\label{varphiknrpre}
		\sqrt{\lambda_j}\varphi_j = \boldsymbol{v}_j^T\mathcal{K}\boldsymbol{\psi}_{\Omega}, \forall 1\leq j \leq R_1, 
	\end{align}
	\begin{align}
		\label{knrvarphipre}
		\mathcal{K}\boldsymbol{\psi}_\Omega = V[\sqrt{\lambda_1}\varphi_1,\dots,\sqrt{\lambda_{R_1}}\varphi_{R_1}]^T,  
	\end{align}
	and 
	\begin{align}
		\label{eqphikphi}
		\mathcal{K}(\psi_{\Omega, j}, \psi_{\Omega, j}) = \sum_{i=1}^{R_1}[\psi_{\Omega, j}, \varphi_i]^2, \forall 1\le j\leq R_1, 
	\end{align}
	where $\boldsymbol{\psi}_{\Omega, j}$ is the $j^{\textit{th}}$ component of $\boldsymbol{\psi}_\Omega$. 
	Similarly, let $\{\tau_j\}_{j=1}^{R_2}$ be the set of eigenvalues of $\mathcal{K}(\boldsymbol{\psi}_{\partial \Omega}, \boldsymbol{\psi}_{\partial \Omega})$ sorted in the descending order and let $L_{\overline{\boldsymbol{x}}_2}$ be as in \eqref{defLpsi}. Then, the eigenvalues of $L_{\overline{\boldsymbol{x}}_2}$ are $\{\tau_j/N_{\partial \Omega}\}_{j=1}^{R_2}$. Meanwhile, let $\zeta_1, \dots, \zeta_{R_2}$ be the eigenfunctions of $L_{\overline{\boldsymbol{x}}_2}$ that have normalized functional norms. Let $W=[\boldsymbol{w}_1, \dots,  \boldsymbol{w}_{R_2}]$ be the orthonormal eigenvector matrix of $\mathcal{K}(\boldsymbol{\psi}_{\partial\Omega}, \boldsymbol{\psi}_{\partial \Omega})$. Then, we obtain 
	\begin{align}
		\label{defzetaj}
		\sqrt{\tau_j}\boldsymbol{w}_j=[\boldsymbol{\psi}_{\partial \Omega}, \zeta_j], \forall 1\leq j \leq R_2, 
	\end{align}
	\begin{align}
		\label{varpsiknrpre}
		\sqrt{\tau_j}\zeta_j = \boldsymbol{w}_j^T\mathcal{K}\boldsymbol{\psi}_{\partial\Omega}, \forall 1\leq j \leq R_2, 
	\end{align}
	\begin{align}
		\label{knrvarpsipre}
		\mathcal{K}\boldsymbol{\psi}_{\partial\Omega} = W[\sqrt{\tau_1}\zeta_1,\dots,\sqrt{\tau_{R_2}}\zeta_{R_2}]^T,  
	\end{align}
	and 
	\begin{align}
		\label{eqphikpsi}
		\mathcal{K}(\psi_{\partial \Omega, j}, \psi_{\partial \Omega, j}) = \sum_{i=1}^{R_2}[\psi_{\partial \Omega, j}, \zeta_i]^2, \forall 1\leq j\leq R_2, 
	\end{align}
	where $\boldsymbol{\psi}_{\partial \Omega, j}$ is the $j^{\textit{th}}$ element of $\boldsymbol{\psi}_{\partial \Omega}$. 
\end{pro}
\begin{proof}
	We only give proofs related to $L_{\overline{\boldsymbol{x}}_1}$ since the arguments for $L_{\overline{\boldsymbol{x}}_2}$ are similar. Let $\widetilde{\lambda}_j$ be $j^{\textit{th}}$ eigenvalue of $L_{\overline{\boldsymbol{x}}_1}$ and let $\varphi_j$ be the corresponding eigenfunction with a normalized norm. Then, we have
	\begin{align}
		\label{eigencor}
		L_{\overline{\boldsymbol{x}}_1}[\varphi_j]=\frac{1}{N_\Omega}(\mathcal{K}\boldsymbol{\psi}_{\Omega})^T[\boldsymbol{\psi}_{\Omega}, \varphi_j]=\widetilde{\lambda}_j\varphi_j. 
	\end{align}
	Acting $\boldsymbol{\psi}_{\Omega}$ on both sides of \eqref{eigencor}, we get
	\begin{align}
		\label{kppeig}
		\mathcal{K}(\boldsymbol{\psi}_{\Omega}, \boldsymbol{\psi}_{\Omega})[\boldsymbol{\psi}_{\Omega}, \varphi_j]=N_\Omega\widetilde{\lambda}_j[\boldsymbol{\psi}_\Omega, \varphi_j],
	\end{align}
	which implies that $N_\Omega\widetilde{\lambda}_j$ is the $j^{\textit{th}}$ eigenvalue of $\mathcal{K}(\boldsymbol{\psi}_{\Omega}, \boldsymbol{\psi}_{\Omega})$ and $[\boldsymbol{\psi}_\Omega, \varphi_j]$ is the corresponding eigenvector. Hence, we confirm that if $\{\lambda_j\}_{j=1}^{R_1}$ contains the eigenvalues of $\mathcal{K}(\boldsymbol{\psi}_{\Omega}, \boldsymbol{\psi}_\Omega)$,  $\{\lambda_j/N_\Omega\}_{j=1}^{R_1}$ is the set of the eigenvalues of $L_{\overline{\boldsymbol{x}}_1}$. Furthermore, let $V=[\boldsymbol{v}_1, \dots,  \boldsymbol{v}_{R_1}]$ be the orthonormal eigenvector matrix of $\mathcal{K}(\boldsymbol{\psi}_\Omega, \boldsymbol{\psi}_\Omega)$. Then, \eqref{kppeig} implies that
	\begin{align}
		\label{eigKpsisps}
		\boldsymbol{v}_j =\frac{[\boldsymbol{\psi}_\Omega, \varphi_j]}{|[\boldsymbol{\psi}_\Omega, \varphi_j]|}.
	\end{align}
	From \eqref{eigencor}, we get
	\begin{align}
		\label{kldvp}
		(\mathcal{K}\boldsymbol{\psi}_{\Omega})^T[\boldsymbol{\psi}_\Omega, \varphi_j]=\lambda_j\varphi_j, 
	\end{align}
	which yields
	\begin{align}
		\label{varphidd}
		\langle (\mathcal{K}\boldsymbol{\psi}_{\Omega})^T[\boldsymbol{\psi}_\Omega, \varphi_j], (\mathcal{K}\boldsymbol{\psi}_{\Omega})^T[\boldsymbol{\psi}_\Omega, \varphi_j]\rangle =\lambda_j^2\langle \varphi_j, \varphi_j\rangle. 
	\end{align}
	Since $\varphi_j$ is normalized, we obtain from \eqref{kldvp} and  \eqref{varphidd} that
	\begin{align*}
		\lambda_j|[\boldsymbol{\psi}_\Omega, \varphi_j]|^2=[\boldsymbol{\psi}_\Omega, \varphi_j]^T\mathcal{K}(\boldsymbol{\psi}_{\Omega}, \boldsymbol{\psi}_{\Omega})[\boldsymbol{\psi}_\Omega, \varphi_j]=\lambda_j^2,
	\end{align*}
	which gives
	\begin{align}
		\label{ldj}
		|[\boldsymbol{\psi}_\Omega, \varphi_j]| = \sqrt{\lambda_j}. 
	\end{align}
	Hence, \eqref{eigKpsisps} and \eqref{ldj} imply that $\sqrt{\lambda_j}\boldsymbol{v}_j=[\boldsymbol{\psi}_\Omega, \varphi_j]$, which confirms \eqref{defvj}. 
	Then, the equality \eqref{varphiknrpre}  follows directly from \eqref{kldvp} and \eqref{defvj}. Meanwhile, \eqref{knrvarphipre} results from \eqref{varphiknrpre} and the fact that $V$ is orthonormal. 
	
	Let $\boldsymbol{\psi}_{\Omega, j}$ be the $j^{\textit{th}}$ element of $\boldsymbol{\psi}_\Omega$. Denote by $V_{ji}$ the component of $V$ at the $j^{\textit{th}}$ row and the $i^{\textit{th}}$ column. From \eqref{knrvarphipre}, we get
	\begin{align}
		\label{kphieq}
		\mathcal{K}\boldsymbol{\psi}_{\Omega, j} = \sum_{i=1}^{R_1}V_{ji}\sqrt{\lambda_i}\varphi_i.
	\end{align}
	Thus, we get from \eqref{kphieq} that
	\begin{align*}
		[\boldsymbol{\psi}_{\Omega, j}, \mathcal{K}\boldsymbol{\psi}_{\Omega, j}] = \langle \mathcal{K}\boldsymbol{\psi}_{\Omega, j}, \mathcal{K}\boldsymbol{\psi}_{\Omega, j}\rangle = \sum_{i=1}^{R_1}\lambda_iV_{ji}^2=\sum_{i=1}^{R_1}[\boldsymbol{\psi}_{\Omega, j}, \varphi_i]^2,
	\end{align*}
	where the last equality follows \eqref{defvj}. Thus, we conclude \eqref{eqphikphi}. 
\end{proof}
Using \eqref{knrvarphipre} and \eqref{knrvarpsipre}, $\mathcal{H}_b$ in \eqref{defsetsns} can be rewritten in the basis of the eigenfunctions $\{\varphi_i\}_{i=1}^{{R_1}}$ and $\{\zeta_i\}_{i=1}^{{R_2}}$, i.e.,
\begin{align}
	\label{deHbeig}
	\mathcal{H}_b=\bigg\{f=\sum_{i=1}^{R_1}b_{1,i}\sqrt{\lambda_i}\varphi_i + \sum_{i=1}^{R_2}b_{2,i}\sqrt{\tau_i}\zeta_i, \sum_{i=1}^{R_1}b_{1,i}^2+\sum_{i=1}^{R_2}b_{2,i}^2\leq 1\bigg\}. 
\end{align}
For any $\boldsymbol{r}:=(r_1, r_2)\in [R_1]\times[R_2]$, we define
\begin{align}
	\label{sbuspaces}
	\begin{split}
		\mathcal{H}_b^{\boldsymbol{r}}=&\bigg\{f=\sum_{i=1}^{r_1}b_{1,i}\sqrt{\lambda_i}\varphi_i+\sum_{i=1}^{r_2}b_{2,i}\sqrt{\tau_i}\zeta_i, \sum_{i=1}^{{r}}b_{1,i}^2+b_{2,i}^2\leq 1\bigg\},\\
		\overline{\mathcal{H}}_b^{\boldsymbol{r}}=&\bigg\{f=\sum_{i=1}^{R_1-r_1}b_{1,i}\sqrt{\lambda_{i+r_1}}\varphi_{i+r_1}+\sum_{i=1}^{R_2-r_2}b_{2,i}\sqrt{\tau_{i+r_2}}\zeta_{i+r_2}, \sum_{i=1}^{R_1-r_1}b_{1,i}^2+\sum_{i=1}^{R_2-r_2}b_{2,i}^2\leq 1\bigg\}. 
	\end{split}
\end{align}
Define 
\begin{align}
	\label{defEhar}
	\mathcal{E}(\mathcal{H}_a, \boldsymbol{r})=\max_{h\in \mathcal{H}_b^{\boldsymbol{r}}}\mathcal{E}(h, \mathcal{H}_a),
\end{align}
which is the worst error in approximating any $h\in \mathcal{H}_b^{\boldsymbol{r}}$ by functions in $\mathcal{H}_a$.  The next proposition bounds $\mathcal{E}(\mathcal{H}_a)$ by $\mathcal{E}(\mathcal{H}_a, \boldsymbol{r})$. 
\begin{pro}
	\label{probdeha}
	Let  $\boldsymbol{r}:=(r_1, r_2)\in [R_1]\times[R_2]$ and  $\mathcal{E}(\mathcal{H}_a, \boldsymbol{r})$ be as in \eqref{defEhar}. 
	Under the conditions of Proposition \ref{kbasics}, 
	for any $\boldsymbol{r}\in [R_1]\times[R_2]$, we have
	\begin{align}
		\label{eqbdeha}
		\mathcal{E}(\mathcal{H}_a)\leq 2\mathcal{E}(\mathcal{H}_a, \boldsymbol{r})+4\lambda_{r_1+1}+4\tau_{r_2+1},
	\end{align}
	with the convention that $\lambda_{R_1+1}=0$ and $\tau_{R_2+1}=0$.  
\end{pro}
\begin{proof}
	Let $\mathcal{H}_b^{\boldsymbol{r}}$ and  $\overline{\mathcal{H}}_b^{\boldsymbol{r}}$ be defined in \eqref{sbuspaces}. For any $h\in \mathcal{H}_b$, there exist  $h_1\in \mathcal{H}_b^{\boldsymbol{r}}$ and  $h_2\in \overline{\mathcal{H}}_b^{\boldsymbol{r}}$ such that $h=h_1+h_2$. Thus, we have
	\begin{align*}
		\begin{split}
			\mathcal{E}(\mathcal{H}_a) =& \max_{\substack{h_1\in \mathcal{H}_b^{\boldsymbol{r}}\\ h_2\in \overline{\mathcal{H}}_b^{\boldsymbol{r}}}}\min_{v\in \mathcal{H}_a}\|v-h_1-h_2\|^2_{\mathcal{U}}\\
			\leq& 2\max_{h_1\in \mathcal{H}_b^{\boldsymbol{r}}}\min_{v\in \mathcal{H}_a}\|v-h_1\|^2_{\mathcal{U}}+2\max_{h_2\in \overline{\mathcal{H}}_b^{\boldsymbol{r}}}\|h_2\|^2_{\mathcal{U}} \leq  2\mathcal{E}(\mathcal{H}_a, \boldsymbol{r}) + 4 \lambda_{r_1+1}+4\tau_{r_2+1}, 
		\end{split}
	\end{align*}
	which concludes \eqref{eqbdeha}. 
\end{proof}
The above proposition implies that it is sufficient to bound $ \mathcal{E}(\mathcal{H}_a, \boldsymbol{r}) $ in order to estimate the Nystr\"{o}m approximation error. 
Let $\boldsymbol{\phi}$ be as in \eqref{defindiline}. For clarification, we denote by $\boldsymbol{\phi}:=(\boldsymbol{\phi}_\Omega, \boldsymbol{\phi}_{\partial \Omega})$, where $\boldsymbol{\phi}_\Omega$ represents the set of linear operators for interior points and $\boldsymbol{\phi}_{\partial \Omega}$ includes boundary linear operators. Let $\widehat{\boldsymbol{x}}:=(\widehat{\boldsymbol{x}}_1, \widehat{\boldsymbol{x}}_2)$ be as in Assumption \ref{hypunisp}, let $M_\Omega$ be the size of $\widehat{\boldsymbol{x}}_1$, and let $M_{\partial \Omega}$ be the length of $\widehat{\boldsymbol{x}}_2$. 
We define the operators $L_{\widehat{\boldsymbol{x}}_1}$ and $L_{\widehat{\boldsymbol{x}}_2}$  such that for any function $v\in \mathcal{U}$, 
\begin{align}
	\label{defLindp}
	L_{\widehat{\boldsymbol{x}}_1}[v]=\frac{1}{M_\Omega}(\mathcal{K}\boldsymbol{\phi}_{\Omega})^T[\boldsymbol{\phi}_{\Omega}, v] \text{ and } L_{\widehat{\boldsymbol{x}}_2}[v]=\frac{1}{M_{\partial\Omega}}(\mathcal{K}\boldsymbol{\phi}_{\partial\Omega})^T[\boldsymbol{\phi}_{\partial\Omega}, v].
\end{align}

The following proposition provides upper bounds for the Hilbert--Schmidt norms of $L_{\overline{\boldsymbol{x}}_1}-L_{\widehat{\boldsymbol{x}}_1}$ and $L_{\overline{\boldsymbol{x}}_2}-L_{\widehat{\boldsymbol{x}}_2}$.  We recall that for any linear operator $L: \mathcal{H}\mapsto\mathcal{H}$, where $\mathcal{H}$ is a Hilbert space, we denote by $\|L\|_{HS}$ and $\|L\|_2$ the Hilbert--Schmidt norm and the spectral norm, respectively, i.e.
\begin{align}
	\label{hsnsn}
	\|L\|_{HS}=\sqrt{\sum_{i,j}\langle \boldsymbol{e}_i, L\boldsymbol{e}_j\rangle^2_{\mathcal{H}}} \text{ and } \|L\|_2 = \max_{\|v\|_{\mathcal{H}}\leq 1}\|Lv\|_{\mathcal{H}}. 
\end{align}
where $\{\boldsymbol{e}_i, i=1, \dots\}$ is a complete orthogonal basis of $\mathcal{H}$. We note that for any linear operator $L: \mathcal{H}\mapsto\mathcal{H}$, $\|L\|_2\leq \|L\|_{HS}$.

\begin{pro}
	\label{coroHSB}
	Suppose that Assumptions \ref{hypunisp} and \ref{hypliopt} hold. Let $L_{\overline{\boldsymbol{x}}_1}, L_{\overline{\boldsymbol{x}}_2}$ be as in \eqref{defLpsi} and let $L_{\widehat{\boldsymbol{x}}_1}, L_{\widehat{\boldsymbol{x}}_2}$ be given in \eqref{defLindp}. Then, 
	with a probability $1-\delta$,  $0<\delta<1$, there exists a constant $C>0$ such that 
	\begin{align}
		\label{bdintLpsiLphi}
		\begin{split}
			\|L_{\overline{\boldsymbol{x}}_1}-L_{\widehat{\boldsymbol{x}}_1}  \|_{HS}\leq \frac{4C\ln(2/\delta)}{\sqrt{M_\Omega}} \text{ and } \|L_{\overline{\boldsymbol{x}}_2}-L_{\widehat{\boldsymbol{x}}_2}  \|_{HS}\leq \frac{4C\ln(2/\delta)}{\sqrt{M_{\partial\Omega}}}. 
		\end{split}
	\end{align}
\end{pro}
\begin{proof}
	Here, we adapt the proof of Corollary 8 in \cite{jin2013improved}. Let $\boldsymbol{\psi}^{\boldsymbol{x}}$ be given in Assumption \ref{hypliopt} for $\boldsymbol{x}\in \overline{\boldsymbol{x}}$.  For any $v\in \mathcal{U}$, we define
	\begin{align*}
		\xi(\boldsymbol{x})[v] = \sum_{i=D_b+1}^D\mathcal{K}\boldsymbol{\psi}^{\boldsymbol{x}}_i[\boldsymbol{\psi}^{\boldsymbol{x}}_i, v], \forall \boldsymbol{x}\in \overline{\boldsymbol{x}}_1,  \text{ and } \widehat{\xi}({\boldsymbol{x}})[v] =  \sum_{i=1}^{D_b}\mathcal{K}\boldsymbol{\psi}^{\boldsymbol{x}}_i[\boldsymbol{\psi}^{\boldsymbol{x}}_i, v], \forall \boldsymbol{x}\in \overline{\boldsymbol{x}}_2. 
	\end{align*} 
	Then, $L_{\widehat{\boldsymbol{x}}_1}=\frac{1}{M_{\Omega}}\sum_{i=1}^{M_{\Omega}}\xi(\widehat{\boldsymbol{x}}_i)$, and $L_{\widehat{\boldsymbol{x}}_2}=\frac{1}{M_{\partial\Omega}}\sum_{i=M_\Omega+1}^M\widehat{\xi}({\widehat{\boldsymbol{x}}}_i)$.  According to  Assumption \ref{hypunisp}, $E[\xi(\boldsymbol{x})]=L_{\overline{\boldsymbol{x}}_1}$ and  $E[\widehat{\xi}(\boldsymbol{x})]=L_{\overline{\boldsymbol{x}}_2}$.  Denote by $\mathcal{H}_1$ the span of $\{\varphi_i\}_{i=1}^{R_1}$ and by $\mathcal{H}_2$ the span of $\{\zeta_i\}_{i=1}^{R_2}$, which are endowed with the norm of $\mathcal{U}$. 
	Under Assumption \ref{hypunisp}, we observe that $L_{\overline{\boldsymbol{x}}_1}$ and $L_{\widehat{\boldsymbol{x}}_1}$ map $\mathcal{H}_1$ to $\mathcal{H}_1$, and that  $L_{\overline{\boldsymbol{x}}_2}$ and $L_{\widehat{\boldsymbol{x}}_2}$ map $\mathcal{H}_2$ to $\mathcal{H}_2$.  Thus, by the definition of the Hilbert--Schmidt norm in \eqref{hsnsn}, there exists a constant $C$ such that 
	\begin{align}
		\label{bdhsxi}
		\begin{split}
			\|\xi(\boldsymbol{x})\|_{HS} =& \sqrt{\sum_{j,k}^{R_1}\bigg\langle \sum_{i=D_b+1}^D\mathcal{K}\boldsymbol{\psi}^{\boldsymbol{x}}_i[\boldsymbol{\psi}^{\boldsymbol{x}}_i, \varphi_j], \varphi_k\bigg\rangle^2}
			= \sqrt{\sum_{j,k}^{R_1}\bigg( \sum_{i=D_b+1}^D[\boldsymbol{\psi}^{\boldsymbol{x}}_i, \varphi_j] [\boldsymbol{\psi}^{\boldsymbol{x}}_i, \varphi_k]\bigg)^2}\\
			\leq & C\sqrt{\sum_{j,k}^{R_1}\sum_{i=D_b+1}^D[\boldsymbol{\psi}^{\boldsymbol{x}}_i, \varphi_j]^2 [\boldsymbol{\psi}^{\boldsymbol{x}}_i, \varphi_k]^2}
			=  C\sqrt{\sum_{i=D_b+1}^D\bigg(\sum_{j}^{R_1}[\boldsymbol{\psi}^{\boldsymbol{x}}_i, \varphi_j]^2\bigg)^2}\\
			\leq& C\sum_{i=D_b+1}^D\sum_{j=1}^{R_1}[\boldsymbol{\psi}^{\boldsymbol{x}}_i, \varphi_j]^2 = C\sum_{i=D_b+1}^D[\psi_i^{\boldsymbol{x}}, \mathcal{K}\psi_i^{\boldsymbol{x}}]\leq C,
		\end{split}
	\end{align}
	where the last equality follows by \eqref{eqphikphi} and we use Assumption \ref{hypliopt} in the last inequality. Similarly, we get
	\begin{align}
		\label{bdhsxihat}
		\begin{split}
			\|\widehat{\xi}(\boldsymbol{x})\|_{HS} \leq  C\sum_{i=1}^{D_b}[\psi_i^{\boldsymbol{x}}, \mathcal{K}\psi_i^{\boldsymbol{x}}]\leq C. 
		\end{split}
	\end{align}
	Hence, using \eqref{bdhsxi}, \eqref{bdhsxihat}, and Proposition 1 of \cite{smale2009geometry} (see also Proposition 6 of \cite{jin2013improved}), we conclude \eqref{bdintLpsiLphi}. 
	
\end{proof}
The following proposition gives an upper bound for $\mathcal{E}(\mathcal{H}_a, \boldsymbol{r})$. 
\begin{pro}
	\label{probder}
	Suppose that Assumptions \ref{hypunisp} and \ref{hypliopt} hold. 
	Let $\mathcal{E}(\mathcal{H}_a, \boldsymbol{r})$ be as in \eqref{defEhar}. Let $M_{\partial \Omega}=M-M_{ \Omega}$ and let $N_{\partial \Omega}=N-N_{ \Omega}$.  Then, for any $\boldsymbol{r}:=(r_1, r_2)\in [R_1]\times[R_2]$, with a probability at least $1-\delta$, $0<\delta<1$, there exists a constant $C$ such that 
	\begin{align}
		\label{bdEr}
		\begin{split}
			\mathcal{E}(\mathcal{H}_a, \boldsymbol{r}) 
			\leq & \max\bigg\{ \frac{C N_{\Omega}^2\ln^2(2/\delta)}{\lambda_{r_1}M_\Omega}, \frac{CN_{\partial\Omega}^2\ln^2(2/\delta)}{\tau_{r_2}M_{\partial\Omega}} \bigg\}. 
		\end{split}
	\end{align}
\end{pro}
\begin{proof}
	The proof is an adaption of the arguments of Theorem 7 in \cite{jin2013improved}. 
	Let $\{\lambda_i\}_{i=1}^{R_1}$, $\{\tau_i\}_{i=1}^{R_2}$, $\{\varphi_i\}_{i=1}^{R_1}$, and  $\{\zeta_i\}_{i=1}^{R_2}$ be the eigenvalues and the eigenfunctions in Proposition \ref{kbasics}. 
	Given $\delta\in [0, 1]$ and $\boldsymbol{r}:=(r_1, r_2)\in [R_1]\times[R_2]$, we define 
	\begin{align*}
		\begin{split}
			\mathcal{H}_{c, \delta}^{\boldsymbol{r}}=&\bigg\{h=\sum_{i=1}^{r_1}c_{1,i}\sqrt{\lambda_i}\varphi_i,\frac{1}{N_\Omega^2} \sum_{i=1}^{r_1}c_{1,i}^2\lambda_i^2\leq \delta\bigg\} \text{ and } \overline{\mathcal{H}}_{c, \delta}^{\boldsymbol{r}}=\bigg\{h=\sum_{i=1}^{r_2}c_{2,i}\sqrt{\tau_i}\zeta_i, \frac{1}{N_{\partial \Omega}^2}\sum_{i=1}^{r_2}c_{2,i}^2\tau_i^2\leq 1-\delta\bigg\},\\
			\mathcal{H}_{d,\delta}^{\boldsymbol{r}}=&\{v\in \mathcal{U}, \|v\|^2_{\mathcal{U}}\leq\delta N_{\Omega}^2/\lambda_{r_1}\}, \text{ and } \overline{\mathcal{H}}_{d,\delta}^{\boldsymbol{r}}=\{v\in \mathcal{U}, \|v\|^2_{\mathcal{U}}\leq (1-\delta)N_{\partial\Omega}^2/\tau_{r_2}\}.
		\end{split}
	\end{align*}
	Thus, $\mathcal{H}_{c,\delta}^{\boldsymbol{r}}\subset{\mathcal{H}}_{d,\delta}^{\boldsymbol{r}}$ and $\overline{{\mathcal{H}}}_{c,\delta}^{\boldsymbol{r}}\subset\overline{\mathcal{H}}_{d,\delta}^{\boldsymbol{r}}$. Meanwhile, for $h_1\in \mathcal{H}_{c,\delta}^{\boldsymbol{r}}$ and $h_2\in \overline{\mathcal{H}}_{c,\delta}^{\boldsymbol{r}}$, using Proposition \ref{kbasics}, we have
	\begin{align}
		\label{lpsih}
		\begin{split}
			L_{\overline{\boldsymbol{x}}_1}[h_1] + L_{\overline{\boldsymbol{x}}_2}[h_2]=& \frac{1}{N_\Omega}\sum_{i=1}^{r_1}c_{1,i}\sqrt{\lambda_i}(\mathcal{K}\boldsymbol{\psi}_{\Omega})^T[\boldsymbol{\psi}_{\Omega}, \varphi_i]+\frac{1}{N_{\partial\Omega}}\sum_{i=1}^{r_2}c_{2,i}\sqrt{\tau_i}(\mathcal{K}\boldsymbol{\psi}_{\partial\Omega})^T[\boldsymbol{\psi}_{\partial\Omega}, \zeta_i]\\
			=& \frac{1}{N_\Omega}\sum_{i=1}^{r_1}c_{1,i}\lambda_i(\mathcal{K}\boldsymbol{\psi}_\Omega)^T\boldsymbol{v}_i+\frac{1}{N_{\partial\Omega}}\sum_{i=1}^{r_2}c_{2,i}\tau_i(\mathcal{K}\boldsymbol{\psi}_{\partial \Omega})^T\boldsymbol{w}_i\\ 
			=& \frac{1}{N_\Omega}\sum_{i=1}^{r_1}c_{1,i}\lambda_i\sqrt{\lambda_i}\varphi_i+\frac{1}{N_{\partial\Omega}}\sum_{i=1}^{r_2}c_{2,i}\tau_i\sqrt{\tau_i}\zeta_i.
		\end{split}
	\end{align}
	Thus, \eqref{lpsih} implies that for any $h\in \mathcal{H}_b^r$, there exist $\delta\in [0, 1]$, $h_1\in \mathcal{H}_{c,\delta}^{\boldsymbol{r}}$, and $h_2\in \overline{\mathcal{H}}_{c,\delta}^{\boldsymbol{r}}$ such that $h=L_{\overline{\boldsymbol{x}}_1}[h_1] + L_{\overline{\boldsymbol{x}}_2}[h_2]$. Hence, we have
	\begin{align*}
		\begin{split}
			\mathcal{E}(\mathcal{H}_a, r)=& \max_{h\in \mathcal{H}_b^r}\mathcal{E}(g, \mathcal{H}_a) = \max_{\substack{h_1\in \mathcal{H}_{c,\delta}^{\boldsymbol{r}}\\ h_2\in \overline{\mathcal{H}}_{c,\delta}^{\boldsymbol{r}}}}\min_{v\in \mathcal{H}_a}\|L_{\overline{\boldsymbol{x}}_1}[h_1] + L_{\overline{\boldsymbol{x}}_2}[h_2]-v\|^2_{\mathcal{U}}\\
			\leq& \max_{\substack{ \substack{h_1\in \mathcal{H}_{d,\delta}^{\boldsymbol{r}}\\ h_2\in \overline{\mathcal{H}}_{d,\delta}^{\boldsymbol{r}}} }}\min_{v\in \mathcal{H}_a}\|L_{\overline{\boldsymbol{x}}_1}[h_1] + L_{\overline{\boldsymbol{x}}_2}[h_2]-v\|^2_{\mathcal{U}}. 
		\end{split}
	\end{align*}
	By the definitions of $L_{\widehat{\boldsymbol{x}}_1}$ and $L_{\widehat{\boldsymbol{x}}_2}$, we have $L_{\widehat{\boldsymbol{x}}_1}[h_1]+ L_{\widehat{\boldsymbol{x}}_2}[h_2] \in \mathcal{H}_a$. Thus, we get
	\begin{align}
		\label{bdhar1}
		\begin{split}
			\mathcal{E}(\mathcal{H}_a, r) \leq& \max_{\substack{ \substack{h_1\in \mathcal{H}_{d,\delta}^{\boldsymbol{r}}\\ h_2\in \overline{\mathcal{H}}_{d,\delta}^{\boldsymbol{r}}} }}\min_{v\in \mathcal{H}_a}\|L_{\overline{\boldsymbol{x}}_1}[h_1] + L_{\overline{\boldsymbol{x}}_2}[h_2]-v\|^2_{\mathcal{U}}\\
			\leq& \max_{\substack{ \substack{h_1\in \mathcal{H}_{d,\delta}^{\boldsymbol{r}}\\ h_2\in \overline{\mathcal{H}}_{d,\delta}^{\boldsymbol{r}}} }}\|L_{\overline{\boldsymbol{x}}_1}[h_1] + L_{\overline{\boldsymbol{x}}_2}[h_2]-L_{\widehat{\boldsymbol{x}}_1}[h_1] - L_{\widehat{\boldsymbol{x}}_2}[h_2]\|^2_{\mathcal{U}}\\
			\leq & \max_{\delta\in [0,1]}\bigg\{\frac{2\delta N_\Omega^2}{\lambda_{r_1}}\|L_{\overline{\boldsymbol{x}}_1}-L_{\widehat{\boldsymbol{x}}_1}\|_2^2 + \frac{2(1-\delta)N_{\partial\Omega}^2}{\tau_{r_2}}\|L_{\overline{\boldsymbol{x}}_2}-L_{\widehat{\boldsymbol{x}}_2}\|_2^2\bigg\}\\
			\leq & \max\bigg\{ \frac{2 N_{\Omega}^2}{\lambda_{r_1}}\|L_{\overline{\boldsymbol{x}}_1}-L_{\widehat{\boldsymbol{x}}_1}\|_{HS}^2, \frac{2N_{\partial\Omega}^2}{\tau_{r_2}}\|L_{\overline{\boldsymbol{x}}_2}-L_{\widehat{\boldsymbol{x}}_2}\|_{HS}^2 \bigg\}.
		\end{split}
	\end{align}
	Thus, \eqref{bdhar1} together with Proposition \ref{coroHSB} implies that \eqref{bdEr} holds. 
\end{proof}
Finally, we are ready to give a bound for $\mathcal{E}(\mathcal{H}_a)$. 
\begin{proof}[{\color{black}Proof of Theorem \ref{corobdls}}]
	The inequality \eqref{bdE} follows directly from Propositions \ref{probdeha} and \ref{probder}. 
\end{proof}

\end{document}